\documentclass[11pt,reqno]{amsart}
\usepackage{mathrsfs}
\usepackage{amsmath}
\usepackage{amssymb, amsmath}
\usepackage{color}
\usepackage{enumerate}
\usepackage{bm}

\numberwithin{equation}{section}

\usepackage[left=.75 in, right=.75 in,top=.75 in, bottom=.75 in]{geometry}

\allowdisplaybreaks

\let\lam=\lambda

\let\om=\omega

\let\Lam=\Lambda

\let\Om=\Omega

\let\th=T
\let\pa=\partial

\def\dive{\mathop{\rm div}\nolimits}

\newcommand{\beq}{\begin{equation}}
\newcommand{\eeq}{\end{equation}}
\newcommand{\ben}{\begin{eqnarray}}
\newcommand{\een}{\end{eqnarray}}
\newcommand{\beno}{\begin{eqnarray*}}
\newcommand{\eeno}{\end{eqnarray*}}

\newtheorem{theorem}{Theorem}[section]
\newtheorem{definition}[theorem]{Definition}
\newtheorem{lemma}[theorem]{Lemma}
\newtheorem{proposition}[theorem]{Proposition}

\newtheorem{remark}[theorem]{Remark}

\begin{document}

\title[contact line]{Local well-posedness of the contact line problem in 2-D Stokes flow}

\author{Yunrui Zheng}
\address{Beijing International Center for Mathematical Research, Peking University, Beijing, 100871, P. R. China}
\email{Y. Zheng: ruixue@mail.ustc.edu.cn}
\author{Ian Tice}
\address{Department of Mathematical Sciences, Carnegie Mellon University, Pittsburgh, PA 15213, USA}
\email{I. Tice: iantice@andrew.cmu.edu}
\footnote{I. Tice was supported by a grant from the Simons Foundation (401468).}

\begin{abstract}
  We consider the evolution of contact lines for viscous fluids in a two-dimensional open-top vessel. The domain is bounded above by a free moving boundary and otherwise by the solid wall of a vessel. The dynamics of the fluid are governed by the incompressible Stokes equations under the influence of gravity, and the interface between fluid and air is under the effect of capillary forces.  Here we develop a local well-posedness theory of the problem in the framework of nonlinear energy methods.  We utilize several techniques, including: energy estimates of a geometric formulation of the Stokes equations, a Galerkin method with a time-dependent basis for an $\epsilon$--perturbed linear Stokes problem in moving domains, the contraction mapping principle for the $\epsilon$--perturbed nonlinear full contact line problem, and a continuity argument for uniform energy estimates.
\end{abstract}

%

\maketitle

\section{Introduction}

\subsection{Formulation of the problem in Eulerian coordinates}

We consider a $2$--D open top vessel as a bounded, connected open set $\mathcal{V}\subseteq\mathbb{R}^2$ which consists of two ``almost'' disjoint sections, i.e., $\mathcal{V}=\mathcal{V}_{top}\cup\mathcal{V}_{bot}$. The word ``almost'' means $\mathcal{V}_{top}\cap\mathcal{V}_{bot}$ is a set of measure $0$ in $\mathbb{R}^2$. We assume that the ``top'' part $\mathcal{V}_{top}$ consists of a rectangular channel defined by
\[
\mathcal{V}_{top}=\mathcal{V}\cap\mathbb{R}^2_{+}=\{y\in\mathbb{R}^2: -\ell<y_1<\ell, 0\le y_2<L\}
\]
for some $\ell$, $L>0$, where $\mathbb{R}^2_{+}$ is the half plane $\mathbb{R}^2_{+}=\{y\in\mathbb{R}^2: y_2\ge0\}$. Similarly, we write the ``bottom'' part as
\[
\mathcal{V}_{bot}=\mathcal{V}\cap\mathbb{R}^2_{-}=\mathcal{V}\cap\{y\in\mathbb{R}^2: y_2\le0\}.
\]
In addition, we also assume that the boundary $\partial\mathcal{V}$ of $\mathcal{V}$ is $C^2$ away from the points $(\pm\ell, L)$.

Now we consider a viscous incompressible fluid filling the $\mathcal{V}_{bot}$ entirely and $\mathcal{V}_{top}$ partially. More precisely, we assume that the fluid occupies the domain $\Om(t)$ with an upper free surface,
\[
\Om(t)=\mathcal{V}_{bot}\cup\{y\in\mathbb{R}^2: -\ell<y_1<\ell,\ 0<y_2<\zeta(y_1,t)\},
\]
where the free surface $\zeta(y_1,t)$ is assumed to be a graph of the function $\zeta: [-\ell, \ell]\times \mathbb{R}_{+}\rightarrow\mathbb{R}$ satisfying $0<\zeta(\pm\ell,t)\le L$ for all $t\in\mathbb{R}_+$, which means the fluid does not spill out of the top domain. For simplicity, we write the free surface as $\Sigma(t)=\{y_2=\zeta(y_1,t)\}$ and the interface between fluid and solid as $\Sigma_s(t)=\partial\Om(t)\setminus\Sigma(t)$.

For each $t\ge0$, the fluid is described by its velocity and pressure $(u,P):\Om(t)\rightarrow\mathbb{R}^2\times\mathbb{R}$, the dynamics of which are governed by the incompressible Stokes equations for $t>0$ :
\begin{equation}\label{eq:stokes}
  \left\{
  \begin{aligned}
    &\dive S(P,u)=\nabla P-\mu\Delta u=0 \quad & \text{in} &\ \Om(t),\\
    &\dive u=0 \quad & \text{in} &\ \Om(t),\\
    &S(P,u)\nu=g\zeta\nu-\sigma \mathcal{H}(\zeta)\nu \quad & \text{on} &\ \Sigma(t),\\
    &(S(P,u)\nu-\beta u)\cdot\tau=0 \quad & \text{on} &\ \Sigma_s(t),\\
    &u\cdot\nu=0 \quad & \text{on} &\ \Sigma_s(t),\\
    &\pa_t\zeta=u\cdot\nu=u_2-u_1\pa_1\zeta \quad & \text{on} &\ \Sigma(t),\\
    &\pa_t\zeta(\pm\ell,t)=\mathscr{V}\left([\![\gamma]\!]\mp\sigma\frac{\pa_1\zeta}{(1+|\pa_1\zeta|^2)^{1/2}}(\pm\ell,t)\right).
  \end{aligned}
  \right.
\end{equation}
with the initial data $\zeta(y_1,t=0)=\zeta(0)$, $\pa_t\zeta(y_1,t=0)=\pa_t\zeta(0)$ and $\pa_t^2\zeta(y_1,t=0)=\pa_t^2\zeta(0)$.

In the above system \eqref{eq:stokes}, $S(p,u)$ is the viscous stress tensor determined by
\[
S(P,u)=PI-\mu\mathbb{D}u,
\]
where $I$ is the $2\times2$ identity matrix, $\mu>0$ is the coefficient of viscosity, $\mathbb{D}u=\nabla u+\nabla^\top u$ is the symmetric gradient of $u$ for $\nabla^\top u$ the transpose of the matrix $\nabla u$, $P$ is the difference between the full pressure and the hydrostatic pressure. $\nu$ is the outward unit normal. $\tau$ is the unit tangent. $\sigma>0$ is the coefficient of surface tension, and
\[
\mathcal{H}(\zeta)=\pa_1\left(\frac{\pa_1\zeta}{(1+|\pa_1\zeta|^2)^{1/2}}\right)
\]
is the twice of mean curvature of the free surface. $\beta>0$ is the Navier slip friction coefficient on the vessel walls. The function  $\mathscr{V}:\mathbb{R}\rightarrow\mathbb{R}$ is the contact point velocity response function which is a $C^2$ increasing diffeomorphism satisfying $\mathscr{V}(0)=0$. $[\![\gamma]\!]:=\gamma_{sv}-\gamma_{sf}$ for $\gamma_{sv}$, $\gamma_{sf}\in\mathbb{R}$, where $\gamma_{sv}$, $\gamma_{sf}$ are a measure of the free-energy per unit length with respect to the solid-vapor and solid-fluid intersection. In addition, we assume that the Young relation \cite{Yo} holds
\beq
\frac{|[\![\gamma]\!]|}{\sigma}<1,
\eeq
which is necessary for the existence of equilibrium state. For convenience, we introduce the inverse function $\mathscr{W} = \mathscr{V}^{-1}$ and rewrite the final equation in \eqref{eq:stokes} as
\beq
\mathscr{W}(\pa_t\zeta(\pm\ell,t))=[\![\gamma]\!]\mp\sigma\frac{\pa_1\zeta}{(1+|\pa_1\zeta|^2)^{1/2}}(\pm\ell,t).
\eeq

\subsection{A steady equilibrium state}

A steady-state equilibrium solution to \eqref{eq:stokes} corresponds to  $u=0$, $P(y,t)=P_0(y)$, and $\zeta(y_1,t)=\zeta_0(y)$.  These satisfy
\begin{equation}\label{eq:equibrium}
  \left\{
  \begin{aligned}
    &\nabla P_0=0 \quad & \text{in} &\ \Om(0),\\
    &P_0=g\zeta_0-\sigma\mathcal{H}(\zeta_0),\quad & \text{on} &\ (-\ell,\ell),\\
    &\sigma\frac{\pa_1\zeta_0}{\sqrt{1+|\pa_1\zeta_0|^2}}(\pm\ell)=\pm[\![\gamma]\!].
  \end{aligned}
  \right.
\end{equation}
It is well-known (see for instance the discussion in the introduction of \cite{GT2}) that there exists a smooth solution $\zeta_0 : [-\ell,\ell] \to (0,L)$.

\subsection{Geometric reformulation}
Let $\zeta_0\in C^\infty[-\ell,\ell]$ be the equilibrium surface given by \eqref{eq:equibrium}. We then define the equilibrium domain $\Om\subset\mathbb{R}^2$ by
\[
\Om:=\mathcal{V}_{b}\cup\{x\in\mathbb{R}^2|-\ell<x_1<\ell, 0<x_2<\zeta_0(x_1)\}.
\]
The boundary $\pa\Om$ of the equilibrium $\Om$ is defined by
\[
\pa\Om:=\Sigma\sqcup\Sigma_s,
\]
where
\[
\Sigma:=\{x\in\mathbb{R}^2|-\ell<x_1<\ell, x_2=\zeta_0(x_1)\}, \quad \Sigma_s=\pa\Om\setminus \Sigma.
\]
Here $\Sigma$ is the equilibrium free surface. The corner angle $\om\in(0,\pi)$ of $\Om$ is the contact angle formed by the  fluid and solid.
 We will view the function $\zeta(y_1,t)$ of the free surface as the perturbation of $\zeta_0(y_1)$:
\beq\label{eq:perturbation surface}
\zeta(y_1,t)=\zeta_0(y_1)+\eta(y_1,t).
\eeq

Let $\phi\in C^\infty(\mathbb{R})$ be such that $\phi(z)=0$ for $z\le\frac14\min\zeta_0$ and $\phi(z)=z$ for $z\ge\frac12\min\zeta_0$.
Now we define the mapping $\Phi: \Om\mapsto\Om(t)$, by
\beq\label{def:map}
\Phi(x_1,x_2,t)=\left(x_1, x_2+\frac{\phi(x_2)}{\zeta_0(x_1)}\bar{\eta}(x_1,x_2,t)\right)=(y_1,y_2)\in\Om(t),
\eeq
with $\bar{\eta}$ is defined by
\beq
\bar{\eta}(x_1,x_2,t):=\mathcal{P}E\eta(x_1,x_2-\zeta_0(x_1),t),
\eeq
where $E: H^s(-\ell,\ell)\mapsto H^s(\mathbb{R})$ is a bounded extension operator for all $0\le s\le 3$ and $\mathcal{P}$ is the lower Poisson extension given by
\[
\mathcal{P}f(x_1,x_2)=\int_{\mathbb{R}}\hat{f}(\xi)e^{2\pi|\xi| x_2}e^{2\pi ix_1\xi}\,\mathrm{d}\xi.
\]
If $\eta$ is sufficiently small (in appropriate Sobolev spaces), the mapping $\Phi$ is a $C^1$ diffeomorphism of $\Om$ onto $\Om(t)$ that maps the components of $\pa\Om$ to the corresponding components of $\pa\Om(t)$.

We have the Jacobian matrix $\nabla\Phi$ and the transform matrix $\mathcal{A}$ of $\Phi$
\beq\label{eq:Jaccobian}
\nabla\Phi=\left(
\begin{array}{cc}
  1&0\\
  A&J
\end{array}
\right), \quad
\mathcal{A}=(\nabla\Phi)^{-\top}=\left(
\begin{array}{cc}
  1&-AK\\
  0&K
\end{array}
\right),
\eeq
for
\beq\label{eq:components}
A=\frac{\phi}{\zeta_0}\pa_1\bar{\eta}-\frac{\phi}{\zeta_0^2}\pa_1\zeta_0\bar{\eta},\quad J=1+\frac{\phi^\prime}{\zeta_0}\bar{\eta}+\frac{\phi}{\zeta_0}\pa_2\bar{\eta},\quad K=\frac{1}{J}.
\eeq

We define the transformed differential operators as follows.
\[
(\nabla_{\mathcal{A}}f)_i:=\mathcal{A}_{ij}\pa_jf,\quad\dive_{\mathcal{A}}X:=\mathcal{A}_{ij}\pa_j X_i,\quad \Delta_{\mathcal{A}}f:=\dive_{\mathcal{A}}\nabla_{\mathcal{A}}f,
\]
for appropriate $f$ and $X$. We write the stress tensor
\[
S_{\mathcal{A}}(P,u)=PI-\mu\mathbb{D}_{\mathcal{A}}u
\]
where $I$ the $2\times2$ identity matrix and $(\mathbb{D}_{\mathcal{A}}u)_{ij}=\mathcal{A}_{ik}\pa_ku_j+\mathcal{A}_{jk}\pa_ku_i$ the symmetric $\mathcal{A}$--gradient. Note that if we extend $\dive_{\mathcal{A}}$ to act on symmetric tensors in the natural way, then $\dive_{\mathcal{A}}S_{\mathcal{A}}(P,u)=-\mu\Delta_{\mathcal{A}}u+\nabla_{\mathcal{A}}P$ for vectors fields satisfying $\dive_{\mathcal{A}}u=0$.

We assume that $\Phi$ is a diffeomorphism. Then we can transform the problem \eqref{eq:stokes} to the equilibrium domain $\Om$ for $t\ge0$. In the new coordinates, \eqref{eq:stokes} becomes the $\mathcal{A}$--Stokes problem
\begin{equation}\label{eq:geometric}
  \left\{
  \begin{aligned}
    &\dive_{\mathcal{A}}S_{\mathcal{A}}(P,u)=-\mu\Delta_{\mathcal{A}}u+\nabla_{\mathcal{A}}P=0,\quad &\text{in}&\quad \Om,\\
    &\dive_{\mathcal{A}}u=0,\quad &\text{in}&\quad \Om,\\
    &S_{\mathcal{A}}(P,u)\mathcal{N}=g\zeta\mathcal{N}-\sigma\mathcal{H}(\zeta)\mathcal{N},\quad &\text{on}&\quad\Sigma,\\
    &(S_{\mathcal{A}}(P,u)\nu-\beta u)\cdot\tau=0,\quad &\text{on}&\quad\Sigma_s,\\
    &u\cdot\nu=0,\quad &\text{on}&\quad\Sigma_s,\\
    &\pa_t\zeta=u\cdot\mathcal{N},\quad &\text{on}&\quad\Sigma,\\
    &\mathscr{W}(\pa_t\zeta(\pm\ell,t))=[\![\gamma]\!]\mp\sigma\frac{\pa_1\zeta}{\sqrt{1+|\zeta|^2}}(\pm\ell,t),\\
    &\zeta(x_1,0)=\zeta_0(x_1)+\eta_0(x_1),\quad\pa_t\zeta(x_1,0)=\pa_t\eta(x_1,0),\quad\pa_t^2\zeta(x_1,0)=\pa_t^2\eta(x_1,0).
  \end{aligned}
  \right.
\end{equation}
Here we have still written $\mathcal{N}:=-\pa_1\zeta e_1+e_2$ for the normal to $\Sigma(t)$.

Since all terms in \eqref{eq:geometric} are in terms of $\eta$, \eqref{eq:geometric} is connected to the geometry of the free surface. This geometric structure is essential to control higher-order derivatives.

\subsection {Perturbation}
We will construct the solution to \eqref{eq:geometric} as a perturbation around the equilibrium state $(0, P_0, \zeta_0)$.  To this end we define new perturbed unknowns $(u, p, \eta)$ so that $u=0+u$, $P=P_0+p$, and $\zeta=\zeta_0+\eta$. Then we will reformulate \eqref{eq:geometric} in terms of the new unknowns.

First, we rewrite the terms of mean curvature on the equilibrium free surface. By a Taylor expansion in $z$,
\beq
\frac{y+z}{\sqrt{1+|y+z|^2}}=\frac{y}{\sqrt{1+|y|^2}}+\frac{z}{({1+|y|^2)^{3/2}}}+\mathcal{R}(y,z).
\eeq
Combining with the assumption \eqref{eq:perturbation surface}, we then know that
\beq
\frac{\pa_1\zeta}{\sqrt{1+|\pa_1\zeta|^2}}=\frac{\pa_1\zeta_0}{\sqrt{1+|\pa_1\zeta_0|^2}}+\frac{\pa_1\eta}{(1+|\pa_1\zeta_0|^2)^{3/2}}+\mathcal{R}(\pa_1\zeta_0,\pa_1\eta).
\eeq
where the remainder term $\mathcal{R}\in C^\infty(\mathbb{R}^2)$ is given by
\beq\label{def:R}
\mathcal{R}(y,z)=\int_0^z3\frac{(s-z)(s+y)}{(1+(y+s)^2)^{5/2}}\,\mathrm{d}s.
\eeq
Thus
\begin{equation}\label{eq:Taylor_1}
\begin{aligned}
g\zeta-\sigma\mathcal{H}(\zeta)=(g\zeta_0-\sigma\mathcal{H}(\zeta_0))+g\eta-\sigma\pa_1\left(\frac{\pa_1\eta}{(1+|\pa_1\zeta_0|)^{3/2}}\right)-\sigma\pa_1(\mathcal{R}(\pa_1\zeta_0,\pa_1\eta))\\
=p_0+g\eta-\sigma\pa_1\left(\frac{\pa_1\eta}{(1+|\pa_1\zeta_0|)^{3/2}}\right)-\sigma\pa_1(\mathcal{R}(\pa_1\zeta_0,\pa_1\eta)),
\end{aligned}
\end{equation}
and
\begin{equation}\label{eq:Taylor_2}
  \begin{aligned}
    &[\![\gamma]\!] \mp\sigma\frac{\pa_1\zeta}{\sqrt{1+|\pa_1\zeta|^2}}(\pm\ell,t)=[\![\gamma]\!]\mp\sigma\frac{\pa_1\zeta_0}{\sqrt{1+|\pa_1\zeta_0|^2}}(\pm\ell,t)\mp\sigma\frac{\pa_1\eta}{(1+|\zeta_0|^2)^{3/2}}\\
    &\quad\mp\mathcal{R}(\pa_1\zeta_0,\pa_1\eta)(\pm\ell,t)=\mp\sigma\frac{\pa_1\eta}{(1+|\pa_1\zeta_0|^2)^{3/2}}(\pm\ell,t)\mp\mathcal{R}(\pa_1\zeta_0,\pa_1\eta)(\pm\ell,t).
  \end{aligned}
\end{equation}
Next we rewrite the terms related to the stress tensor in \eqref{eq:geometric}. Clearly,
\beq\label{eq:stress_new}
\begin{aligned}
\dive_{\mathcal{A}}S_{\mathcal{A}}(P,u)=\dive_{\mathcal{A}}S_{\mathcal{A}}(p,u),\ \text{in}\ \Om,\\
S_{\mathcal{A}}(P,u)\mathcal{N}=S_{\mathcal{A}}(p,u)\mathcal{N}+P_0\mathcal{N},\ \text{on}\ \Sigma,\\
S_{\mathcal{A}}(P,u)\nu\cdot\tau=S_{\mathcal{A}}(p,u)\nu\cdot\tau,\ \text{on}\ \Sigma_s.
\end{aligned}
\eeq
Finally, we rewrite the inverse $\mathscr{W}\in C^2(\mathbb{R})$ of the contact point response function. Since $\mathscr{W}(0)=0$, we expand $\mathscr{W}$ as
\beq
\mathscr{W}(z)=\mathscr{W}'(0)z+\tilde{\mathscr{W}}(z).
\eeq
Then we write $\kappa=\mathscr{W}'(0)>0$, since $\mathscr{W}$ is increasing. For convenience, we write
\beq
\hat{\mathscr{W}}(z)=\frac1\kappa\tilde{\mathscr{W}}(z)=\frac1\kappa\mathscr{W}(z)-z.
\eeq

Thus, combining \eqref{eq:equibrium}, \eqref{eq:Taylor_1}--\eqref{eq:stress_new}, we arrive at the following perturbative form of Stokes equations
\begin{equation}\label{eq:geometric perturbation}
  \left\{
  \begin{aligned}
    &\dive_{\mathcal{A}}S_{\mathcal{A}}(p,u)=-\mu\Delta_{\mathcal{A}}u+\nabla_{\mathcal{A}}p=0,\quad &\text{in}&\quad \Om,\\
    &\dive_{\mathcal{A}}u=0,\quad &\text{in}&\quad \Om,\\
    &S_{\mathcal{A}}(p,u)\mathcal{N}=g\eta\mathcal{N}-\sigma\pa_1\left(\frac{\pa_1\eta}{(1+|\pa_1\zeta_0|)^{3/2}}\right)\mathcal{N}-\sigma\pa_1(\mathcal{R}(\pa_1\zeta_0,\pa_1\eta))\mathcal{N},\quad &\text{on}&\quad\Sigma,\\
    &(S_{\mathcal{A}}(p,u)\nu-\beta u)\cdot\tau=0,\quad &\text{on}&\quad\Sigma_s,\\
    &u\cdot\nu=0,\quad &\text{on}&\quad\Sigma_s,\\
    &\pa_t\eta=u\cdot\mathcal{N},\quad &\text{on}&\quad\Sigma,\\
    &\kappa\pa_t\eta(\pm\ell,t)+\kappa\hat{\mathscr{W}}(\pa_t\eta(\pm\ell,t))=\mp\sigma\left(\frac{\pa_1\eta}{(1+|\zeta_0|^2)^{3/2}}+\mathcal{R}(\pa_1\zeta_0,\pa_1\eta)\right)(\pm\ell,t).
  \end{aligned}
  \right.
\end{equation}
with the initial data $\eta(x_1,0)=\eta_0(x_1)$, $\pa_t\eta(x_1,0)$ and $\pa_t^2\eta(x_1,0)$.
Here $\mathcal{A}$ and $\mathcal{N}$ are still determined in terms of $\zeta=\zeta_0+\eta$. In the following, we write $\mathcal{N}_0$ be the non-unit normal for the equilibrium surface $\Sigma$, and $\mathcal{N}=\mathcal{N}_0-\pa_1\eta e_1$.

\subsection{Main theorem}
In order to state our result, we need to explain our notation for Sobolev spaces and norms. We take $H^k(\Om)$ and $H^k(\Sigma)$ for $k\ge0$ to be the usual Sobolev spaces, and take $W_\delta^k(\Om)$ and $W_\delta^k(\Sigma)$ for $k\ge0$ and $\delta \in (0,1)$ to be the weighted Sobolev spaces defined in \eqref{def:weighted_sobolev}. We write norms $\|\pa_t^ju\|_k$ and $\|\pa_t^jp\|_k$ in the space $H^k(\Om)$, and $\|\pa_t^j\eta\|_k$ in space $H^k(\Sigma)$.

Now, we define the energy and dissipation used in this paper. The energy is
\beq
\mathcal{E}(t)=\|u\|_{W_\delta^2}^2+\|\pa_tu\|_1^2+\|p\|_{\mathring{W}_\delta^1}^2+\|\pa_tp\|_{\mathring{H}^0}^2+\|\eta\|_{W_\delta^{5/2}}^2+\|\pa_t\eta\|_{3/2}^2+\sum_{j=0}^2\|\pa_t^j\eta\|_{\mathring{H}^1}^2,
\eeq
and the dissipation is
\beq
\begin{aligned}
\mathcal{D}(t)&=\sum_{j=0}^1\left(\|\pa_t^ju\|_{W_\delta^2}^2+\|\pa_t^jp\|_{\mathring{W}_\delta^1}^2+\|\pa_t^j\eta\|_{W_\delta^{5/2}}^2\right)+\sum_{j=0}^2\left(\|\pa_t^ju\|_1^2+\|\pa_t^ju\|_{H^0(\Sigma_s)}^2+[\pa_t^ju\cdot\mathcal{N}]_\ell^2\right)\\
&\quad+\sum_{j=0}^2\left(\|\pa_t^jp\|_0^2+\|\pa_t^j\eta\|_{3/2}^2\right)+\|\pa_t^3\eta\|_{W_\delta^{1/2}}^2,
\end{aligned}
\eeq
where $[f]_\ell^2$ is defined by \eqref{sum_point} and Remark \eqref{rem:norm}, $\mathring{H}^s((-\ell,\ell))$ is defined in \eqref{space_H_0} and $\mathring{W}_\delta^k(\Om)$ is defined in \eqref{space_W_0}.

With the notation established we may now state our main result.

\begin{theorem}\label{thm:main}
  Assume the initial data satisfy the inclusions $\eta_0\in\mathring{W}_\delta^{5/2}((-\ell,\ell))$, $\pa_t\eta(0)\in\mathring{H}^{3/2}((-\ell,\ell))$, and $\pa_t^2\eta(0)\in\mathring{H}^1((-\ell,\ell))$ and that they satisfy the compatibility condition described in Section 3. Then there exists $0<\alpha_0$, $T_0<1$, such that if
  \[
  \mathfrak{E}_0:=\|\eta_0\|_{W_\delta^{5/2}}^2+\|\pa_t\eta(0)\|_{3/2}^2+\sum_{j=0}^2\|\pa_t^j\eta(0)\|_1^2\le\alpha_0
  \]
   and $0<T<T_0$, then there exists a unique solution $(u,p,\eta)$ to \eqref{eq:geometric perturbation} on the interval $[0,T]$ that achieves the initial data and satisfies
  \beq
  \begin{aligned}
    \sup_{0\le t\le T}\mathcal{E}(t)+\int_0^T\mathcal{D}(t)\,\mathrm{d}t\le C\mathfrak{E}_0
  \end{aligned}
  \eeq
  for a universal constant $C>0$. Moreover, $\Phi$ defined by \eqref{def:map} is a $C^1$ diffeomorphism for each $t\in[0,T]$.
\end{theorem}
\begin{remark}
  Since $\Phi$ is a $C^1$ diffeomorphism, we can change coordinates from $\Om$ to $\Om(t)$ to gain solutions of \eqref{eq:stokes}.
\end{remark}

The techniques used for the proof of Theorem \eqref{thm:main} are developed throughout the rest of this paper. We will sketch the main ideas of the proof here.

\textbf{$\epsilon$-perturbed linear $\mathcal{A}$-Stokes}.  Our method is ultimately based on the following geometric formulation of linear Stokes equations and a fixed point argument.  We suppose that $\eta$ (and hence $\mathcal{A}$, $\mathcal{N}$, etc.) is given and then solve the linear $\mathcal{A}$--Stokes equations for $(u,p,\xi)$:
\beq\label{eq:linear_stokes}
\left\{
\begin{aligned}
  &\dive_{\mathcal{A}}S_{\mathcal{A}}(p,u)=F^1,\quad&\text{in}&\quad\Om,\\
  &\dive_{\mathcal{A}}u=0,\quad&\text{in}&\quad\Om,\\
  &S_{\mathcal{A}}(p,u)\mathcal{N}=\left(g\xi-\sigma\pa_1\left(\frac{\pa_1\xi}{(1+|\pa_1\zeta_0|^2)^{3/2}}\right)-\sigma\pa_1 F^3\right)\mathcal{N}+F^4,\quad&\text{on}&\quad\Sigma,\\
  &(S_{\mathcal{A}}(p,u)\nu-\beta u)\cdot\tau=F^5,\quad&\text{on}&\quad\Sigma_s,\\
  &u\cdot\nu=0,\quad&\text{on}&\quad\Sigma_s,\\
  &\pa_t\xi=u\cdot\mathcal{N},\quad&\text{on}&\quad (-\ell,\ell),\\
  &\mp\sigma\frac{\pa_1\xi}{(1+|\pa_1\zeta_0|^2)^{3/2}}(\pm\ell)=\kappa(u\cdot\mathcal{N})(\pm\ell)\pm\sigma F^3(\pm\ell)-\kappa\hat{\mathscr{W}}(\pa_t\eta(\pm\ell)),
\end{aligned}
\right.
\eeq
where $F^3=\mathcal{R}(\pa_1\zeta_0,\pa_1\eta)$.  The local existence theory we aim to develop is designed to produce solutions in the functional framework needed for the global analysis in \cite{GT2}.  Thus it is essential in the present paper that we develop solutions with some degree of regularity.  Unfortunately, in attempting to work directly with \eqref{eq:linear_stokes} in a higher-regularity fixed point argument we encounter serious difficulties with estimating a couple key terms.  For instance we need to estimate interaction terms of the form
\beq\label{eq:difficulty}
\int_{-\ell}^\ell\pa_z\mathcal{R}\pa_1\pa_t^2\eta\pa_1\pa_t^3\xi,
\eeq
but the regularity theory for \eqref{eq:linear_stokes} does not quite meet the demands of this term ($\pa_t^3\xi$ is only in $W_\delta^{1/2}$ due to the a priori estimate for $\pa_t^3\eta$).

Fortunately, it's possible to bypass this difficulty by making a small perturbation of the equations, which provides a crucial extra estimate.  We consider the following $\epsilon$--perturbed linear $\mathcal{A}$-Stokes instead of \eqref{eq:linear_stokes}:
\beq\label{eq:linear_perturbed_stokes}
\left\{
\begin{aligned}
  &\dive_{\mathcal{A}}S_{\mathcal{A}}(p,u)=F^1,\quad&\text{in}&\quad\Om,\\
  &\dive_{\mathcal{A}}u=0,\quad&\text{in}&\quad\Om,\\
  &S_{\mathcal{A}}(p,u)\mathcal{N}=\left(g(\xi+\epsilon\partial_t\xi)-\sigma\pa_1\left(\frac{\pa_1\xi+\epsilon\pa_1\partial_t\xi}{(1+|\pa_1\zeta_0|^2)^{3/2}}\right)-\sigma\pa_1 F^3\right)\mathcal{N}+F^4,\quad&\text{on}&\quad\Sigma,\\
  &(S_{\mathcal{A}}(p,u)\nu-\beta u)\cdot\tau=F^5,\quad&\text{on}&\quad\Sigma_s,\\
  &u\cdot\nu=0,\quad&\text{on}&\quad\Sigma_s,\\
  &\pa_t\xi=u\cdot\mathcal{N},\quad&\text{on}&\quad (-\ell,\ell),\\
  &\mp\sigma\frac{\pa_1\xi+\epsilon\pa_1\partial_t\xi}{(1+|\pa_1\zeta_0|^2)^{3/2}}(\pm\ell)=\kappa(u\cdot\mathcal{N})(\pm\ell)\pm\sigma F^3(\pm\ell)-\kappa\hat{\mathscr{W}}(\pa_t\eta(\pm\ell)),
\end{aligned}
\right.
\eeq
Using the mean curvature term shows then that $\pa_t^3\xi$ has the same regularity as $\pa_t^2\xi$, namely inclusion in $H^1$. This allows us to estimate the term \eqref{eq:difficulty} while retaining the same basic form of the energy-dissipation estimates that the problem \eqref{eq:linear_stokes} enjoys.  We thus base our analysis on this $\epsilon$-perturbed problem.


\textbf{Solving the $\epsilon-$problem}.  We construct solutions to \eqref{eq:linear_perturbed_stokes} by a Galerkin method.  The finite dimensional approximations must satisfy the condition $\dive_\mathcal{A} u=0$, which is a time-dependent condition since $\mathcal{A}$ varies in time.  This presents the technical difficulty of needed a time-dependent basis for the Galerkin scheme.   Fortunately, the analysis of Theorem 4.3 in \cite{GT1} provides exactly the needed basis.  In the Galerkin scheme we integrate the equation $\partial_t \xi = u \cdot \mathcal{N}$ in time in order to solve for $\xi$ in terms of $u$.  Upon plugging this into the equation we arrive at an integral equation for the finite dimensional approximations of $u$, which can be readily solved with standard techniques.  We then develop a collection of a priori estimates that allows us to pass to the limit in the approximations to produce a solution for which we know that the time derivatives exist.  This is the content of Theorem \ref{thm:linear_low}.  After this we show that the solutions enjoy certain needed regularity gains.  This is the content of Theorem \ref{thm:higher order}, which crucially exploits the $\epsilon-$perturbation.

\textbf{Contraction mapping}.  Proceeding from the linear problem, we seek to develop solutions to an $\epsilon-$approximation of the nonlinear problem, namely \eqref{eq:modified_geometric perturbation}.  We accomplish this via a contraction mapping argument on a complete metric space determined by the estimates available from Theorem \ref{thm:higher order}.  Here we encounter a number of challenges.  First the metric space must be tuned through the selection of a time-scale parameter and an energy smallness parameter to show that the linear solution map takes the metric space to itself.  Proving this requires a careful control of the structure of the estimates provided by Theorem \ref{thm:higher order}.  With the mapping in hand we must show that it is a contraction.  Unfortunately, we cannot use the natural high-regularity norms as the metric on the space, as we cannot show that we get a contraction at high regularity.  This forces us to endow the metric space with a lower regularity metric, but this does not cause much harm due to weak lower semicontinuity arguments.  Thus we can ultimately prove in Theorem \ref{thm: fixed point} that the solution map contracts and hence that there exist solutions to the nonlinear $\epsilon$--perturbed Stokes equation, at least for small time  $T_\epsilon>0$.

\textbf{Continuity method for uniform energy estimate}. In principle the temporal existence interval $T_\epsilon$ may tend to $0$ as $\epsilon \to 0$, so to send $\epsilon \to 0$ in a useful way we must show that this does not happen.  Due to arguments from Section 8 in \cite{GT2}, we can employ a continuity method to get  uniform bounds for the $\epsilon-$solutions.  This is accomplished in Theorem \ref{thm:uniform_energy}.  The estimates are then enough to extend the solutions to temporal existence intervals independent of $\epsilon$.  The bounds also provide us with enough control to send $\epsilon \to 0$ and recover solutions to the original problem \eqref{eq:geometric perturbation}, completing the proof of Theorem \ref{thm:main}.

\subsection{Notation and terminology}
Now, we mention some definitions, notation and conventions that we will use throughout this paper.

\begin{enumerate}[1.]
  \item Constants. The symbol $C>0$ will denote a universal constant that only depends on the parameters of the problem and $\Om$, but does not depend on the data, etc. They are allowed to change from line to line. We will write $C=C(z)$ to indicate that the constant $C$ depends on $z$. We will write $a\lesssim b$ to mean that $a\le C b$ for a universal constant $C>0$.\\
  \item Norms. We will write $H^k$ for $H^k(\Om)$ for $k\ge0$, and $H^s(\Sigma)$ with $s\in\mathbb{R}$ for usual Sobolev spaces. We will typically write $H^0=L^2$, though we will also use $L^2([0,T];H^k)$ (or $L^2([0,T];H^s(\Sigma))$) to denote the space of temporal square--integrable functions with values in $H^k$ (or $H^s(\Sigma)$).  Sometimes we will write $\|\cdot\|_k$ instead of $\|\cdot\|_{H^k(\Om)}$ or $\|\cdot\|_{H^k(\Sigma)}$.  When we do this it will be clear on which set the norm is evaluated from the context and the argument of the norm.
\end{enumerate}

\subsection{Plan of the paper}
In Section 2, we review the machinery of time-dependent function spaces, $\dive_{\mathcal{A}}$--free vector fields, and  weighted Sobolev spaces.   In Section 3, we construct the initial data and derive estimates.  In Section 4  we study the local well-posedness of the $\epsilon$--linear problem.  In Section 5 we construct solutions to \eqref{eq:geometric perturbation} using a contraction mapping argument and a continuity method, and then we complete the proof of the main result.

\section{Functional setting and basic estimates}

\subsection{Function spaces}

First, we define some time-independent spaces:
\beq
\mathring{H}^0(\Om)=\{p\in H^0(\Om)|\int_{\Om}p=0\},
\eeq
\beq
\mathring{H}^0((-\ell,\ell))=\{\eta\in {H}^0((-\ell,\ell))|\int_{-\ell}^\ell\eta=0\},
\eeq
\beq\label{space_H_0}
\mathring{H}^k(\Om)=H^k(\Om)\cap\mathring{H}^0(\Om),\quad \mathring{H}^s((-\ell,\ell))=H^s((-\ell,\ell))\cap\mathring{H}^0((-\ell,\ell)),
\eeq
\beq
{}_0H^1(\Om)=\{u\in H^1(\Om)|u\cdot\nu=0\ \text{on}\ \Sigma_s\},
\eeq
endowed with the usual $H^1$ norm.  We also set
\beq
W:=\{u\in{}_0H^1(\Om)|u\cdot\mathcal{N}_0\in H^1(-\ell, \ell)\cap\mathring{H}^0(-\ell, \ell)\},
\eeq
endowed with norm $\|u\|_W:=\|u\|_1+\|u\cdot\mathcal{N}_0\|_{H^1((-\ell,\ell))}$, and we write
\beq
V:=\{u\in W| \dive u=0\}.
\eeq

Throughout the paper we will often utilize the following Korn-type inequality.
\begin{lemma}\label{lem:korn}
  For any $u\in{}_0H^1(\Om)$, it holds that
  \beq\label{eq:korn}
  \|u\|_1^2\lesssim \|\mathbb{D}u\|_0^2.
  \eeq
\end{lemma}
\begin{proof}
The inequality \eqref{eq:korn} follows easily from the inequality
  \beq\label{eq:korn_1}
  \|u\|_1^2\lesssim \|\mathbb{D}u\|_0^2+\|u\|_0^2 \text{ for all } u \in H^1(\Om),
  \eeq
and a standard compactness argument.  The inequality \eqref{eq:korn_1} is may be proved in various ways.  See \cite{N} for a direct proof.  It can also be derived from the Ne\v{c}as inequality: see for example Lemma IV.7.6 in \cite{boyer_fabrie}.
\end{proof}

Suppose that $\eta$ is given and that $\mathcal{A}$, $J$ and $\mathcal{N}$, etc are determined in terms of $\eta$. Let us define
\beq
((u,v)):=\int_{\Om}\frac{\mu}{2}\mathbb{D}_{\mathcal{A}}u:\mathbb{D}_{\mathcal{A}}vJ+\int_{\Sigma_s}\beta(u\cdot\tau)(v\cdot\tau)J.
\eeq
We also define
\beq
(\phi,\psi)_{1,\Sigma}:=\int_{-\ell}^{\ell}g\phi\psi+\sigma\frac{\pa_1\phi\pa_1\psi}{(1+|\pa_1\zeta_0|^2)^{3/2}},
\eeq
and
\beq\label{sum_point}
[a,b]_\ell:=\kappa(a(\ell)b(\ell)+a(-\ell)b(-\ell)).
\eeq
\begin{remark}\label{rem:norm}
  Throughout this paper, we write $\|\xi\|_{1,\Sigma}$ as
  \[
  \|\xi\|_{1,\Sigma}^2=\int_{-\ell}^{\ell}g|\xi|^2+\sigma\frac{|\pa_1\xi|^2}{(1+|\pa_1\zeta_0|^2)^{3/2}},
  \]
  and write $[\phi]_\ell$ as
  \[
  [\phi]_\ell^2=\kappa (\phi(\ell)^2+\phi(-\ell)^2).
  \]
\end{remark}

For convenience, let us define the spaces
\beq
\mathcal{H}^0(\Om):=\{u:\Om\rightarrow\mathbb{R}^2|\sqrt{J}u\in H^0(\Om)\},
\eeq
with norm $\|u\|_{\mathcal{H}^0(\Om)}:=(\int_\Om|u|^2J)^{1/2}$ and
\beq
{}_0\mathcal{H}^1(\Om):=\{u:\Om\rightarrow\mathbb{R}^2|((u,u))<\infty, u
\cdot\nu=0\ \text{on}\ \Sigma_s\},
\eeq
endowed with norm $\|u\|_{{}_0\mathcal{H}^1(\Om)}=((u,u))^{1/2}$.

Let us define some time-dependent spaces. We define the space
\beq
\mathcal{W}(t):=\{v(t)\in{}_0\mathcal{H}^1(\Om)|v\cdot\mathcal{N}\in H^1(-\ell,\ell)\cap\mathring{H}^0(-\ell,\ell)\},
\eeq
which we endow with the inner-product
\beq
(u,v)_{\mathcal{W}}=((u,v))+(u\cdot\mathcal{N},v\cdot\mathcal{N})_{1,\Sigma}.
\eeq
We also define the subspace $\mathcal{V}(t)$ of $\mathcal{W}(t)$ by
\beq
\mathcal{V}(t):=\{u(t)\in\mathcal{W}(t)|\dive_{\mathcal{A}}u=0\}.
\eeq
Finally, we define the inner products on $L^2([0, T]; H^k(\Om))$ for $k=0,1$ as
\beq
(u,v)_{\mathcal{H}^1_T}=\int_0^T(u(t),v(t))_{\mathcal{H}^1}\,\mathrm{d}t,
\eeq
and write $\mathcal{H}^1_T$ as the corresponding spaces with the corresponding norms $\|u\|_{\mathcal{H}^1_T}$.
We define the subspaces of $\mathcal{H}^1_T$ as follows:
\beq
{}_0\mathcal{H}^1_T:=\{u\in\mathcal{H}^1_T| u
\cdot\nu=0\ \text{on}\ \Sigma_s\},
\eeq
which we endow with the norm $\mathcal{H}^1_T$,
\beq
\mathcal{W}_T:=\{v\in{}_0\mathcal{H}^1_T|v\cdot\mathcal{N}\in H^1(-\ell,\ell)\cap\mathring{H}^0(-\ell,\ell)\},
\eeq
which we endow with the norm $\|v\|_{\mathcal{W}_T}:=\|v\|_{\mathcal{H}^1_T}+\left(\int_0^T\|v\cdot\mathcal{N}\|_{H^1(-\ell,\ell)}^2\right)^{1/2}$, and
\beq
\mathcal{V}_T:=\{u\in\mathcal{W}_T|\dive_{\mathcal{A}}u=0\}.
\eeq
The following lemma implies that ${}_0\mathcal{H}^1(\Om)$ is equivalent to ${}_0H^1(\Om)$.
\begin{lemma}\label{lem:equivalence_norm}
There exists a universal $\alpha_0>0$ such that if
\beq
\sup_{0\le t\le T}\|\eta(t)\|_{W_\delta^{5/2}}<\alpha_0,
\eeq
then
\beq
\frac{1}{\sqrt{2}}\|u\|_k\le\|u\|_{\mathcal{H}^k}\le\sqrt{2}\|u\|_k
\eeq
for $k=0, 1$ and for all $t\in[0, T]$. As a consequence, for $k=0, 1$,
  \beq
  \|u\|_{L^2H^k(\Om)}\le \|u\|_{\mathcal{H}^k_T(\Om)}\le \|u\|_{L^2H^k(\Om)}.
  \eeq
\end{lemma}
\begin{proof}
The case $k=0$ is proved in Lemma 2.1 in \cite{GT1}.  A result similar to that stated above for $k=1$ is also prove in \cite{GT1} for a norm not involving the boundary terms.  However, the argument used there may be readily coupled to a trace estimate to handle the boundary term.
\end{proof}

For our problem, we need weighted Sobolev spaces. Suppose that $\om\in(0,\pi)$ is the angle formed by $\zeta_0$ at the corner $M$ of $\Om$, for $M=\{(-\ell, \zeta_0(-\ell)), (\ell, \zeta_0(\ell))\}$  the corner points of $\Om$. We now introduce the critical weight $\delta_\om:=\max\{0, 2-\pi/\om\}\in[0,1)$.  For $\delta\in(\delta_\om,1)$, we define
\beq\label{def:weighted_sobolev}
W_{\delta}^k(\Om)(t) :=\{u(t)|\|u(t)\|_{W_{\delta}^k(\Om)}<\infty\},
\eeq
with the norm
\beq
\|u(t)\|_{W_{\delta}^k(\Om)}=\bigg(\sum_{|\alpha|\le k}\int_{\Om} d^{2\delta}|\pa^\alpha u(x,t)|^2\,\mathrm{d}x\bigg)^{1/2},
\eeq
where $d=dist(\cdot, M)$. A consequence of Hardy's inequality (for example Lemma 7.1.1 in \cite{KMR1}) reveals that we have the continuous embeddings
\beq
W_{\delta}^1(\Om)\hookrightarrow H^0(\Om),\ W_{\delta}^2(\Om)\hookrightarrow H^1(\Om),\ H^1(\Om)\hookrightarrow W_{-\delta}^0(\Om),
\eeq
when $\delta\in(0,1)$.

The trace spaces $W_{\delta}^{k-1/2}(\pa\Om)$ can be defined in the usual way: see for example Section 7.1.3 in \cite{KMR1}. It can be shown that the following useful lemma holds.

\begin{lemma}
  Suppose that $0<\delta<1$ and $\tau$ be the unit tangential of $\pa\Om$. Then
  \beq
  \left|\int_{\pa\Om}f(v\cdot\tau)\right|\lesssim \|f\|_{W_{\delta}^{1/2}(\pa\Om)}\|v\|_{H^1(\Om)},
  \eeq
  for all $f\in W_{\delta}^{1/2}(\pa\Om)$ and $v\in H^1(\Om)$.
\end{lemma}
\begin{proof}
  We choose $p$, $q$ such that $1<p<\frac{2}{1+\delta}$ and $\frac1p+\frac1q=1$. Employing the H\"older inequality, we may derive that
  \beq
  \left|\int_{\pa\Om}f(v\cdot\tau)\right|\le \|f\|_{L^p(\pa\Om)}\|v\|_{L^q(\pa\Om)}.
  \eeq
  The Sobolev embedding implies that
  \beq
  \|f\|_{L^p(\pa\Om)}\lesssim \|f\|_{W_{\delta}^{1/2}(\pa\Om)},
  \eeq
  and the Sobolev embedding together with standard theory imply that
  \beq
  \|v\|_{L^q(\pa\Om)}\lesssim \|v\|_{H^{1/2}(\pa\Om)}\lesssim \|v\|_{H^1(\Om)}.
  \eeq
\end{proof}
Also, we define the spaces
\beq\label{space_W_0}
\mathring{W}_\delta^k(\Om):=\left\{u\in W_\delta^k(\Om)|\int_{\Om}u=0 \right\},
\eeq
for $k\ge1$.
The spaces $L^2([0, T];W_\delta^k(\Om))$ is defined by
\beq
\|u\|_{L^2([0, T];W_\delta^k(\Om))}^2:=\int_0^T\|u(t)\|_{W_\delta^k(\Om)}^2<\infty.
\eeq

Now, we want to show that the time-independent spaces are related to the time-dependent spaces. We consider the matrix
\beq\label{def:M}
M:=M(t)=K\nabla\Phi=(J\mathcal{A}^\top)^{-1},
\eeq
which induces a linear operator $\mathcal{M}_t : u\mapsto M(t)u$.

\begin{proposition}\label{prop:isomorphism}
  Assume that $\eta\in H^r((-\ell, \ell))$ for $r>\frac32$.
  \begin{enumerate}[(1)]
  \item For each $t
  \in [0,T]$, $\mathcal{M}_t $ is a bounded isomorphism from $H^k(\Om)$ to $H^k(\Om)$ for $k=0, 1, 2$.
  \item For each $t
  \in [0,T]$, $\mathcal{M}_t $ is a bounded isomorphism from ${}_0H^1(\Om)$ to ${}_0\mathcal{H}^1(\Om)$. Moreover,
  \beq
  \|Mu\|_{{}_0\mathcal{H}^1}\lesssim (1+\|\eta\|_r)\|u\|_1
  \eeq
  \item Let $u\in H^1(\Om)$. Then $\dive u=p$ if and only if $\dive_{\mathcal{A}}(Mu)=Kp$.
  \end{enumerate}
\end{proposition}
\begin{proof}
  See \cite{GT1} and \cite{GT2}.
\end{proof}
The following proposition is also useful.
\begin{proposition}\label{prop:solid_boundary}
  If $u\cdot\nu=0$ on $\Sigma_s$, then $Ru\cdot\nu=0$ on $\Sigma_s$, where $R:=\pa_tMM^{-1}$.
\end{proposition}
\begin{proof}
  According to Proposition 4.4 in \cite{GT2}, we have known that $Mu\cdot\nu=0\Leftrightarrow u\cdot\nu=0$ on $\Sigma_s$, which implies that
  $M^{-1}u\cdot\nu=0\Leftrightarrow u\cdot\nu=0$ on $\Sigma_s$.
  Then by definition of $R$,
  \beq
  Ru\cdot\nu=\pa_tMM^{-1}u\cdot\nu=-M\pa_t(M^{-1}u)\cdot\nu=0,
  \eeq
  since $\pa_t(M^{-1}u)\cdot\nu=\pa_t(M^{-1}u\cdot\nu)=0$.
\end{proof}

\section{Initial data}\label{section_initial}

\subsection{Construction of initial data}\label{subsection_1}
Before we study the well-posedness of \eqref{eq:geometric perturbation}, we first consider the initial data and the initial energy $\mathcal{E}(0)$. Suppose that $\eta_0\in W_\delta^{5/2}(\Sigma)$, $\pa_t\eta(0)\in H^{3/2}(\Sigma)$, $\pa_t^2\eta(0)\in H^1(\Sigma)$ and that
\[
\mathfrak{E}_0(\eta):=\|\eta_0\|_{W_\delta^{5/2}(\Sigma)}^2+\|\pa_t\eta(0)\|_{H^{3/2}(\Sigma)}^2+\sum_{j=0}^2\|\pa_t^j\eta(0)\|_{H^1(\Sigma)}^2\le\alpha
\]
where $\alpha>0$ is small enough to satisfy the conditions in Lemma \ref{lem:equivalence_norm} and Theorem 5.8 in \cite{GT2}. We now construct the initial data $u(t=0)=u_0$ and $p(t=0)=p_0$. When $t=0$, we consider the elliptic equation
\beq
\left\{
\begin{aligned}
  &\dive_{\mathcal{A}(0)}S_{\mathcal{A}(0)}(p_0,u_0)=0,\quad &\text{in}&\quad \Om,\\
    &\dive_{\mathcal{A}(0)}u_0=0,\quad &\text{in}&\quad \Om,\\
    &u_0\cdot\mathcal{N}(0)=\pa_t\eta(0),\quad &\text{on}&\quad\Sigma,\\
    &\mu\mathbb{D}_{\mathcal{A}(0)}u_0\mathcal{N}(0)\cdot\mathcal{T}(0)=0,\quad &\text{on}&\quad\Sigma,\\
    &u_0\cdot\nu=0,\quad &\text{on}&\quad\Sigma_s,\\
    &\mu\mathbb{D}_{\mathcal{A}(0)}u_0\nu\cdot\tau-\beta u_0\cdot\tau=0,\quad &\text{on}&\quad\Sigma_s.
\end{aligned}
\right.
\eeq
We employ the Theorem 5.9 in \cite{GT2} to deduce that there exists a unique $(u_0,p_0)\in W_\delta^2\times\mathring{W}_\delta^1$, and
\beq
\|u_0\|_{W_\delta^2}^2+\|p_0\|_{\mathring{W}_\delta^1}^2\lesssim \|\pa_t\eta(0)\|_{W_\delta^{3/2}}^2\lesssim\|\pa_t\eta(0)\|_{3/2}^2.
\eeq
Clearly, from the embedding $W_\delta^2(\Om)\hookrightarrow H^1(\Om)$ and the boundary condition, $u_0\in\mathcal{V}(0)$.

Then we construct $\pa_tu(0)$ and $\pa_tp(0)$. In order to preserve the divergence free condition, we construct $D_tu(0)$ instead of $\pa_tu(0)$, where $D_tu$ is defined in \eqref{def:Dt_u}. Now we temporally differentiate the equation \eqref{eq:geometric perturbation}, then take $t=0$,
\beq\label{eq:initial_weak}
\left\{
\begin{aligned}
  &\dive_{\mathcal{A}(0)}S_{\mathcal{A}(0)}(\pa_tp(0),D_tu(0))=\tilde{F}(0),\quad &\text{in}&\quad \Om,\\
    &\dive_{\mathcal{A}(0)}D_tu(0)=0,\quad &\text{in}&\quad \Om,\\
    &S_{\mathcal{A}(0)}(\pa_tp(0),D_tu(0))\mathcal{N}(0)=g\pa_t\eta(0)\mathcal{N}(0)-\sigma\pa_1\left(\frac{\pa_1\pa_t\eta(0)}{(1+|\pa_1\zeta_0|)^{3/2}}\right)\mathcal{N}(0)\\
    &\quad+\pa_tF^3(0)\mathcal{N}(0)+\tilde{F}^4(0),\quad &\text{on}&\quad\Sigma,\\
    &(S_{\mathcal{A}(0)}(\pa_tp(0),D_tu(0))\nu-\beta D_tu(0))\cdot\tau=\tilde{F}^5,\quad &\text{on}&\quad\Sigma_s,\\
    &D_tu(0)\cdot\nu=0,\quad &\text{on}&\quad\Sigma_s,\\
    &D_tu(0)\cdot\mathcal{N}(0)=\pa_t^2\eta(0),\quad &\text{on}&\quad\Sigma,\\
    &\kappa\pa_t^2\eta(\pm\ell,0)+\kappa\pa_t\hat{\mathscr{W}}(\pa_t\eta(\pm\ell))(0)=\mp\sigma\left(\frac{\pa_1\pa_t\eta(0)}{(1+|\zeta_0|^2)^{3/2}}+\pa_tF^3(0)\right)(\pm\ell),
\end{aligned}
\right.
\eeq
where
\begin{align*}
\tilde{F}^1(0)&=-\dive_{\pa_t\mathcal{A}(0)}S_{\mathcal{A}(0)}(p_0,u_0)+\mu\dive_{\mathcal{A}(0)}\mathbb{D}_{\pa_t\mathcal{A}(0)}u_0+\mu\dive_{\mathcal{A}(0)}\mathbb{D}_{\mathcal{A}(0)}(R(0)u_0),\\
\pa_tF^3(0)&=\pa_z\mathcal{R}(\pa_1\zeta_0,\pa_1\eta_0)\pa_1\pa_t\eta(0),\\
\tilde{F}^4(0)&=\mu\mathbb{D}_{\mathcal{A}(0)}(R(0)u_0)\mathcal{N}(0)+\mu\mathbb{D}_{\pa_t\mathcal{A}(0)}u_0\mathcal{N}(0)\\
  &\quad+\left[g\eta_0-\sigma\pa_1\left(\frac{\pa_1\eta_0}{(+|\pa_1\zeta_0|^2)^{3/2}}+\mathcal{R}(\pa_1\zeta_0,\pa_1\eta_0)\right)\right]\pa_t\mathcal{N}(0),\\
\tilde{F}^5(0)&=\mu\mathbb{D}_{\mathcal{A}(0)}(R(0)u_0)\nu\cdot\tau+\mu\mathbb{D}_{\pa_t\mathcal{A}(0)}u_0\nu\cdot\tau+\beta R(0)u_0\cdot\tau.
\end{align*}
Then we have the pressureless weak formulation
\beq\label{weak_initial}
\begin{aligned}
  &((D_tu(0),w))+[D_tu(0)\cdot\mathcal{N}(0), w\cdot\mathcal{N}(0)]_\ell+[\pa_t\hat{\mathscr{W}}(\pa_t\eta)(0),w\cdot\mathcal{N}(0)]_\ell\\
  &=-(\pa_t\eta(0),w\cdot\mathcal{N}(0))_{1,\Sigma}-\int_{-\ell}^\ell\pa_z\mathcal{R}(\pa_1\zeta_0,\pa_1\eta_0)\pa_1\pa_t\eta(0)\pa_1(w\cdot\mathcal{N}(0))\\
  &\quad-\int_{-\ell}^\ell\left[g\eta_0-\sigma\pa_1\left(\frac{\pa_1\eta_0}{(+|\pa_1\zeta_0|^2)^{3/2}}+\mathcal{R}(\pa_1\zeta_0,\pa_1\eta_0)\right)\right]\pa_t\mathcal{N}(0)\cdot w-\int_{\Sigma_s}\beta (R(0)u_0\cdot\tau)(w\cdot\tau)J(0)\\
  &\quad-\int_\Om\left(\dive_{\pa_t\mathcal{A}(0)}S_{\mathcal{A}(0)}(p_0,u_0)\cdot w+\frac\mu2\mathbb{D}_{\pa_t\mathcal{A}(0)}u_0:\mathbb{D}_{\mathcal{A}(0)}w+\frac\mu2\mathbb{D}_{\mathcal{A}(0)}(R(0)u_0):\mathbb{D}_{\mathcal{A}(0)}w\right)J(0),
\end{aligned}
\eeq
for each $w\in \mathcal{V}(0)$. Then utilizing the last equation of \eqref{eq:initial_weak}, we may rewrite the weak formulation as
 \beq\label{new_weak_initial}
\begin{aligned}
  &B(D_tu(0),w):=((D_tu(0),w))+(D_tu(0)\cdot\mathcal{N}(0),w\cdot\mathcal{N}(0))_{1,\Sigma}\\
  &=(\pa_t^2\eta(0),w\cdot\mathcal{N}(0))_{1,\Sigma}-(\pa_t\eta(0),w\cdot\mathcal{N}(0))_{1,\Sigma}-\int_{-\ell}^\ell\pa_z\mathcal{R}(\pa_1\zeta_0,\pa_1\eta_0)\pa_1\pa_t\eta(0)\pa_1(w\cdot\mathcal{N}(0))\\
  &\quad-\int_{-\ell}^\ell\left[g\eta_0-\sigma\pa_1\left(\frac{\pa_1\eta_0}{(+|\pa_1\zeta_0|^2)^{3/2}}+\mathcal{R}(\pa_1\zeta_0,\pa_1\eta_0)\right)\right]\pa_t\mathcal{N}(0)\cdot w-\int_{\Sigma_s}\beta (R(0)u_0\cdot\tau)(w\cdot\tau)J(0)\\
  &\quad-\int_\Om\left(\dive_{\pa_t\mathcal{A}(0)}S_{\mathcal{A}(0)}(p_0,u_0)\cdot w+\frac\mu2\mathbb{D}_{\pa_t\mathcal{A}(0)}u_0:\mathbb{D}_{\mathcal{A}(0)}w+\frac\mu2\mathbb{D}_{\mathcal{A}(0)}(R(0)u_0):\mathbb{D}_{\mathcal{A}(0)}w\right)J(0)\\
  &\quad-[\pa_t^2\eta(0), w\cdot\mathcal{N}(0)]_\ell-[\pa_t\hat{\mathscr{W}}(\pa_t\eta)(0),w\cdot\mathcal{N}(0)]_\ell:=L(w).
\end{aligned}
\eeq
Since $B(\cdot,\cdot): \mathcal{V}(0)\times \mathcal{V}(0)\rightarrow\mathbb{R}$ is a bilinear mapping satisfying
\[
B(v,w)\lesssim \|v\|_{\mathcal{W}}\|w\|_{\mathcal{W}},\quad B(v,v)=\|v\|_{\mathcal{W}}^2,
\]
and $L:\mathcal{V}(0)\rightarrow\mathbb{R}$ is a bounded linear functional on $\mathcal{V}(0)$,
 the Lax-Milgram Theorem guarantees that there exists a unique $D_tu(0)\in \mathcal{V}(0)$ such that \eqref{weak_initial} holds for each $w\in \mathcal{V}(0)$. Moreover,
\beq
\|D_tu(0)\|_1^2\lesssim \|\eta_0\|_{W_\delta^{5/2}}^2+\|\pa_t\eta(0)\|_{3/2}^2+\|\pa_t^2\eta(0)\|_1^2.
\eeq
Now from Theorem 4.6 in \cite{GT2}, we may recover $\pa_tp(0)\in \mathring{H}^0(\Om)$ such that
\beq
\|\pa_tp(0)\|_0^2\lesssim \|\eta_0\|_{W_\delta^{5/2}}^2+\|\pa_t\eta(0)\|_{3/2}^2+\|\pa_t^2\eta(0)\|_1^2.
\eeq

\subsection{Compatibility}
In the construction of initial data above, $\eta_0$, $\pa_t\eta(0)$, and $\pa_t^2\eta(0)$ need to satisfy some compatibility conditions. At the corner points $x_1=\pm\ell$,
\beq\label{compati_1}
\kappa\pa_t\eta(0)+\kappa\hat{\mathscr{W}}(\pa_t\eta(0))=\mp\sigma\left(\frac{\pa_1\eta_0}{(1+|\pa_1\zeta_0|^2)^{3/2}}+\mathcal{R}(\pa_1\zeta_0,\pa_1\eta_0)\right),
\eeq
and
\beq\label{compati_2}
\kappa\pa_t^2\eta(0)+\kappa\hat{\mathscr{W}}^\prime(\pa_t\eta(0))\pa_t^2\eta(0)=\mp\sigma\left(\frac{\pa_1\pa_t\eta(0)}{(1+|\pa_1\zeta_0|^2)^{3/2}}+\pa_z\mathcal{R}(\pa_1\zeta_0,\pa_1\eta_0)\pa_1\pa_t\eta(0)\right).
\eeq

\section{Linear problem}

Suppose that $\eta$ is given and that $\mathcal{A}$, $J$, $\mathcal{N}$, etc. are determined in terms of $\eta$.  Before turning to an analysis of the linear problem, we define various quantities in terms of $\eta$:
\beq
\begin{aligned}
  &\mathfrak{D}(\eta):=\sum_{j=0}^1\|\pa_t^j\eta\|_{L^2W_\delta^{5/2}}^2+\sum_{j=0}^2\|\pa_t^j\eta\|_{L^2H^{3/2}}^2+\|\pa_t^3\eta\|_{L^2W_\delta^{1/2}}^2+\sum_{j=1}^3\|[\pa_t^j\eta]_\ell\|_{L^2([0,T])}^2,\\
  &\mathfrak{E}(\eta):=\|\eta\|_{L^\infty W_\delta^{5/2}}^2+\|\pa_t\eta\|_{L^\infty H^{3/2}}^2+\sum_{j=0}^2\|\pa_t^j\eta\|_{L^\infty H^1}^2,\\
  &\mathfrak{K}(\eta):=\mathfrak{D}(\eta)+\mathfrak{E}(\eta),
\end{aligned}
\eeq
and
\beq
\mathfrak{E}_0=\mathfrak{E}_0(\eta):=\|\eta_0\|_{W_\delta^{5/2}}^2+\|\pa_t\eta(0)\|_{L^\infty H^{3/2}}^2+\sum_{j=0}^2\|\pa_t^j\eta(0)\|_{L^\infty H^1}^2.
\eeq
Throughout this section, we always assume that $\mathfrak{K}(\eta)\le\alpha$ and $\alpha>0$ is sufficiently small.

In the rest sections, we write $d=dist(\cdot, N)$, where $N=\{(-\ell, \zeta_0(-\ell)), (\ell, \zeta_0(\ell))\}$ is the set of corner points of $\pa\Om$. In the subsequent estimates, the following lemma is useful.  The proof  is trivial, so we omit it.

\begin{lemma}\label{lem:distance}
  Suppose that $d=dist(\cdot, N)$ and that $0<\delta<1$. Then $d^{-\delta}\in L^r(\Om)$ for $2<r<\frac2\delta$.
\end{lemma}

For the purpose of constructing solutions to the nonlinear system, we need to consider the following modified linear problem
\beq\label{eq:modified_linear}
\left\{
\begin{aligned}
  &\dive_{\mathcal{A}}S_{\mathcal{A}}(p,u)=F^1,\quad&\text{in}&\quad\Om,\\
  &\dive_{\mathcal{A}}u=0,\quad&\text{in}&\quad\Om,\\
  &S_{\mathcal{A}}(p,u)\mathcal{N}=\left(\mathcal{L}(\xi+\epsilon\partial_t\xi)-\sigma\pa_1 F^3\right)\mathcal{N}+F^4,\quad&\text{on}&\quad\Sigma,\\
  &(S_{\mathcal{A}}(p,u)\nu-\beta u)\cdot\tau=F^5,\quad&\text{on}&\quad\Sigma_s,\\
  &u\cdot\nu=0,\quad&\text{on}&\quad\Sigma_s,\\
  &\pa_t\xi=u\cdot\mathcal{N},\quad&\text{on}&\quad (-\ell,\ell),\\
  &\mp\sigma\frac{\pa_1\xi}{(1+|\pa_1\zeta_0|^2)^{3/2}}(\pm\ell)=\kappa(u\cdot\mathcal{N})(\pm\ell)\pm\sigma F^3(\pm\ell)-\kappa\hat{\mathscr{W}}(\pa_t\eta(\pm\ell)),
\end{aligned}
\right.
\eeq
where $F^3=\mathcal{R}(\pa_1\zeta_0, \pa_1\eta)$ is defined in \eqref{def:R} and $\mathcal{L}(\varphi)=g\varphi-\sigma\pa_1\left(\frac{\varphi}{(1+|\pa_1\zeta_0|^2)^{3/2}}\right)$, and \eqref{eq:modified_linear} is endowed with the initial data $\xi(0)=\eta_0$, $\pa_t\xi(0)=\pa_t\eta(0)$ and $\pa_t^2\xi(0)=\pa_t^2\eta(0)$ satisfying the compatibility \eqref{compati_1} and \eqref{compati_2}. We also assume that $\epsilon$ satisfies
\beq\label{eq:ep}
\frac{\epsilon\sigma}{(1+\min|\pa_1\zeta_0|^2)^{3/2}}\le\frac14.
 \eeq
 Here we consider the $\epsilon-$ perturbation in order to close the energy estimates for twice temporal differentiation of equations.  See the introduction for a discussion of the motivation.

 \subsection{Initial data}\label{subsection_initial}
 Since the equation \eqref{eq:modified_linear} is different from \eqref{eq:linear_stokes}, we cannot expect that the initial data for \eqref{eq:modified_linear} are the same as in Section \ref{section_initial}. However, we can use the same method as in Section \ref{section_initial} to construct the new initial data for \eqref{eq:modified_linear}.    Since $\xi(0)=\eta_0$, $\pa_t\xi(0)=\pa_t\eta(0)$ and $\pa_t^2\xi(0)=\pa_t^2\eta(0)$, we use the argument of Section \ref{subsection_1} to construct the initial data $u_0^\epsilon\in W_\delta^2(\Om)$, $p_0^\epsilon\in\mathring{W}_\delta^1(\Om)$, $D_tu^\epsilon(0)\in H^1(\Om)$ and $\pa_tp^\epsilon(0)\in\mathring{H}^0(\Om)$.  An essential ingredient in this is that the boundary conditions in the $\epsilon-$dependent modified problem \eqref{eq:modified_linear} give rise to precisely the same compatibility conditions for $\eta(0),\partial_t \eta(0), \partial_t^2 \eta(0)$ as in Section \ref{section_initial}, and so we may avoid modifying the data to enforce the compatibility conditions.  The constructed data obey the following estimates:
 \beq
 \begin{aligned}
&\|u_0^\epsilon\|_{W_\delta^2}^2+\|p_0^\epsilon\|_{\mathring{W}_\delta^1}^2\lesssim\|\pa_t\eta(0)\|_{3/2}^2,\\
&\|D_tu^\epsilon(0)\|_1^2\lesssim \|\eta_0\|_{W_\delta^{5/2}}^2+\|\pa_t\eta(0)\|_{3/2}^2+\|\pa_t^2\eta(0)\|_1^2,\\
&\|\pa_tp^\epsilon(0)\|_0^2\lesssim \|\eta_0\|_{W_\delta^{5/2}}^2+\|\pa_t\eta(0)\|_{3/2}^2+\|\pa_t^2\eta(0)\|_1^2.
\end{aligned}
\eeq

\subsection{Weak solution}
To analyze \eqref{eq:modified_linear}, we need to consider two notations of solution: weak and strong. Using the following lemma, we define the weak solutions of \eqref{eq:modified_linear}.

\begin{lemma}
  Suppose that $(u,p,\xi)$ are smooth enough and satisfy \eqref{eq:modified_linear} and that $v\in\mathcal{W}(t)$. Then
  \begin{equation}\label{eq:weak_formulation}
    \begin{aligned}
      &((u,v))-(p,\dive_{\mathcal{A}}v)_{\mathcal{H}^0}+(\xi+\epsilon\partial_t\xi,v\cdot\mathcal{N})_{1,\Sigma}+[u\cdot\mathcal{N},v\cdot\mathcal{N}]_{\ell}\\
      &=\int_{\Om}F^1\cdot vJ-\int_{-\ell}^{\ell}\sigma F^3\pa_1(v\cdot\mathcal{N})+F^4\cdot v-\int_{\Sigma_s}F^5(v\cdot\tau)J-[v\cdot\mathcal{N},\hat{\mathscr{W}}(\pa_t\eta)]_\ell+\epsilon \mathfrak{b}(\pa_t\xi,v\cdot\mathcal{N})_\ell,
    \end{aligned}
  \end{equation}
  where  $\mathfrak{b}$ denotes the bilinear form
\begin{equation*}
  \mathfrak{b}(\pa_t\xi,v\cdot\mathcal{N})_\ell=\sigma\frac{\pa_1\pa_t\xi}{(1+|\pa_1\zeta_0|^2)^{3/2}}v\cdot\mathcal{N}(\ell)-\sigma\frac{\pa_1\pa_t\xi}{(1+|\pa_1\zeta_0|^2)^{3/2}}v\cdot\mathcal{N}(-\ell).
\end{equation*}
\end{lemma}
\begin{proof}
  This can be shown in the usual way by taking the inner product of the first equation in \eqref{eq:modified_linear} with $u$, and integrating by parts over $\Om$, then employing all of the other equations in \eqref{eq:modified_linear}.  We omit the details here for the sake of brevity.
\end{proof}
\begin{definition}
 Suppose that $\mathcal{F}\in(\mathcal{H}^1_T)^\ast$. A weak solution to \eqref{eq:geometric perturbation} is a triple $(u,p,\xi)$, where
  \beq
  \begin{aligned}
  u\in L^2([0,T];{}_0H^1(\Om))\quad \text{with}\quad u(\cdot,t)\in\mathcal{V}(t)\ \text{for a.e.}\ t,\\
   p\in L^2([0,T],L^2(\Om)),\quad \xi\in L^2([0,T];H^1(-\ell,\ell)),
   \end{aligned}
  \eeq
  that satisfies
  \begin{equation}
    \begin{aligned}
      &\int_0^T((u,v))-\int_0^T(p,\dive_{\mathcal{A}}v)_{\mathcal{H}^0}+\int_0^T(\xi+\epsilon\partial_t\xi,v\cdot\mathcal{N})_{1,\Sigma}+\int_0^T[u\cdot\mathcal{N},v\cdot\mathcal{N}]_{\ell}\\
      &=\int_0^T\int_{\Om}F^1\cdot vJ-\int_0^T\int_{-\ell}^{\ell}\sigma F^3\pa_1(v\cdot\mathcal{N})+F^4\cdot v-\int_0^T\int_{\Sigma_s}F^5(v\cdot\tau)J-\int_0^T[v\cdot\mathcal{N}, \hat{\mathscr{W}}(\pa_t\eta)]_\ell\\
      &\quad+\int_0^T\epsilon \mathfrak{b}(\pa_t\xi,v\cdot\mathcal{N})_\ell,
    \end{aligned}
  \end{equation}
  for a.e. $t$ and each $v\in\mathcal{W}(t)$. If we take the test function $v\in \mathcal{V}(t)$, we have the pressureless weak solution $(u,\xi)$ satisfies
  \begin{equation}\label{eq:weak_pressureless}
    \begin{aligned}
      &\int_0^T((u,v))+\int_0^T(\xi+\epsilon\partial_t\xi,v\cdot\mathcal{N})_{1,\Sigma}+\int_0^T[u\cdot\mathcal{N},v\cdot\mathcal{N}]_{\ell}\\
      &=\int_0^T\int_{\Om}F^1\cdot vJ-\int_0^T\int_{-\ell}^{\ell}\sigma F^3\pa_1(v\cdot\mathcal{N})+F^4\cdot v-\int_0^T\int_{\Sigma_s}F^5(v\cdot\tau)J-\int_0^T[v\cdot\mathcal{N}, \hat{\mathscr{W}}(\pa_t\eta)]_\ell\\
      &\quad+\int_0^T\epsilon \mathfrak{b}(\pa_t\xi,v\cdot\mathcal{N})_\ell.
    \end{aligned}
  \end{equation}
\end{definition}
\begin{remark}\label{def:dual}
For convenience, we write
\beq
\left<\mathcal{F},v\right>_{(\mathcal{H}^1)^\ast}=\int_{\Om}F^1\cdot vJ-\int_{-\ell}^{\ell}F^4\cdot v-\int_{\Sigma_s}F^5(v\cdot\tau)J,
\eeq
and
\beq
\left<\mathcal{F},v\right>_{(\mathcal{H}^1_T)^\ast}=\int_0^T\int_{\Om}F^1\cdot vJ-\int_0^T\int_{-\ell}^{\ell}F^4\cdot v-\int_0^T\int_{\Sigma_s}F^5(v\cdot\tau)J,
\eeq
for each $v\in \mathcal{V}$.
We also write
\beq
\mathfrak{b}(\varphi,\varphi)_\ell=\mathfrak{b}(\varphi)_\ell^2.
\eeq
\end{remark}

In the following, we will see that weak solutions to \eqref{eq:weak_pressureless} will arise  as a byproduct of the construction of strong solutions to \eqref{eq:weak_pressureless}. Hence, we now ignore the existence of weak solutions and record a uniqueness result based on some integral equalities and bounds satisfied by weak solutions.
\begin{proposition}
  Weak solutions to \eqref{eq:weak_pressureless} are unique.
\end{proposition}
\begin{proof}
  If $(u^1,\xi^1)$ and $(u^2,\xi^2)$ are both weak solutions to \eqref{eq:weak_pressureless}, then $(w=u^1-u^2,\theta=\xi^1-\xi^2)$ is a weak solution with $F^1=F^3=F^4=F^5=0$ and the initial data $w(0)=\theta(0)=0$. Using the test function $w\chi_{[0,t]}\in\mathcal{V}_T$, where $\chi_{[0,t]}$ is a temporal indicator function, we have that
  \beq\label{eq:weak_diff}
  \frac12\|\theta(t)\|_{1,\Sigma}^2+\epsilon\int_0^t\|w\cdot\mathcal{N}\|_{1,\Sigma}^2+\int_0^t((w(s),w(s)))\,\mathrm{d}s+\int_0^t[w\cdot\mathcal{N}(t)]_\ell^2-\epsilon \mathfrak{b}(w\cdot\mathcal{N},w\cdot\mathcal{N})_\ell\,\mathrm{d}s=0.
  \eeq
  Since the bounds \eqref{eq:ep} for $\epsilon$, $\int_0^t[w\cdot\mathcal{N}(t)]_\ell^2-\epsilon \mathfrak{b}(w\cdot\mathcal{N},w\cdot\mathcal{N})_\ell\,\mathrm{d}s\ge0$. Thus \eqref{eq:weak_diff} implies that $w=0$, $\theta=0$. Hence, weak solutions to \eqref{eq:weak_pressureless} are unique.
\end{proof}

\subsection{Strong solution}

Before we define strong solutions, we need to define an operator $D_t$ via
\beq\label{def:Dt_u}
D_tu:=\pa_tu-Ru\quad\text{for}\quad R:=\pa_tMM^{-1},
\eeq
with $M=K\nabla\Phi$, where $K$ and $\Phi$ are defined as in \eqref{eq:components} and \eqref{def:map}, respectively. It is easy to see that $D_t$ preserves the $\dive_{\mathcal{A}}$--free condition since
\beq
  J\dive_{\mathcal{A}}(D_tv)=J\dive_{\mathcal{A}}(M\pa_t(M^{-1}v))=\dive(\pa_t(M^{-1}v))=\pa_t\dive(M^{-1}v)=\pa_t(J\dive_{\mathcal{A}}v),
\eeq
where in the second and last equality, we used the equality $J\dive_{\mathcal{A}}v=\dive(M^{-1}v)$, which is proved, according to Lemma \ref{lem:properties_A} and the definition \eqref{def:M} of $M$, as
\beq
J\dive_{\mathcal{A}}v=J\mathcal{A}_{ij}\pa_jv_i=\pa_j(J\mathcal{A}_{ij}v_i)=\dive(J\mathcal{A}^\top v)=\dive(M^{-1}v).
\eeq

We now give our definition of strong solutions.
\begin{definition}
  Suppose that the forcing functions satisfy
  \beq\label{force_condition}
  \begin{aligned}
    F^1\in L^2([0, T];W_\delta^0(\Om)), \ F^3\in L^2([0, T];W_\delta^{3/2}(\Sigma)),\\
    F^4\in L^2([0, T];W_\delta^{1/2}(\Sigma)),\ F^5\in L^2([0, T];W_\delta^{1/2}(\Sigma_s)),\\
    \mathcal{F}\in C^0([0, T];(\mathcal{H}^1)^\ast),\ \pa_t\mathcal{F}\in L^2([0, T];(\mathcal{H}^1)^\ast).
  \end{aligned}
  \eeq
  We also assume that the initial data are the same as in Section \ref{subsection_initial}. If there exists a pair $(u,p,\xi)$ achieving the initial data and satisfying the \eqref{eq:modified_linear} in the strong sense of
  \beq\label{strong_1}
    u\in L^2([0, T];W_\delta^2(\Om))\cap\mathcal{V}_T,\ p\in L^2([0, T];\mathring{W}_\delta^1(\Om)),\ \xi\in L^2([0, T];W_\delta^{5/2}(\Sigma)),
  \eeq
  and
  \beq\label{strong_2}
  \begin{aligned}
  \pa_t^j u\in L^2([0, T];H^1(\Om)),\ \pa_t^j u\in L^2([0, T];H^0(\Sigma_s)),\ [\pa_t^ju\cdot\mathcal{N}]_\ell\in L^2([0, T]),\\
   \pa_t^jp\in L^2([0, T];H^0(\Om)),\ \pa_t^j\xi\in L^2([0, T]; H^{3/2}(\Sigma)),
  \end{aligned}
  \eeq
  for $j=0, 1$, we call it a strong solution.
\end{definition}

\begin{lemma}\label{lem:continuous}
Suppose that the right-hand side of the following is finite. Then $u\in C^0([0,T];H^1(\Om))$, and
  \[
  \|u\|_{L^\infty H^1}^2\lesssim \|u_0^\epsilon\|_{W_\delta^2}^2+\|u\|_{L^2H^1}^2+\|\pa_tu\|_{L^2H^1}^2.
  \]
\end{lemma}
\begin{proof}
We first estimate
  \begin{align*}
    \frac{d}{dt}\|u\|_{H^1}^2&=\int_\Om 2u\cdot\pa_tu+2\nabla u:\nabla\pa_tu\\
    &\le \|u\|_{H^1}^2+\|\pa_tu\|_{H^1}^2.
  \end{align*}
 This may be integrated  over $[0, T]$ to see that
  \begin{align*}
  \|u\|_{L^\infty H^1}^2&\le \|u_0^\epsilon\|_{H^1}^2+\|u\|_{L^2H^1}^2+\|\pa_tu\|_{L^2H^1}^2\\
  &\lesssim \|u_0^\epsilon\|_{W_\delta^2}^2+\|u\|_{L^2H^1}^2+\|\pa_tu\|_{L^2H^1}^2,
  \end{align*}
  where the last inequality is obtained by the embedding $W_\delta^2(\Om)\hookrightarrow H^1(\Om)$.
\end{proof}

Now we state our main theorem for the strong solutions.
  \begin{theorem}\label{thm:linear_low}
    Suppose that the forcing terms $F^1$, $F^4$, and $F^5$ satisfy the condition \eqref{force_condition}, that the initial data are the same as Section \ref{subsection_initial}.
    Suppose that $\mathfrak{K}(\eta)\le\alpha$ is smaller than $\alpha_0$ in Lemma \ref{lem:equivalence_norm} and Theorem 5.9 in \cite{GT2}.
    Then there exists a unique strong solution $(u,p,\xi)$ solving \eqref{eq:modified_linear} such that
    $(u,p,\xi)$ satisfies \eqref{strong_1} and \eqref{strong_2}. The solution obeys the estimates
    \beq\label{est:galerkin}
  \begin{aligned}
  &\|u\|_{L^2 H^1}^2+\|u\|_{L^2 H^0(\Sigma_s)}^2+\|[u\cdot\mathcal{N}]_\ell\|_{L^2([0,T])}^2+\|u\|_{L^2 W_\delta^2}^2+\|\pa_tu\|_{L^2H^1}^2+\|\pa_tu\|_{L^2 H^0(\Sigma_s)}^2\\
  &\quad+\|[\pa_tu\cdot\mathcal{N}]_\ell\|_{L^2([0,T])}^2+\|p\|_{L^2 H^0}^2+\|p\|_{L^2\mathring{W}_\delta^1}^2+\|\pa_tp\|_{L^2H^0}^2+\|\xi\|_{L^\infty H^1}^2+\|\xi\|_{L^2 H^{3/2}}^2\\
  &\quad+\|\xi\|_{L^2 W_\delta^{5/2}}^2
  +\|\pa_t\xi\|_{L^\infty H^1}^2+\|\pa_t\xi\|_{L^2H^{3/2}}^2\\
  &\lesssim C(\epsilon)T\mathfrak{E}(\eta)+\mathfrak{E}_0+\|\mathcal{F}(0)\|_{(\mathcal{H}^1)^\ast}^2+\mathfrak{K}(\eta)+\mathfrak{E}(\eta)(\|F^1\|_{L^2 W_\delta^0}^2+\|F^4\|_{L^2 W_\delta^{1/2}}^2+\|F^5\|_{L^2 W_\delta^{1/2}}^2)\\
  &\quad+(1+\mathfrak{E}(\eta))\|\pa_t(F^1-F^4-F^5)\|_{(\mathcal{H}^1_T)^{\ast}}^2.
  \end{aligned}
  \eeq
    Moreover, $(D_tu,\pa_tp, \pa_t\xi)$ satisfies
    \beq\label{eq:weak_dt_u}
    \left\{
    \begin{aligned}
      &-\mu\Delta_{\mathcal{A}}D_tu+\nabla_{\mathcal{A}}\pa_tp=D_tF^1+G^1,\quad&\text{in}&\quad\Om,\\
      &\dive_{\mathcal{A}}(D_tu)=0,\quad&\text{in}&\quad\Om,\\
      &S_{\mathcal{A}}(\pa_tp, D_tu)\mathcal{N}=\mathcal{L}(\pa_t\xi+\epsilon\pa_t^2\xi)\mathcal{N}-\sigma\pa_1\pa_tF^3\mathcal{N}+\pa_tF^4+G^4,\quad&\text{on}&\quad\Sigma,\\
      &(S_{\mathcal{A}}(\pa_tp, D_tu)\nu-\beta D_tu)\cdot\tau=\pa_tF^5+G^5,\quad&\text{on}&\quad\Sigma_s,\\
      &D_tu\cdot\nu=0,\quad&\text{on}&\quad\Sigma_s,\\
      &\pa_t^2\xi=D_tu\cdot\mathcal{N},\quad&\text{on}&\quad\Sigma,\\
      &\mp\sigma\frac{\pa_1\pa_t\xi}{(1+|\pa_1\zeta_0|^2)^{3/2}}(\pm\ell)=\kappa(D_tu\cdot\mathcal{N})(\pm\ell)\pm\sigma\pa_tF^3(\pm\ell)-\kappa\pa_t\hat{\mathscr{W}}((\pa_t\eta)(\pm\ell)).
    \end{aligned}
    \right.
    \eeq
    in the weak sense of \eqref{eq:weak_pressureless}, where $G^1$ is defined by
    \beq
    G^1=R^\top\nabla_{\mathcal{A}}p+\dive_{\mathcal{A}}\left(\mathbb{D}_{\mathcal{A}}(Ru)+\mathbb{D}_{\pa_t\mathcal{A}}u-R\mathbb{D}_{\mathcal{A}}u\right),
    \eeq
    and $G^4$ by
    \beq
    G^4=\mu\mathbb{D}_{\mathcal{A}}(Ru)\mathcal{N}-(pI-\mu\mathbb{D}_{\mathcal{A}}u)\pa_t\mathcal{N}+\mu\mathbb{D}_{\pa_t\mathcal{A}}u\mathcal{N}+\mathcal{L}(\xi+\epsilon\pa_t\xi)\pa_t\mathcal{N}-\sigma\pa_1F^3\pa_t\mathcal{N},
    \eeq
    $G^5$ by
    \beq
    G^5=(\mu\mathbb{D}_{\mathcal{A}}(Ru)\nu+\mu\mathbb{D}_{\pa_t\mathcal{A}}u\nu+\beta Ru)\cdot\tau
    \eeq
    More precisely, \eqref{eq:weak_dt_u} holds in the weak sense of
    \beq
    \begin{aligned}
      &((\pa_tu,v))+(\pa_t\xi+\epsilon\pa_t^2\xi,v\cdot\mathcal{N})_{1,\Sigma}+[\pa_tu\cdot\mathcal{N},v\cdot\mathcal{N}]_\ell+[\hat{\mathscr{W}}^\prime\pa_t^2\eta,v\cdot\mathcal{N}]_\ell-\epsilon \mathfrak{b}(\pa_t\xi,v\cdot\mathcal{N})_\ell\\
   &=(\xi+\epsilon\pa_t\xi,R v\cdot\mathcal{N})_{1,\Sigma}-(p,\dive_{\mathcal{A}}(Rv))_{\mathcal{H}^0}+\int_{\Om}\left[\pa_tF^1\cdot v+\pa_tJKF^1\cdot v\right]J\\
  &\quad-\int_{-\ell}^{\ell}[\pa_tF^3\pa_1(v\cdot\mathcal{N})+F^3\pa_1(v\cdot\pa_t\mathcal{N})+\pa_tF^4\cdot v]-\int_{\Sigma_s}\left[\pa_tF^5 v+\pa_tJKF^5v\right]\cdot\tau J\\\
   &\quad-\int_\Om\frac{\mu}{2}(\mathbb{D}_{\pa_t\mathcal{A}}u:\mathbb{D}_{\mathcal{A}}v+\mathbb{D}_{\mathcal{A}}u:\mathbb{D}_{\pa_t\mathcal{A}}v+\pa_tJK\mathbb{D}_{\mathcal{A}}u:\mathbb{D}_{\mathcal{A}}v)J-\int_{\Sigma_s}\beta(u\cdot\tau)(v\cdot\tau)\pa_tJ.
    \end{aligned}
    \eeq
  \end{theorem}
  \begin{proof}
    Our proof is inspired by a result in \cite{GT1}. We divide the proof into several steps.

    Step 1 -- The Galerkin setup.

    In order to utilize the Galerkin method, we must first construct a countable basis of $H^2(\Om)\cap\mathcal{V}(t)$ for each $t\in[0, T]$. Since the requirement $\dive_{\mathcal{A}}v=0$ is time--dependent, any basis of this space must also be time--dependent. For each $t\in[0, T]$, the space $H^2(\Om)\cap\mathcal{V}(t)$ is separable, so the existence of a countable basis is not an issue. The technical difficulty is that, in order for the basis to be useful in Galerkin method, we must be able to express these time derivatives in terms of finitely many basis elements. Fortunately, it is possible to overcome this difficulty by employing the matrix $M(t)$, defined by \eqref{def:M}.

    Since $H^2(\Om)\cap V$ is separable, it possess a countable basis $\{w^j\}_{j=1}^\infty$. Note that this basis is not time--dependent. Define $v^j=v^j(t):=M(t)w^j$. According to Proposition \ref{prop:isomorphism}, $v^j(t)\in H^2(\Om)\cap\mathcal{V}(t)$, and $\{v^j(t)\}_{j=1}^\infty$ is a basis of $H^2(\Om)\cap\mathcal{V}(t)$ for each $t\in\mathbb{R}^+$. Moreover, we can express $\pa_tv^j(t)$ in terms of $v^j(t)$ as
    \beq
    \pa_tv^j(t)=\pa_tM(t)w^j=\pa_tM(t)M^{-1}(t)M(t)w^j=R(t)v^j(t),
    \eeq
    where $R(t)$ is defined by
    \beq
    R(t):=\pa_tM(t)M^{-1}(t).
    \eeq
    For any integer $m\ge1$, we define the finite dimensional space
    \[
    \mathcal{V}_m(t):=\text{span}\{v^1(t), \cdots, v^m(t)\}\subseteq H^2(\Om)\cap\mathcal{V}(t),
    \]
     and we write
    \beq
    \mathcal{P}^m_t: H^2(\Om)\rightarrow\mathcal{V}_m(t)
    \eeq
    for the $H^2(\Om)$ orthogonal projection onto $\mathcal{V}_m(t)$. Clearly, for each $v\in H^2(\Om)\cap\mathcal{V}(t)$, we have that $\mathcal{P}^m_tv\rightarrow v$ as $m\rightarrow\infty$.

    Step 2 -- Solving the approximate problem.

    For our Galerkin problem, we construct a solution to the pressureless problem as follows. For each $m\ge1$, we define an approximate solution
    \beq
    u^m(t):=d^m_j(t)v^j(t),\ \text{with}\ d^m_j: [0, T]\rightarrow\mathbb{R}\ \text{for}\ j=1, \dots, m,
    \eeq
    where as usual we use the Einstein convention of summation of the repeated index $j$. We similarly define
    \beq
    \xi^m(t)=\eta_0+\int_0^t u^m(s)\cdot\mathcal{N}(s)\,\mathrm{d}s,
    \eeq
    where we understand here that $u^m(\cdot)$ denotes the trace onto $\Sigma$.

    We want to choose the coefficients $d^m_j(t)\in C^1([0,T])$ so that
    \begin{equation}\label{eq: Galerkin}
    \begin{aligned}
      &((u^m,v))+(\xi^m+\epsilon\pa_t\xi^m,v\cdot\mathcal{N})_{1,\Sigma}+[u^m\cdot\mathcal{N},v\cdot\mathcal{N}]_{\ell}-\epsilon \mathfrak{b}(\pa_t\xi^m,v\cdot\mathcal{N})_\ell\\
      &=\int_{\Om}F^1\cdot vJ-\int_{-\ell}^{\ell}F^3\pa_1(v\cdot\mathcal{N})+F^4\cdot v-\int_{\Sigma_s}F^5(v\cdot\tau)J-[v\cdot\mathcal{N},\hat{\mathscr{W}}(\pa_t\eta)]_\ell,
    \end{aligned}
  \end{equation}
  for each $v\in \mathcal{V}_m(t)$. We supplement this with the initial data
  \beq
  u^m(0)=\mathcal{P}^m_0u_0^\epsilon\in\mathcal{V}_m(0).
  \eeq
  We may compute
  \begin{equation}
    \begin{aligned}
      (\xi^m(t)+\epsilon\pa_t\xi^m,v\cdot\mathcal{N}(t))_{1,\Sigma}=\left(\eta_0+\int_0^t u^m(s)\cdot\mathcal{N}(s)\,\mathrm{d}s+\epsilon u^m(t)\cdot\mathcal{N}(t), v\cdot\mathcal{N}(t)\right)_{1,\Sigma}\\
      =(\eta_0, v\cdot\mathcal{N}(t))_{1,\Sigma}+\epsilon d^m_i(t)(v^i\cdot\mathcal{N}(t),v\cdot\mathcal{N}(t))_{1,\Sigma}+\int_0^t d^m_i(s)(v^i\cdot\mathcal{N}(s),v\cdot\mathcal{N}(t))_{1,\Sigma}\,\mathrm{d}s.
    \end{aligned}
  \end{equation}

  Then we see that \eqref{eq: Galerkin} is equivalent to an equation for $d^m_j$ given by
  \begin{equation}\label{eq: Galerkin_2}
    \begin{aligned}
     & d^m_i((v^i,v^j))+\epsilon d^m_i(v^i\cdot\mathcal{N}(t),v^j\cdot\mathcal{N}(t))_{1,\Sigma}+\int_0^t d^m_i(s)(v^i\cdot\mathcal{N}(s),v^j\cdot\mathcal{N}(t))_{1,\Sigma}\,\mathrm{d}s\\
     &\quad+d^m_i[v^i\cdot\mathcal{N},v^j\cdot\mathcal{N}]_{\ell}-\epsilon d^m_i \mathfrak{b}(v^i\cdot\mathcal{N},v^j\cdot\mathcal{N})_\ell\\
      &=\int_{\Om}F^1\cdot v^jJ-\int_{-\ell}^{\ell}F^3\pa_1(v^j\cdot\mathcal{N})+F^4\cdot v^j-\int_{\Sigma_s}F^5(v^j\cdot\tau)J-(\eta_0, v^j\cdot\mathcal{N}(t))_{1,\Sigma}\\
      &\quad-[v^j\cdot\mathcal{N},\hat{\mathscr{W}}(\pa_t\eta)]_\ell,
    \end{aligned}
  \end{equation}
  for $i,j=1, \dots, m$.

  Since $\{v^j(t)\}_{j=1}^\infty$ is a basis of $H^2(\Om)\cap\mathcal{V}(t)$, the $m\times m$ matrix $A=(A_{jk})$with $j, k$ entry $A_{jk}=((v^j,v^k))+\epsilon (v^j\cdot\mathcal{N},v^k\cdot\mathcal{N})_{1,\Sigma}+[v^j\cdot\mathcal{N},v^k\cdot\mathcal{N}]_\ell-\epsilon \mathfrak{b}(v^i\cdot\mathcal{N},v^j\cdot\mathcal{N})_\ell$ is positive definite. For any vector $\lam=(\lam_1,\ldots,\lam_m)^\top\neq0$, a straightforward computation shows  that
  \begin{align*}
  \lam^\top A\lam=\frac\mu2\int_{\Om}|\lam_i\mathbb{D}_{A}v^i|^2J+\beta\int_{\Sigma_s}|\lam_iv^i\cdot\tau|^2J+\epsilon\|\lam_iv^i\cdot\mathcal{N}\|_{1,\Sigma}^2+[\lam_iv^i\cdot\mathcal{N}]_\ell^2-\epsilon \mathfrak{b}(\lam_iv^i\cdot\mathcal{N})_\ell^2 >0
  \end{align*}
  where the last inequality is due to the facts that $\{v^j\}_{j=1}^m$ is a basis of $\mathcal{V}_m$, $\lam\neq0$, and \eqref{eq:ep}. Thus $A$ is invertible.
  Then we view \eqref{eq: Galerkin_2} as an integral system of the form
  \beq\label{eq:integral}
  d^m(t)+\int_0^t\mathfrak{C}(t,s)d^m(s)\,\mathrm{d}s=\mathfrak{F}(t),
  \eeq
  where the $m\times m$ matrix $\mathfrak{C}$ belongs to $C^1(D)$ with $D=\{(t,s)|0\le s\le t, 0\le t\le T\}$, and the forcing term $\mathfrak{F}\in C^1([0,T])$ since $\pa_t\eta(\pm \ell,\cdot) \in
H^2((0,T)) \hookrightarrow C^{1,1/2}([0,T])$.

   From the usual theory of integral equations (for instance, see \cite{Wa}), there exists a unique $d^m\in C^1([0,T])$ satisfying $d^m=A d^m=\mathfrak{F}(t)-\int_0^t\mathfrak{C}(t,s)d^m(s)\,\mathrm{d}s$.

  Step 3 -- Estimates for initial data.

  For $u^m(0)$, since $\mathcal{P}_0^m$ is the orthogonal projection, we may use Lemma \ref{lem:equivalence_norm}, the Sobolev embeddings, and the initial data in Section \ref{subsection_initial} to obtain the bounds
  \beq\label{initial_u_m}
  \|u^m(0)\|_{{}_0\mathcal{H}^1}\lesssim \|u^m(0)\|_1\lesssim \|u^m(0)\|_{W_\delta^2}\lesssim \|u_0^\epsilon\|_{W_\delta^2}\lesssim\|\pa_t\eta(0)\|_{3/2},
  \eeq
  and
  \beq\label{initial_pa_t_xi_m}
  \|\pa_t\xi^m(0)\|_1=\|u^m(0)\cdot\mathcal{N}(0)\|_1\lesssim \|u_0^\epsilon\cdot\mathcal{N}(0)\|_1\lesssim \|\pa_t\eta(0)\|_1.
  \eeq

  Step 4 -- Energy estimates for $u^m$.

  By construction, $u^m(t)\in\mathcal{V}_m(t)$, so we may choose $v=u^m$ as a test function \eqref{eq: Galerkin}. Since $\pa_t\xi^m=u^m\cdot\mathcal{N}$, we have that
   \beq
   \begin{aligned}
     &\frac{d}{dt}\frac12\|\xi^m\|_{1,\Sigma}^2+\epsilon\|\pa_t\xi^m\|_{1,\Sigma}^2+\|u^m\|_{{}_0\mathcal{H}^1(\Om)}^2+[u^m\cdot\mathcal{N}]_\ell^2-\epsilon \mathfrak{b}(u^m\cdot\mathcal{N})_\ell^2\\
      &=\int_{\Om}F^1\cdot u^mJ-\int_{-\ell}^{\ell}F^3\pa_t\pa_1\xi^m+F^4\cdot u^m-\int_{\Sigma_s}F^5(u^m\cdot\tau)J-[\hat{\mathscr{W}}(\pa_t\eta),u^m\cdot\mathcal{N}]_\ell,
   \end{aligned}
   \eeq
   using the H\"older inequality for $1<q<\frac{2}{1+\delta}$ with $0<\delta<1$ and $\frac{1}{p}+\frac{1}{q}=1$, Lemma \ref{lem:distance} with $2<r<\frac2\delta$ and $r^\prime=\frac{2r}{r-2}$, the Cauchy inequality, Sobolev inequalities, and the usual trace theory,  we have that
   \begin{equation}
    \begin{aligned}
      &\frac{d}{dt}\frac12\|\xi^m\|_{1,\Sigma}^2+\epsilon\|\pa_t\xi^m\|_{1,\Sigma}^2+\|u^m\|_{{}_0\mathcal{H}^1(\Om)}^2+[u^m\cdot\mathcal{N}]_\ell^2-\epsilon \mathfrak{b}(u^m\cdot\mathcal{N})_\ell^2\\
      &\lesssim \|J\|_{L^\infty(\Sigma)}\|F^1\|_{W_\delta^0}\|d^{-\delta}\|_{L^r}\|u^m\|_{L^{r^\prime}}+\|F^4\|_{L^q}\|u^m\|_{L^p(\Sigma)}+\|J\|_{L^\infty(\Sigma_s)}\|F^5\|_{L^q}\|u^m\|_{L^p(\Sigma_s)}\\
      &\quad+C(\epsilon)\|\eta\|_1^2\|\eta\|_{W_\delta^{5/2}}^2+[\hat{\mathscr{W}}(\pa_t\eta)]_\ell^2\\
      &\lesssim(1+\|\eta\|_{W_\delta^{5/2}})(\|F^1\|_{W_\delta^0}+\|F^4\|_{W_\delta^{1/2}}+\|F^5\|_{W_\delta^{1/2}})\|u^m\|_1+C(\epsilon)\|\eta\|_1^2\|\eta\|_{W_\delta^{5/2}}^2+\|\pa_t\eta\|_1^2[\pa_t\eta]_\ell^2\\
      &\lesssim (1+\|\eta\|_{W_\delta^{5/2}}^2)(\|F^1\|_{W_\delta^0}^2+\|F^4\|_{W_\delta^{1/2}}^2+\|F^5\|_{W_\delta^{1/2}}^2)+C(\epsilon)\|\eta\|_1^2\|\eta\|_{W_\delta^{5/2}}^2+\|\pa_t\eta\|_1^2[\pa_t\eta]_\ell^2.
    \end{aligned}
  \end{equation}

Then we employ the Gronwall's inequality and \eqref{eq:ep} to arrive at the bound
  \begin{equation}\label{eq:xi_m}
    \begin{aligned}
      &\sup_{0\le t\le T}\|\xi^m\|_1^2+\epsilon\|\pa_t\xi^m\|_{L^2H^1}^2+\|u^m\|_{L^2H^1(\Om)}+\|u^m\|_{L^2H^0(\Sigma_s)}+\int_0^T[u^m\cdot\mathcal{N}]_\ell^2\\
      &\lesssim (1+\|\eta\|_{L^\infty W_\delta^{5/2}}^2)(\|F^1\|_{L^2W_\delta^0}^2+\|F^4\|_{L^2W_\delta^{1/2}}^2+\|F^5\|_{L^2W_\delta^{1/2}}^2)+C(\epsilon)T\|\eta\|_{L^\infty H^1}^2\|\eta\|_{L^\infty W_\delta^{5/2}}^2\\
      &\quad+\|\pa_t\eta\|_{L^\infty H^1}^2\|[\pa_t\eta]_\ell\|_{L^2}^2+\|\eta_0\|_{W_\delta^{5/2}}^2.
    \end{aligned}
  \end{equation}

  Step 5 -- Energy estimate for $\pa_tu^m$.

  Suppose that $v=b_i^mv^i$ for $b_i^m\in C^1([0, T])$. It is easily verified that $\pa_tv(t)-R(t)v(t)\in\mathcal{V}_m(t)$ as well. We now use this $v$ in \eqref{eq: Galerkin}, temporally differentiate the resulting equation, and then subtract this from the equation \eqref{eq: Galerkin} with test function $\pa_tv-Rv$. This eliminates the terms of $\pa_tv$ and leaves us with the equality
  \begin{equation}\label{eq:pa_tu_m_1}
  \begin{aligned}
\mathrm{}    &((\pa_tu^m,v))+(u^m\cdot\mathcal{N},v\cdot\mathcal{N})_{1,\Sigma}+\epsilon(\pa_t(u^m\cdot\mathcal{N}),v\cdot\mathcal{N})_{1,\Sigma}+(\xi^m+\epsilon\xi^m_{t}, Rv\cdot\mathcal{N})_{1,\Sigma}\\
&+(\xi^m+\epsilon\xi^m_{t},v\cdot\pa_t\mathcal{N})_{1,\Sigma}+[\pa_tu^m\cdot\mathcal{N},v\cdot\mathcal{N}]_\ell+[u^m\cdot\pa_t\mathcal{N},v\cdot\mathcal{N}]_\ell+[u^m\cdot\mathcal{N},R v\cdot\mathcal{N}]_\ell\\
&+[u^m\cdot\mathcal{N},v\cdot\pa_t\mathcal{N}]_\ell-\epsilon \mathfrak{b}(\pa_t(u^m\cdot\mathcal{N}),v\cdot\mathcal{N})_\ell+\int_{\Sigma_s}\beta(u^m\cdot\tau)(v\cdot\tau)\pa_tJ\\
    &=\pa_t\mathcal{F}^m(v)-\mathcal{F}^m(\pa_tv)+\mathcal{F}^m(Rv)-((u^m,Rv))\\
    &\quad-\int_\Om\frac{\mu}{2}(\mathbb{D}_{\pa_t\mathcal{A}}u^m:\mathbb{D}_{\mathcal{A}}v+\mathbb{D}_{\mathcal{A}}u^m:\mathbb{D}_{\pa_t\mathcal{A}}v+\pa_tJK\mathbb{D}_{\mathcal{A}}u^m:\mathbb{D}_{\mathcal{A}}v)J,
  \end{aligned}
  \end{equation}
  where for brevity we have written
  \begin{align*}
    \mathcal{F}^m(v)&=\int_{\Om}F^1\cdot vJ-\int_{-\ell}^{\ell}F^3\pa_1(v\cdot\mathcal{N})+F^4\cdot v-\int_{\Sigma_s}F^5(v\cdot\tau)J-[v\cdot\mathcal{N},\hat{\mathscr{W}}(\pa_t\eta)]_\ell,
  \end{align*}

  According to the Lemma \ref{lem:properties_A}, then
  \beq
  \begin{aligned}
  &(\xi+\epsilon\xi^m_{t}, Rv\cdot\mathcal{N})_{1,\Sigma}+(\xi+\epsilon\xi^m_{t},v\cdot\pa_t\mathcal{N})_{1,\Sigma}+[u\cdot\mathcal{N},R v\cdot\mathcal{N}]_\ell+[u^m\cdot\mathcal{N},v\cdot\pa_t\mathcal{N}]_\ell\\
  &=-(\xi+\epsilon\xi^m_{t}, v\cdot\pa_t\mathcal{N})_{1,\Sigma}+(\xi+\epsilon\xi^m_{t},v\cdot\pa_t\mathcal{N})_{1,\Sigma}-[u\cdot\mathcal{N}, v\cdot\pa_t\mathcal{N}]_\ell+[u^m\cdot\mathcal{N},v\cdot\pa_t\mathcal{N}]_\ell=0.
  \end{aligned}
  \eeq
  We choose the test function $v=\pa_tu^m-Ru^m$. Then we have that
  \beq
  [\pa_tu^m\cdot\mathcal{N}, (\pa_tu^m-R u^m)\cdot\mathcal{N}]_\ell=[\pa_tu^m\cdot\mathcal{N}]_\ell^2,
  \eeq
  because of the fact that $\mathcal{N}=\mathcal{N}_0-\pa_1\eta e_1$ and $R u^m\cdot\mathcal{N}=u^m\cdot\pa_t\mathcal{N}=u^m\cdot\pa_1\pa_t\eta e_1=0$.
  Similarly,
  \beq
  \pa_t(u^m\cdot\mathcal{N})=\pa_tu^m\cdot\mathcal{N}-u^m\cdot\pa_t\mathcal{N}=\pa_tu^m\cdot\mathcal{N}-u^m\cdot R^\top\mathcal{N}=(\pa_tu^m-R u^m)\cdot\mathcal{N},
  \eeq
  and hence
  \beq
  \begin{aligned}
  (u^m\cdot\mathcal{N}, (\pa_tu^m-R u^m)\cdot\mathcal{N})_{1,\Sigma}+\epsilon(\pa_t(u^m\cdot\mathcal{N}), (\pa_tu^m-R u^m)\cdot\mathcal{N})_{1,\Sigma}\\
  =(u^m\cdot\mathcal{N}, \pa_t(u^m\cdot\mathcal{N}))_{1,\Sigma}+\epsilon(\pa_t(u^m\cdot\mathcal{N}), \pa_t(u^m\cdot\mathcal{N}))_{1,\Sigma}\\
  =\frac{d}{dt}\frac12\|u^m\cdot\mathcal{N}\|_{1,\Sigma}^2+\epsilon\|\pa_t(u^m\cdot\mathcal{N})\|_{1,\Sigma}^2.
  \end{aligned}
  \eeq
  Plugging the test function $v=\pa_tu^m-Ru^m$ into \eqref{eq:pa_tu_m_1} reveals that
  \beq
  \frac{d}{dt}\frac12\|u^m\cdot\mathcal{N}\|_{1,\Sigma}^2+\epsilon\|\pa_t(u^m\cdot\mathcal{N})\|_{1,\Sigma}^2+\|\pa_tu^m\|_{{}_0\mathcal{H}^1}^2+[\pa_tu^m\cdot\mathcal{N}]_\ell^2-\epsilon \mathfrak{b}(\pa_tu^m\cdot\mathcal{N})_\ell^2=I+II+III,
  \eeq
  where
  \beq
  I=-((u^m, R(\pa_tu^m-R u^m)))+((\pa_tu^m, R u^m)),
  \eeq
  \beq
  \begin{aligned}
  II&=-\int_\Om\frac{\mu}{2}(\mathbb{D}_{\pa_t\mathcal{A}}u^m:\mathbb{D}_{\mathcal{A}}(\pa_tu^m-R u^m))J\\
  &\quad-\int_\Om\frac{\mu}{2}(\mathbb{D}_{\mathcal{A}}u^m:\mathbb{D}_{\pa_t\mathcal{A}}(\pa_tu^m-R u^m)+\pa_tJK\mathbb{D}_{\mathcal{A}}u^m:\mathbb{D}_{\mathcal{A}}(\pa_tu^m-R u^m))J\\
  &\quad-\int_{\Sigma_s}\beta(u^m\cdot\tau)((\pa_tu^m-Ru^m)\cdot\tau)\pa_tJ,
  \end{aligned}
  \eeq
  and
  \beq
  \begin{aligned}
  III&=\pa_t\mathcal{F}^m(\pa_tu^m-Ru^m)-\mathcal{F}^m(\pa_t(\pa_tu^m-R u^m))+\mathcal{F}^m(R(\pa_tu^m-R u^m))\\
  &=\int_{\Om}\left[\pa_tF^1\cdot(\pa_tu^m-R u^m)+\pa_tJKF^1\cdot(\pa_tu^m-R u^m)+F^1\cdot R(\pa_tu^m-R u^m)\right]J\\
  &\quad-\int_{-\ell}^{\ell}[\pa_tF^3\pa_1((\pa_tu^m-R u^m)\cdot\mathcal{N})+\pa_tF^4\cdot(\pa_tu^m-R u^m)+F^4\cdot R(\pa_tu^m-R u^m)]\\
  &\quad-\int_{\Sigma_s}\left[\pa_tF^5(\pa_tu^m-R u^m)+\pa_tJKF^5(\pa_tu^m-R u^m)+F^5 R(\pa_tu^m-R u^m)\right]\cdot\tau J\\
  &\quad-[\hat{\mathscr{W}}^\prime(\pa_t\eta)\pa_t^2\eta, (\pa_tu^m-R u^m)\cdot\mathcal{N}]_\ell,
  \end{aligned}
  \eeq
  where we have used the fact $\pa_t\mathcal{N}=-R^\top\mathcal{N}$ on $\Sigma$.

  We now estimate each term of $I$, $II$, $III$. For any fixed small number $\theta>0$ and $1<s<\min\{\frac{\pi}{\om},2\}$, we choose $p$ and $q$ with $\frac1p+\frac1q=\frac12$ and $2<p<\frac{2}{\delta}$, such that, according to the Sobolev inequality, the Cauchy inequality, and the trace theory of the weighted Sobolev spaces,
  \beq\label{est:I}
  \begin{aligned}
  |I|&\le\|R\|_{L^\infty}\|u^m\|_{{}_0\mathcal{H}^1(\Om)}\|\pa_tu^m\|_{{}_0\mathcal{H}^1(\Om)}+\|u^m\|_{{}_0\mathcal{H}^1(\Om)}\|\mathcal{A}\|_{L^\infty}\|\nabla R\|_{L^p}\|\pa_tu^m\|_{L^q}\\
  &\quad+\|R\|_{L^\infty}^2\|u^m\|_{{}_0\mathcal{H}^1(\Om)}^2+\|u^m\|_{{}_0\mathcal{H}^1(\Om)}\|\mathcal{A}\|_{L^\infty}\|R\|_{L^\infty}\|\nabla R\|_{L^p}\|u^m\|_{L^q}\\
  &\quad+\|R\|_{L^\infty}\|u^m\|_{{}_0\mathcal{H}^1(\Om)}\|\pa_tu^m\|_{{}_0\mathcal{H}^1(\Om)}+\|\pa_tu^m\|_{{}_0\mathcal{H}^1(\Om)}\|\mathcal{A}\|_{L^\infty}\|\nabla R\|_{L^p}\|u^m\|_{L^q}\\
  &\le\|R\|_{s}\|u^m\|_{{}_0\mathcal{H}^1(\Om)}\|\pa_tu^m\|_{{}_0\mathcal{H}^1(\Om)}+\|u^m\|_{{}_0\mathcal{H}^1(\Om)}\|\mathcal{A}\|_{s}\|\nabla R\|_{W_\delta^1}\|\pa_tu^m\|_{1}\\
  &\quad+\|R\|_{s}^2\|u^m\|_{{}_0\mathcal{H}^1(\Om)}^2+\|u^m\|_{{}_0\mathcal{H}^1(\Om)}\|\mathcal{A}\|_{s}\|R\|_{s}\|\nabla R\|_{W_\delta^1}\|u^m\|_{1}\\
  &\quad+\|R\|_{s}\|u^m\|_{{}_0\mathcal{H}^1(\Om)}\|\pa_tu^m\|_{{}_0\mathcal{H}^1(\Om)}+\|\pa_tu^m\|_{{}_0\mathcal{H}^1(\Om)}\|\mathcal{A}\|_{s}\|\nabla R\|_{W_\delta^1}\|u^m\|_{1}\\
  &\lesssim \theta\|\pa_tu^m\|_{{}_0\mathcal{H}^1(\Om)}^2+\left(1+\frac1\theta\right)C_2(\eta)\|u^m\|_{{}_0\mathcal{H}^1(\Om)}^2,
  \end{aligned}
  \eeq
  where
  \beq
  C_2(\eta)=\|R\|_s^2+\|\mathcal{A}\|_s^2\|\nabla R\|_{W_\delta^1}^2+\|\mathcal{A}\|_s\|R\|_s\|\nabla R\|_{W_\delta^1}\lesssim (\|\eta\|_{W_\delta^{5/2}}^2+\|\pa_t\eta\|_{W_\delta^{5/2}}^2)(1+\|\eta\|_{W_\delta^{5/2}}^2).
  \eeq
  Similarly,
  \beq
  \begin{aligned}
  |II|&\lesssim \|\pa_t\mathcal{A}\|_{L^\infty}\|J\|_{L^\infty}\|u^m\|_1\big(\|\pa_tu^m\|_{{}_0\mathcal{H}^1(\Om)}+\|R\|_{L^\infty}\|u^m\|_{{}_0\mathcal{H}^1(\Om)}\\
  &\quad+\|\mathcal{A}\|_{L^\infty}\|\nabla R\|_{L^p}\|u^m\|_{L^q}\big)+\|\pa_t\mathcal{A}\|_{L^\infty}\|J\|_{L^\infty}\|\pa_tu^m\|_1\|u^m\|_{{}_0\mathcal{H}^1(\Om)}\\
  &\quad+\|\mathcal{A}\|_{L^\infty}\|\nabla R\|_{L^p}\|u^m\|_{L^q}\|u^m\|_{{}_0\mathcal{H}^1(\Om)}+\|\pa_tJK\|_{L^\infty}\|\pa_tu^m\|_{{}_0\mathcal{H}^1(\Om)}\|u^m\|_{{}_0\mathcal{H}^1(\Om)}\\
  &\quad+\|\pa_tJK\|_{L^\infty}\|u^m\|_{{}_0\mathcal{H}^1(\Om)}(\|R\|_{L^\infty}\|u^m\|_{{}_0\mathcal{H}^1(\Om)}+\|J\mathcal{A}\|_{L^\infty}\|\nabla R\|_{L^p}\|u^m\|_{L^q})\\
  &\quad+\|\pa_tJ\|_{L^\infty(\Sigma_s)}\|u^m\|_1(\|\pa_tu^m\|_1+\|R\|_{L^\infty(\Sigma_s)}\|u^m\|_1)\\
  &\lesssim \theta\|\pa_tu^m\|_{{}_0\mathcal{H}^1(\Om)}^2+\left(1+\frac1\theta\right)C_3(\eta)\|u^m\|_{{}_0\mathcal{H}^1(\Om)}^2,
  \end{aligned}
  \eeq
  where
  \beq
  \begin{aligned}
  C_3(\eta)&=\big[\|\pa_t\mathcal{A}\|_{L^\infty}\|J\|_{L^\infty}(\|\pa_t\mathcal{A}\|_{L^\infty}\|J\|_{L^\infty}+\|R\|_{L^\infty})\\
  &\quad+\|\mathcal{A}\|_{L^\infty}\|\nabla R\|_{L^p}(1+\|\pa_tJ\|_{L^\infty})+\|\pa_tJK\|_{L^\infty}(\|\pa_tJK\|_{L^\infty}+\|R\|_{L^\infty})\\
  &\quad+\|\pa_tJ\|_{L^\infty(\Sigma_s)}(\|\pa_tJ\|_{L^\infty(\Sigma_s)}+\|R\|_{L^\infty(\Sigma_s)})\big]\lesssim (\|\eta\|_{W_\delta^{5/2}}^2+\|\pa_t\eta\|_{W_\delta^{5/2}}^2)\|\eta\|_{W_\delta^{5/2}}^2.
  \end{aligned}
  \eeq
For $III$, we need more refined estimates. We will separate the estimates for $III$ into several estimates.
First, by the Cauchy-Schwarz inequality, we have that
\beq\label{est:III_1}
\int_{-\ell}^{\ell}[\pa_tF^3\pa_1((\pa_tu^m-R u^m)\cdot\mathcal{N})=\int_{-\ell}^{\ell}\pa_z\mathcal{R}\pa_1\pa_t\eta\pa_1\pa_t^2\xi^m\le C(\epsilon)\|\pa_t\eta\|_1^2+\frac{\epsilon}{2}\|\pa_t^2\xi^m\|_1^2,
\eeq
here we have used the boundedness for $\pa_z\mathcal{R}$, which can be easily proved by the definition of $\mathcal{R}$.

Then we use Lemma \ref{lem:distance}, the weighted Sobolev estimates from Appendix C and D in \cite{GT2}, usual Sobolev embedding Theorem, and H\"older's inequality to derive
\beq\label{est:III_2}
\begin{aligned}
  &\int_{\Om}\left[\pa_tF^1\cdot(\pa_tu^m-R u^m)+\pa_tJKF^1\cdot(\pa_tu^m-R u^m)+F^1\cdot R(\pa_tu^m-R u^m)\right]J\\
  &\quad-\int_{-\ell}^{\ell}[\pa_tF^4\cdot(\pa_tu^m-R u^m)+F^4\cdot R(\pa_tu^m-R u^m)]\\
  &\quad-\int_{\Sigma_s}\left[\pa_tF^5(\pa_tu^m-R u^m)+\pa_tJKF^5(\pa_tu^m-R u^m)+F^5 R(\pa_tu^m-R u^m)\right]\cdot\tau J\\
  &\quad-[\hat{\mathscr{W}}^\prime(\pa_t\eta)\pa_t^2\eta, (\pa_tu^m-R u^m)\cdot\mathcal{N}]_\ell\\
  &\lesssim \theta(\|\pa_tu^m\|_{{}_0\mathcal{H}^1}^2+[\pa_tu^m\cdot\mathcal{N}]_\ell^2)+\|u^m\|_{{}_0\mathcal{H}^1}^2+(1+\|R\|_1^2)\|\pa_t(F^1-F^4-F^5)\|_{(\mathcal{H}^1)^{\ast}}^2\\
  &\quad+(\|\pa_tJ\|_1^2+\|R\|_1^2)(1+\|R\|_1^2)(\|F^1\|_{W_\delta^0}^2+\|F^4\|_{W_\delta^{1/2}}^2+\|F^5\|_{W_\delta^{1/2}}^2)+\|\pa_t\eta\|_1^2[\pa_t^2\eta]_\ell^2.
\end{aligned}
\eeq
  Thus, combining \eqref{est:I}--\eqref{est:III_2}, we have the energy structure
  \beq
  \begin{aligned}
    &\frac{d}{dt}\frac12\|\pa_t\xi^m\|_{1,\Sigma}^2+\epsilon\|\pa_t^2\xi^m\|_{1,\Sigma}^2+\|\pa_tu^m\|_{{}_0\mathcal{H}^1}^2+[\pa_tu^m\cdot\mathcal{N}]_\ell^2-\epsilon \mathfrak{b}(\pa_tu^m\cdot\mathcal{N})_\ell^2\\
    &\lesssim (1+C_2(\eta)+C_3(\eta))\|u^m\|_{{}_0\mathcal{H}^1}^2+C(\epsilon)\|\pa_t\eta\|_1^2+\|\pa_t\eta\|_1^2[\pa_t^2\eta]_\ell^2\\
    &\quad +C_5(\eta)\|\pa_t(F^1-F^4-F^5)\|_{(\mathcal{H}^1)^{\ast}}^2+C_6(\eta)(\|F^1\|_{W_\delta^0}^2+\|F^4\|_{W_\delta^{1/2}}^2+\|F^5\|_{W_\delta^{1/2}}^2),
  \end{aligned}
  \eeq
  where
  \[
  C_5(\eta)=(1+\|R\|_1^2)\lesssim (1+\|\pa_t\eta\|_{3/2}^2+\|\eta\|_{W_\delta^{5/2}}^2),
  \]
  and
  \[
  C_6(\eta)=(\|\pa_tJ\|_1^2+\|R\|_1^2)(1+\|R\|_1^2)\lesssim (\|\pa_t\eta\|_{3/2}^2+\|\eta\|_{W_\delta^{5/2}}^2)(1+\|\pa_t\eta\|_{3/2}^2+\|\eta\|_{W_\delta^{5/2}}^2).
  \]
  We then employ the Gronwall's inequality and \eqref{eq:ep} to see that
  \beq
  \begin{aligned}
  &\sup_{0\le t\le T}\|\pa_t\xi^m\|_{1,\Sigma}^2+\epsilon\|\pa_t^2\xi^m\|_{L^2H^1}^2+\|\pa_tu^m\|_{L^2H^1}^2+\|\pa_tu^m\|_{L^2H^0(\Sigma_s)}^2+\int_0^T[\pa_tu^m\cdot\mathcal{N}]_\ell^2\\
  &\lesssim C(\epsilon)T\|\pa_t\eta\|_{L^\infty H^1}^2+\mathfrak{E}_0+\mathfrak{D}(\eta)\|u^m\|_{L^\infty{}_0\mathcal{H}^1}^2+\|u^m\|_{L^2H^1}^2+\mathfrak{E}(\eta)\mathfrak{K}(\eta)\\
  &\quad+(1+\mathfrak{E}(\eta))(\|F^1\|_{L^2 W_\delta^0}^2+\|F^4\|_{L^2 W_\delta^{1/2}}^2+\|F^5\|_{L^2 W_\delta^{1/2}}^2)\\
  &\quad+(1+\mathfrak{E}(\eta))\|\pa_t(F^1-F^4-F^5)\|_{(\mathcal{H}^1_T)^{\ast}}^2.
  \end{aligned}
  \eeq
  Then applying the smallness of $\mathfrak{K}(\eta)\le \alpha\ll1$ and Lemma \ref{lem:continuous}, we have that
  \beq\label{eq:pa_tu_m}
  \begin{aligned}
  &\sup_{0\le t\le T}\|\pa_t\xi^m\|_{1,\Sigma}^2+\epsilon\|\pa_t^2\xi^m\|_{L^2H^1}^2+\|\pa_tu^m\|_{L^2H^1}^2+\|\pa_tu^m\|_{L^2H^0(\Sigma_s)}^2+\int_0^T[\pa_tu^m\cdot\mathcal{N}]_\ell^2\\
  &\lesssim C(\epsilon)T\mathfrak{E}(\eta)+\mathfrak{E}_0+\mathfrak{E}(\eta)\mathfrak{K}(\eta)+\mathfrak{E}(\eta)(\|F^1\|_{L^2 W_\delta^0}^2+\|F^4\|_{L^2 W_\delta^{1/2}}^2+\|F^5\|_{L^2 W_\delta^{1/2}}^2)\\
  &\quad+(1+\mathfrak{E}(\eta))\|\pa_t(F^1-F^4-F^5)\|_{(\mathcal{H}^1_T)^{\ast}}^2.
  \end{aligned}
  \eeq

  Step 6 -- Passing to the limit.

  We now utilize the energy estimates \eqref{eq:xi_m} and \eqref{eq:pa_tu_m} to pass to the limit $m\rightarrow\infty$. According to Proposition \ref{prop:isomorphism} and energy estimates, we have that the sequence $\{u^m\}$ and $\{\pa_tu^m\}$ are uniformly bounded both in $L^2H^1$ and $L^2H^0(\Sigma_s)$, $\{\xi^m\}$ and $\{\pa_t\xi^m\}$ are uniformly bounded in $L^\infty \mathring{H}^1$, $\{[u^m\cdot\mathcal{N}]_\ell\}$ and $\{[\pa_tu^m\cdot\mathcal{N}]_\ell\}$ are uniformly bounded in $L^2([0,T])$.  Up to the extraction of a subsequence, we then know that
  \beno
  u^m\rightharpoonup u\ \text{weakly-}\ \text{in}\ L^2 H^1\cap L^2H^0(\Sigma_s),\quad \pa_tu^m\rightharpoonup\pa_tu\ \text{weakly in}\ L^2H^1\cap L^2H^0(\Sigma_s),
  \eeno
  \beno
  \xi^m\stackrel{\ast}\rightharpoonup \xi\ \text{weakly-}\ast\ \text{in}\ L^\infty \mathring{H}^1, \quad\pa_t\xi^m\stackrel{\ast}\rightharpoonup \pa_t\xi\ \text{weakly-}\ast\ \text{in}\ L^\infty \mathring{H}^1,
  \eeno
  and
  \beno
  [u^m\cdot\mathcal{N}]_\ell\rightharpoonup[u\cdot\mathcal{N}]_\ell\ \text{weakly-}\ \text{in}\ L^2,\ [\pa_tu^m\cdot\mathcal{N}]_\ell\rightharpoonup[\pa_tu\cdot\mathcal{N}]_\ell\ \text{weakly in}\ L^2.
  \eeno
  By lower semicontinuity, the energy estimates imply that
  \beno
  \|u\|_{L^2 H^1}^2+\|u\|_{L^2 H^0(\Sigma_s)}^2+\|\pa_tu\|_{L^2H^1}^2+\|\pa_tu\|_{L^2 H^0(\Sigma_s)}^2+\|[u\cdot\mathcal{N}]_\ell\|_{L^2}^2+\|[\pa_tu\cdot\mathcal{N}]_\ell\|_{L^2}^2\\
  +\|\xi\|_{L^\infty \mathring{H}^1}^2+\|\pa_t\xi\|_{L^\infty \mathring{H}^1}^2
  \eeno
  is bounded.

  Step 7 -- Improved bounds for $\xi$ and $\pa_t\xi$.

  From the above step, we know that $\xi^m(t)\in\mathring{H}^1((-\ell,\ell))$, $\pa_t\xi^m(t)\in\mathring{H}^1((-\ell,\ell))$, and $\pa_t^2\xi^m(t)\in\mathring{H}^1((-\ell,\ell))$.  Using the test function $v\in\mathcal{V}(t)$ in \eqref{eq: Galerkin} and then appealing to Theorem 4.11 in \cite{GT2} shows that
  \beq
  \begin{aligned}
  &\|\xi^m+\epsilon\pa_t\xi^m\|_{\mathring{H}^{3/2}}^2\lesssim\|u^m\|_1^2+[u^m\cdot\mathcal{N}]_\ell^2+\|\eta\|_{W_\delta^{5/2}}^4+\|\mathcal{F}\|_{(\mathcal{H}^1)^\ast}^2+\|\pa_t\eta\|_1^2[\pa_t\eta]_\ell^2\\
  &\lesssim \|u^m\|_1^2+[u^m\cdot\mathcal{N}]_\ell^2+\|\eta\|_{W_\delta^{5/2}}^4+\|F^1\|_{W_\delta^0}^2+\|F^4\|_{W_\delta^{1/2}}^2+\|F^5\|_{W_\delta^{1/2}}^2+\|\pa_t\eta\|_1^2[\pa_t\eta]_\ell^2.
  \end{aligned}
  \eeq
  Then we may employ \eqref{initial_u_m}, \eqref{initial_pa_t_xi_m} and Sobolev theory to obtain the bound for initial data
  \beq\label{est:xi_pa_txi}
  \begin{aligned}
  \|(\xi^m+\epsilon\pa_t\xi^m)(0)\|_{\mathring{H}^{3/2}}^2&\lesssim\|u^m(0)\|_1^2+[u^m(0)\cdot\mathcal{N}(0)]_\ell^2+\|\eta_0\|_{ W_\delta^{5/2}}^4\\
  &\quad+\|\mathcal{F}(0)\|_{(\mathcal{H}^1)^\ast}^2+\|\pa_t\eta(0)\|_1^2[\pa_t\eta(0)]_\ell^2\\
  &\lesssim \mathfrak{E}_0+\|\mathcal{F}(0)\|_{(\mathcal{H}^1)^\ast}^2.
  \end{aligned}
  \eeq
  If we let $\vartheta^m=\xi^m+\epsilon\pa_t\xi^m$, we can solve this ODE as
  \beq\label{eq:xi_ode}
  \xi^m=\eta_0e^{-\frac t\epsilon}+\frac1\epsilon\int_0^te^{-\frac{t-s}{\epsilon}}\vartheta^m(s)\,\mathrm{d}s.
  \eeq
  Now we estimate $\xi^m$ from \eqref{eq:xi_ode}. Applying Cauchy's inequality,
  \beq\label{est:xi_m_1}
  \int_0^t\frac1\epsilon e^{-\frac{t-s}{\epsilon}}\|\vartheta^m(s)\|_{3/2}\,\mathrm{d}s\le \left(\int_0^t\frac1\epsilon e^{-\frac{t-s}{\epsilon}}\|\vartheta^m(s)\|_{3/2}^2\,\mathrm{d}s\right)^{1/2}\left(\int_0^t\frac1\epsilon e^{-\frac{t-s}{\epsilon}}\,\mathrm{d}s\right)^{1/2}.
\eeq
Since
\beq
\left(\int_0^t\frac1\epsilon e^{-\frac{t-s}{\epsilon}}\,\mathrm{d}s\right)^{1/2}=(1-e^{-\frac t\epsilon})^{1/2}\le1,
\eeq
then integrating in \eqref{est:xi_m_1} and employing Fubini's theorem imply that
\beq
\begin{aligned}
\|\int_0^t\frac1\epsilon e^{-\frac{t-s}{\epsilon}}\|\vartheta^m(s)\|_{3/2}\|_{L^2}^2&\le \int_0^T\int_0^t\frac1\epsilon e^{-\frac{t-s}{\epsilon}}\|\vartheta(s)\|_{3/2}^2\,\mathrm{d}s\,\mathrm{d}t\\
&=\int_0^T\frac1\epsilon e^{-\frac{s}{\epsilon}}\left(\int_s^T\|\vartheta(t-s)\|_{3/2}^2\,\mathrm{d}t\right)\,\mathrm{d}s\\
&\le\int_0^T \frac1\epsilon e^{-\frac{s}{\epsilon}}\int_0^T\|\vartheta(t)\|_{3/2}^2\\
&\le \|\vartheta\|_{L^2H^{3/2}}^2.
\end{aligned}
\eeq
  Thus, \eqref{eq:xi_m} reveals that
  \beq\label{est:xi_enhance}
  \begin{aligned}
  &\|\xi^m\|_{L^2\mathring{H}^{3/2}}^2\le C\|\eta_0\|_{W_\delta^{5/2}}^2+\|\int_0^t\frac1\epsilon e^{-\frac{t-s}{\epsilon}}\|\vartheta^m(s)\|_{3/2}\|_{L^2}^2\\
  &\lesssim (1+\|\eta\|_{L^\infty W_\delta^{5/2}}^2)(\|F^1\|_{L^2W_\delta^0}^2+\|F^4\|_{L^2W_\delta^{1/2}}^2+\|F^5\|_{L^2W_\delta^{1/2}}^2)+C(\epsilon)T\|\eta\|_{L^\infty H^1}^2\|\eta\|_{L^\infty W_\delta^{5/2}}^2\\
      &\quad+\|\eta_0\|_{W_\delta^{5/2}}^2+\|\eta\|_{L^\infty W_\delta^{5/2}}^2\|\eta\|_{L^2 W_\delta^{5/2}}^2+\|\pa_t\eta\|_{L^\infty H^1}^2\|[\pa_t\eta]_\ell\|_{L^2(0,T)}^2.
  \end{aligned}
  \eeq
  Similarly, according to \eqref{eq:pa_tu_m_1}, we know that
  \beq\label{est:pa_txi_enhance}
  \begin{aligned}
  &\|\pa_t\xi^m+\epsilon\pa_t^2\xi^m\|_{L^2\mathring{H}^{3/2}}^2\\
  &\lesssim \|\pa_tu^m\|_{L^2H^1}^2+\int_0^T[\pa_tu^m\cdot\mathcal{N}]_\ell^2+\mathfrak{E}_0+\mathfrak{E}(\eta)\mathfrak{K}(\eta)+(1+\mathfrak{E}(\eta))(\|F^1\|_{L^2 W_\delta^0}^2+\|F^4\|_{L^2 W_\delta^{1/2}}^2+\|F^5\|_{L^2 W_\delta^{1/2}}^2)\\
  &\quad+(1+\mathfrak{E}(\eta))\|\pa_t(F^1-F^4-F^5)\|_{(\mathcal{H}^1_T)^{\ast}}^2\\
  &\lesssim C(\epsilon)T\mathfrak{E}(\eta)+\mathfrak{E}_0+\mathfrak{E}(\eta)\mathfrak{K}(\eta)+(1+\mathfrak{E}(\eta))(\|F^1\|_{L^2 W_\delta^0}^2+\|F^4\|_{L^2 W_\delta^{1/2}}^2+\|F^5\|_{L^2 W_\delta^{1/2}}^2)\\
  &\quad+(1+\mathfrak{E}(\eta))\|\pa_t(F^1-F^4-F^5)\|_{(\mathcal{H}^1_T)^{\ast}}^2.
  \end{aligned}
  \eeq
  If we denote $\pa_t\vartheta^m=\pa_t\xi^m+\epsilon\pa_t^2\xi^m$, and the extension $\pa_t\bar{\vartheta}^m=\pa_t\bar{\xi}^m+\epsilon\pa_t^2\bar{\xi}^m$, then by a standard computation for energy formulation, we may get
  \beq\label{extension_restriction}
  \epsilon\frac{d}{dt}\|\pa_t\bar{\xi}^m\|_{H^2}^2+\|\pa_t\bar{\xi}^m\|_{H^2}^2 \le \|\pa_t\bar{\vartheta}^m\|_{H^2}^2,
  \eeq
  Then by the trace theory and \eqref{est:xi_pa_txi}, we  derive that
  \beq
  \begin{aligned}
    &\|\pa_t\xi^m\|_{L^2\mathring{H}^{3/2}}^2\lesssim \epsilon^2\|\pa_t\xi^m(0)\|_{3/2}^2+\|\pa_t\vartheta^m\|_{L^2\mathring{H}^{3/2}}^2\\
    &\lesssim \|\xi^m(0)\|_{3/2}^2+\|\xi^m(0)+\epsilon\pa_t\xi^m(0)\|_{3/2}^2+\|\pa_t\vartheta^m\|_{L^2\mathring{H}^{3/2}}^2\\
    &\lesssim C(\epsilon)T\mathfrak{E}(\eta)+\mathfrak{E}_0+\|\mathcal{F}(0)\|_{(\mathcal{H}^1)^\ast}^2+\mathfrak{E}(\eta)\mathfrak{K}(\eta)\\
  &\quad+(1+\mathfrak{E}(\eta))(\|F^1\|_{L^2 W_\delta^0}^2+\|F^4\|_{L^2 W_\delta^{1/2}}^2+\|F^5\|_{L^2 W_\delta^{1/2}}^2)+(1+\mathfrak{E}(\eta))\|\pa_t(F^1-F^4-F^5)\|_{(\mathcal{H}^1_T)^{\ast}}^2.
  \end{aligned}
  \eeq
  Then, up to an extraction of subsequence, we know that
  \beq
  \begin{aligned}
  \xi^m\rightharpoonup\xi\ \text{weakly}\ \text{in}\ L^2 H^{3/2},\
  \pa_t\xi^m\rightharpoonup\pa_t\xi\ \text{weakly in}\ L^2H^{3/2}.
  \end{aligned}
  \eeq
  By lower semicontinuity we then know that the quantity
  \[
  \|\xi\|_{L^2 H^{3/2}}+\|\pa_t\xi\|_{L^2H^{3/2}}
  \]
  is bounded.

  Step 8 -- The strong solution

  Due to the convergence, we may pass to the limit in \eqref{eq: Galerkin} for almost every $t\in [0, T]$.
  \beq\label{eq:weak_limit}
  \begin{aligned}
  &((u,v))+(\xi+\epsilon\pa_t\xi, v\cdot\mathcal{N})_{1,\Sigma}+[u\cdot\mathcal{N},v\cdot\mathcal{N}]_\ell-\epsilon \mathfrak{b}(\pa_t\xi,v\cdot\mathcal{N})_\ell\\
  &=\int_\Om F^1\cdot vJ-\int_{-\ell}^{\ell}F^3\pa_1(v\cdot\mathcal{N})+F^4\cdot v-\int_{\Sigma_s}F^5(v\cdot\tau)J-[\hat{\mathscr{W}}(\pa_t\eta),v\cdot\mathcal{N}]_\ell.
  \end{aligned}
  \eeq
  We now introduce the pressure. Define the functional $\Lam_t\in(\mathcal{W}(t))^\ast$ so that $\Lam_t(v)$ equals the difference between the left and right sides of \eqref{eq:weak_limit} with $v\in\mathcal{W}(t)$. Then $\Lam_t(v)=0$ for all $v\in\mathcal{V}(t)$. So, by Theorem 4.6 in \cite{GT2}, there exists a unique $p(t)\in\mathring{H}^0(t)$ such that $(p(t),\dive_{\mathcal{A}}v)_{\mathcal{H}^0}=\Lam_t(v)$ for all $v\in\mathcal{W}(t)$. This is equivalent to
  \beq\label{eq:weak_pressure}
  \begin{aligned}
  &((u,v))+(\xi+\epsilon\pa_t\xi, v\cdot\mathcal{N})_{1,\Sigma}-(p,\dive_{\mathcal{A}}v)_{\mathcal{H}^0}+[u\cdot\mathcal{N},v\cdot\mathcal{N}]_\ell-\epsilon \mathfrak{b}(\pa_t\xi,v\cdot\mathcal{N})_\ell\\
  &=\int_\Om F^1\cdot vJ
  -\int_{-\ell}^{\ell}F^3\pa_1(v\cdot\mathcal{N})+F^4\cdot v-\int_{\Sigma_s}F^5(v\cdot\tau)J-[\hat{\mathscr{W}}(\pa_t\eta),v\cdot\mathcal{N}]_\ell.
  \end{aligned}
  \eeq
  Moreover,
  \beq
  \|p\|_0^2\lesssim \|u\|_1^2+\|F^1\|_{W_\delta^0}^2+\|F^5\|_{W_\delta^{1/2}}^2.
  \eeq
  On the other hand, we pass the limit in \eqref{eq:pa_tu_m_1}
  for a.e. $t\in[0,T]$ to see that $(u(t), p(t), \xi(t))$ is the unique weak solution to the elliptic problem $(5.58)$ in \cite{GT2}. Since $\pa_1F^3(t)\in W_\delta^{1/2}$, and also according to the elliptic theory of \cite{GT2}, this elliptic problem admits a unique strong solution with
  \beq\label{elliptic_1}
  \begin{aligned}
  \|u(t)\|_{W_\delta^2}^2+\|p(t)\|_{\mathring{W}_\delta^1}^2+\|\xi(t)+\epsilon\pa_t\xi(t)\|_{W_\delta^{5/2}}^2
  \lesssim \|F^1\|_{L^2 W_\delta^0}^2+\|F^4\|_{L^2 W_\delta^{1/2}}^2+\|F^5\|_{L^2 W_\delta^{1/2}}^2\\
  +\|\pa_t\xi(t)\|_{W_\delta^{3/2}}^2
  +\|\pa_1F^3(t)\|_{W_\delta^{1/2}}^2+[u(t)\cdot\mathcal{N}]_\ell^2+[\hat{\mathscr{W}}(\pa_t\eta)]_\ell^2
  +[\sigma F^3(t)]_\ell^2.
  \end{aligned}
  \eeq
  Then using the extension and restriction of weighted Sobolev spaces theory, similar to \eqref{extension_restriction}, we may derive that
  \beq\label{elliptic_2}
  \|\xi\|_{L^2 W_\delta^{5/2}}^2\lesssim \|\eta_0\|_{L^2 W_\delta^{5/2}}^2+\|\xi(t)+\epsilon\pa_t\xi(t)\|_{L^2W_\delta^{5/2}}^2.
  \eeq
  Integrating temporally from $0$ to $T$ for \eqref{elliptic_1}, we employ \eqref{elliptic_2} to derive that
  \beq
  \begin{aligned}
  &\|u\|_{L^2 W_\delta^2}^2+\|p\|_{L^2\mathring{W}_\delta^1}^2+\|\xi\|_{L^2 W_\delta^{5/2}}^2\\
  &\lesssim \|\eta_0\|_{W_\delta^{5/2}}^2+\|\pa_t\xi\|_{L^2 W_\delta^{3/2}}^2+\|\pa_1F^3\|_{L^2 W_\delta^{1/2}}^2+\int_0^T\left([\kappa u\cdot\mathcal{N}]_\ell^2+[\hat{\mathscr{W}}(\pa_t\eta)]_\ell^2+[\sigma F^3]_\ell^2\right)\\
  &\quad+\|F^1\|_{L^2 W_\delta^0}^2+\|F^4\|_{L^2 W_\delta^{1/2}}^2+\|F^5\|_{L^2 W_\delta^{1/2}}^2\\
  &\lesssim C(\epsilon)T\mathfrak{E}(\eta)+\mathfrak{E}_0+\|\mathcal{F}(0)\|_{(\mathcal{H}^1)^\ast}^2+\mathfrak{E}(\eta)\mathfrak{K}(\eta)+(1+\mathfrak{E}(\eta))(\|F^1\|_{L^2 W_\delta^0}^2+\|F^4\|_{L^2 W_\delta^{1/2}}^2+\|F^5\|_{L^2 W_\delta^{1/2}}^2)\\
  &\quad+(1+\mathfrak{E}(\eta))\|\pa_t(F^1-F^4-F^5)\|_{(\mathcal{H}^1_T)^{\ast}}^2.
  \end{aligned}
  \eeq

   Step 9 -- The weak solution for $D_tu$ and $\pa_tp$.

   Now we seek to use \eqref{eq:pa_tu_m_1} to determine the PDE satisfied by $D_tu$ and $\pa_tp$. We may pass to the limit $m\rightarrow \infty$, and use \eqref{eq:weak_pressure} with the test function $v$ replaced by $Rv$ to derive that
   \beq\label{eq:d_tu_1}
   \begin{aligned}
   &((\pa_tu,v))+(\pa_t\xi+\epsilon\pa_t^2\xi,v\cdot\mathcal{N})_{1,\Sigma}+[\pa_tu\cdot\mathcal{N},v\cdot\mathcal{N}]_\ell-\epsilon \mathfrak{b}(\pa_t^2\xi,v\cdot\mathcal{N})_\ell\\
   &=(\xi+\epsilon\pa_t\xi,R v\cdot\mathcal{N})_{1,\Sigma}-(p,\dive_{\mathcal{A}}(Rv))_{\mathcal{H}^0}+\int_{\Om}\left[\pa_tF^1\cdot v+\pa_tJKF^1\cdot v\right]J\\
  &\quad-\int_{-\ell}^{\ell}[\pa_tF^3\pa_1(v\cdot\mathcal{N})+F^3\pa_1(v\cdot\pa_t\mathcal{N})+\pa_tF^4\cdot v]-\int_{\Sigma_s}\left[\pa_tF^5 v+\pa_tJKF^5v\right]\cdot\tau J\\\
   &\quad-\int_\Om\frac{\mu}{2}(\mathbb{D}_{\pa_t\mathcal{A}}u:\mathbb{D}_{\mathcal{A}}v+\mathbb{D}_{\mathcal{A}}u:\mathbb{D}_{\pa_t\mathcal{A}}v+\pa_tJK\mathbb{D}_{\mathcal{A}}u:\mathbb{D}_{\mathcal{A}}v)J-\int_{\Sigma_s}\beta(u\cdot\tau)(v\cdot\tau)\pa_tJ\\
   &\quad-[\hat{\mathscr{W}}^\prime(\pa_t\eta)\pa_t^2\eta,v\cdot\mathcal{N}]_\ell.
   \end{aligned}
   \eeq
   According to the Lemma \ref{lem:properties_A}, we know that $-R^\top\mathcal{N}=\pa_t\mathcal{N}$ on $\Sigma$. Then integrating by parts, we have that
   \beq\label{eq:p_1}
   -(p,\dive_{\mathcal{A}}(Rv))_{\mathcal{H}^0}=(R^\top\nabla_{\mathcal{A}}p,v)_{\mathcal{H}^0}+\left<p\pa_t\mathcal{N},v\right>_{-1/2},
   \eeq
   where we have used the Proposition \ref{prop:solid_boundary} to cancel the term on boundary of solid wall.
   Then the definition of $R$ and integration by parts yields that
   \beq\label{eq:d_tu_2}
   \begin{aligned}
     &-\int_\Om\frac{\mu}{2}(\mathbb{D}_{\pa_t\mathcal{A}}u:\mathbb{D}_{\mathcal{A}}v+\mathbb{D}_{\mathcal{A}}u:\mathbb{D}_{\pa_t\mathcal{A}}v+\pa_tJK\mathbb{D}_{\mathcal{A}}u:\mathbb{D}_{\mathcal{A}}v)J\\
     &=-\int_\Om\mu(\mathbb{D}_{\pa_t\mathcal{A}}u-R\mathbb{D}_{\mathcal{A}}u):\nabla_{\mathcal{A}}vJ\\
     &=\left(\dive_{\mathcal{A}}(\mathbb{D}_{\pa_t\mathcal{A}}u-R\mathbb{D}_{\mathcal{A}}u),v\right)_{\mathcal{H}^0}-\left<\mathbb{D}_{\pa_t\mathcal{A}}u\mathcal{N}+\mathbb{D}_{\mathcal{A}}u\pa_t\mathcal{N},v\right>_{-1/2}.
   \end{aligned}
   \eeq
   Similarly, we have that
   \beq\label{eq:d_tu_3}
   \begin{aligned}
   -\int_{\Sigma_s}\beta(u\cdot\tau)(v\cdot\tau)\pa_tJ&=\int_{\Sigma_s}\mu\mathbb{D}_{\mathcal{A}}u\nu\cdot v\pa_tJ\\
   &=\int_{\Sigma_s}\mu\mathbb{D}_{\mathcal{A}}u\nu\cdot v\pa_tJKJ+\mu R\mathbb{D}_{\mathcal{A}}u\nu\cdot vJ-\mu R\mathbb{D}_{\mathcal{A}}u\nu\cdot vJ\\
   &=\int_{\Sigma_s}\mu\mathbb{D}_{\pa_t\mathcal{A}}u\nu\cdot v J-\mu R\mathbb{D}_{\mathcal{A}}u\nu\cdot vJ\\
   &=\int_{\Sigma_s}\mu\mathbb{D}_{\pa_t\mathcal{A}}u\nu\cdot v J+\beta Ru\cdot v J.
   \end{aligned}
   \eeq
  Combining the above equalities \eqref{eq:d_tu_1}--\eqref{eq:d_tu_3},
   \beq
   \begin{aligned}
     &((\pa_tu,v))+(\pa_t\xi+\epsilon\pa_t^2\xi,v\cdot\mathcal{N})_{1,\Sigma}+[\pa_tu\cdot\mathcal{N},v\cdot\mathcal{N}]_\ell-\epsilon \mathfrak{b}(\pa_t^2\xi,v\cdot\mathcal{N})_\ell\\
     &=\int_\Om\left[\dive_{\mathcal{A}}(\mathbb{D}_{\pa_t\mathcal{A}}u-R\mathbb{D}_{\mathcal{A}}u)+R^\top\nabla_{\mathcal{A}}p\right]\cdot v J+\int_\Om(\pa_tF^1+\pa_tJKF^1)\cdot vJ\\
     &\quad+\int_{\Sigma_s}\left(\mu\mathbb{D}_{\pa_t\mathcal{A}}u\nu+\beta Ru\right)\cdot\tau(\tau\cdot v)J+\int_{\Sigma_s}(\pa_tF^5+\pa_tJKF^5) (v\cdot\tau)J\\
     &\quad-\int_{-\ell}^\ell\pa_t F^3\pa_1(v\cdot\mathcal{N})+F^3\pa_1(v\cdot\pa_t\mathcal{N})+F^4\cdot v+(\mathbb{D}_{\pa_t\mathcal{A}}u\mathcal{N}+\mathbb{D}_{\mathcal{A}}u\pa_t\mathcal{N})\cdot v\\
     &\quad+\int_{-\ell}^\ell\left(-p+g(\xi+\epsilon\pa_t\xi)+\pa_1\left(F^3+\frac{\xi+\epsilon\pa_t\xi}{(1+|\pa_1\zeta_0|^2)^{3/2}}\right)\right)\pa_t\mathcal{N}\cdot v\\
     &\quad-[\hat{\mathscr{W}}^\prime(\pa_t\eta)\pa_t^2\eta,v\cdot\mathcal{N}]_\ell,
   \end{aligned}
   \eeq
   where we have used the integration by parts for the term $(\xi+\epsilon\pa_t\xi,R v\cdot\mathcal{N})_{1,\Sigma}$ and the fact that $v\cdot\pa_t\mathcal{N}=0$ at $x_1=\pm\ell$.
   Then there exists a unique $\pa_tq\in\mathring{H}^0$, such that
   \beq
   \begin{aligned}
     &((\pa_tu,v))-(\pa_tp,\dive_{\mathcal{A}}v)_0+(\pa_t\xi+\epsilon\pa_t^2\xi,v\cdot\mathcal{N})_{1,\Sigma}+[\pa_tu\cdot\mathcal{N},v\cdot\mathcal{N}]_\ell-\epsilon \mathfrak{b}(\pa_t^2\xi,v\cdot\mathcal{N})_\ell\\
     &=\int_\Om\left[\dive_{\mathcal{A}}(\mathbb{D}_{\pa_t\mathcal{A}}u-R\mathbb{D}_{\mathcal{A}}u)+R^\top\nabla_{\mathcal{A}}p\right]\cdot v J+\int_\Om(\pa_tF^1+\pa_tJKF^1)\cdot vJ\\
     &\quad+\int_{\Sigma_s}\left(\mu\mathbb{D}_{\pa_t\mathcal{A}}u\nu+\beta Ru\right)\cdot\tau(\tau\cdot v)J+\int_{\Sigma_s}(\pa_tF^5+\pa_tJKF^5) (v\cdot\tau)J\\
     &\quad-\int_{-\ell}^\ell\pa_t F^3\pa_1(v\cdot\mathcal{N})+F^3\pa_1(v\cdot\pa_t\mathcal{N})+F^4\cdot v+(\mathbb{D}_{\pa_t\mathcal{A}}u\mathcal{N}+\mathbb{D}_{\mathcal{A}}u\pa_t\mathcal{N})\cdot v\\
     &\quad+\int_{-\ell}^\ell\left(-p+g(\xi+\epsilon\pa_t\xi)+\pa_1\left(F^3+\frac{\xi+\epsilon\pa_t\xi}{(1+|\pa_1\zeta_0|^2)^{3/2}}\right)\right)\pa_t\mathcal{N}\cdot v\\
     &\quad-[\hat{\mathscr{W}}^\prime(\pa_t\eta)\pa_t^2\eta,v\cdot\mathcal{N}]_\ell,
   \end{aligned}
   \eeq
   and
   \beq
   \begin{aligned}
   \|\pa_tp\|_0^2&\lesssim \|\pa_tu\|_1^2+[\pa_tu\cdot\mathcal{N}]_\ell^2+\|\eta\|_{3/2}^2\|\pa_t\eta\|_{3/2}^2(\|u\|_{W_\delta^2}^2+\|p\|_{\mathring{W}_\delta^1}^2+\|\xi+\epsilon\pa_t\xi\|_{W_\delta^{5/2}}^2\\
   &\quad+\|\eta\|_{W_\delta^{5/2}}^2+\|\eta\|_{3/2}^2+1+\|F^1\|_{W_\delta^0}^2+\|F^4\|_{ W_\delta^{1/2}}^2+\|F^5\|_{ W_\delta^{1/2}}^2\\
   &\quad+\|\pa_t(F^1-F^4-F^5)\|_{(\mathcal{H}^1)^{\ast}}^2)+[\hat{\mathscr{W}}^\prime(\pa_t\eta)\pa_t^2\eta]_\ell^2.
   \end{aligned}
   \eeq
   Thus integrating temporally from $0$ to $T$ reveals that
   \beq
   \begin{aligned}
   \|\pa_tp\|_{L^2H^0}^2&\lesssim C(\epsilon)T\mathfrak{E}(\eta)+\mathfrak{E}_0+\|\mathcal{F}(0)\|_{(\mathcal{H}^1)^\ast}^2+\mathfrak{E}(\eta)\mathfrak{K}(\eta)+(1+\mathfrak{E}(\eta))(\|F^1\|_{L^2 W_\delta^0}^2+\|F^4\|_{L^2 W_\delta^{1/2}}^2\\
   &\quad+\|F^5\|_{L^2 W_\delta^{1/2}}^2)
+(1+\mathfrak{E}(\eta))\|\pa_t(F^1-F^4-F^5)\|_{(\mathcal{H}^1_T)^{\ast}}^2.
  \end{aligned}
   \eeq
  \end{proof}

\subsection{Higher regularity}
In order to state our higher regularity results for the problem \eqref{eq:modified_linear}, we must be able to define the forcing terms and initial data for the problem that results from temporally differentiating \eqref{eq:modified_linear} one time. First, we define some mappings. Given $F^3$, $v$, $q$, $\tilde{\xi}$, we define the vector fields $\mathfrak{G}^1$ in $\Om$, $\mathfrak{G}^3$ on $\Sigma$ and $\mathfrak{G}^4$ on $\Sigma_s$ by
\beq
\begin{aligned}
    \mathfrak{G}^1(v,q)&=R^\top\nabla_{\mathcal{A}}q+\dive_{\mathcal{A}}\left(\mathbb{D}_{\mathcal{A}}(Rv)+\mathbb{D}_{\pa_t\mathcal{A}}v-R\mathbb{D}_{\mathcal{A}}v\right),\\
    \mathfrak{G}^4(v,q,\tilde{\xi})&=\mu\mathbb{D}_{\mathcal{A}}(Rv)\mathcal{N}-(qI-\mu\mathbb{D}_{\mathcal{A}}v)\pa_t\mathcal{N}+\mu\mathbb{D}_{\pa_t\mathcal{A}}v)\mathcal{N}+\mathcal{L}(\tilde{\xi})\pa_t\mathcal{N}-\sigma\pa_1F^3\pa_t\mathcal{N},\\
    \mathfrak{G}^5(v)&=(\mu\mathbb{D}_{\mathcal{A}}(Rv)\nu+\mu\mathbb{D}_{\pa_t\mathcal{A}}v\nu+\beta Rv)\cdot\tau.
\end{aligned}
\eeq
These mappings allow us to define the forcing terms as follows. We write
$F^{1,0}=F^1$, $F^{4,0}=F^4$ and $F^{5,0}=F^5$. Then we write
\beq\label{def:F_11}
F^{1,1}:=D_tF^1+G^1, \quad F^{4,1}:=\pa_tF^4+G^4,\quad F^{5,1}:=\pa_tF^5+G^5.
\eeq
When $F^3$, $u$, $p$ and $\xi$ are sufficiently regular for the following to make sense, we define  the vectors
\beq
  F^{1,2}:=\mathfrak{G}^1(D_tu,\pa_tp)+D_tG^1,
  F^{4,2}:=\mathfrak{G}^4(D_tu,\pa_tp,\pa_t\xi)+\pa_tG^4,
  F^{5,2}:=\mathfrak{G}^5(D_tu)+\pa_tG^5.
\eeq

In order to deduce the higher regularity, we need to control the forcing terms $F^{i,j}$. But for the purpose of solving the nonlinear problem \eqref{eq:modified_geometric perturbation}, it's necessary to assume that $F^{i,0}=0$, $j=1, 4, 5$. Before that, we need the following useful lemma.
\begin{lemma}\label{lem:continuous_2}
Suppose that the right-hand side of the following estimates are finite. Then we have the inclusions $u\in C^0([0,T];W_{\delta}^2(\Om))$, $p\in C^0([0,T];\mathring{W}_{\delta}^1(\Om))$, $\xi\in C^0([0,T];W_{\delta}^{5/2}((-\ell,\ell)))$, as well as the estimates
  \beq\label{est:continuous_u}
  \|u\|_{L^\infty W_{\delta}^2}^2\lesssim \|\pa_t\eta(0)\|_{3/2}^2+\|u\|_{L^2 W_{\delta}^2}^2+\|\pa_tu\|_{L^2 W_{\delta}^2}^2,
  \eeq
  \beq\label{est:continuous_p}
  \|p\|_{L^\infty \mathring{W}_{\delta}^1}^2\lesssim \|\pa_t\eta(0)\|_{3/2}^2+\|p\|_{L^2 \mathring{W}_{\delta}^1}^2+\|\pa_tp\|_{L^2 \mathring{W}_{\delta}^1}^2,
  \eeq
  \beq\label{est:continuous_xi}
  \|\xi\|_{L^\infty W_{\delta}^{5/2}}^2\lesssim \|\eta_0\|_{W_{\delta}^{5/2}}^2+\|\xi\|_{L^2 W_{\delta}^{5/2}}^2+\|\pa_t\xi\|_{L^2 W_{\delta}^{5/2}}^2.
  \eeq
  \beq\label{est:continuous_xi_purterbation}
  \|\xi+\epsilon\pa_t\xi\|_{L^\infty W_{\delta}^{5/2}}^2\lesssim \|\eta_0\|_{W_{\delta}^{5/2}}^2+\|\xi+\epsilon\pa_t\xi\|_{L^2 W_{\delta}^{5/2}}^2+\|\pa_t\xi+\epsilon\pa_t^2\xi\|_{L^2 W_{\delta}^{5/2}}^2.
  \eeq
\end{lemma}
\begin{proof}
  First, \eqref{est:continuous_u} and \eqref{est:continuous_p} are obtained by a computation similar to that of Lemma \ref{lem:continuous}, combined with estimates for the initial data for $u_0^\epsilon$, $p_0^\epsilon$ in Section \ref{subsection_initial}.
  By Theorem 4.6 in \cite{GT2} and the Stokes equation, we have that
 \eqref{est:continuous_xi} can be obtained  after employing the extension theory on weighted Sobolev spaces, and then using the restriction theory on Sobolev spaces. From the third equation of \eqref{eq:modified_linear}, we know that
  \[
  \|(\xi+\epsilon\pa_t\xi)(0)\|_{W_{\delta}^{5/2}}^2\lesssim \|\eta_0\|_{W_{\delta}^{5/2}}^2+(1+\|\eta_0\|_{W_{\delta}^{5/2}}^2)\|u_0\|_{W_{\delta}^2}^2\lesssim\|\eta_0\|_{W_{\delta}^{5/2}}^2+\|\pa_t\eta(0)\|_{3/2}^2 ,
  \]
  which together with \eqref{est:continuous_xi} imply \eqref{est:continuous_xi_purterbation}.
\end{proof}

Now, we need to estimate the forcing terms of $F^{i,j}$.
\begin{lemma}\label{lemma:est_force}
  The following estimates hold whenever the right hand side are finite.
  \beq\label{est:force_11}
  \begin{aligned}
  \|F^{1,1}\|_{L^2W_\delta^0}^2\lesssim \mathfrak{K}(\eta)( \|u\|_{L^2H^1}^2+\|\pa_tu\|_{L^2H^1}^2+\|u\|_{L^2W_\delta^2}^2+\|\pa_tu\|_{L^2W_\delta^2}^2+\|p\|_{L^2\mathring{W}_\delta^1}^2+\|\pa_tp\|_{L^2\mathring{W}_\delta^1}^2),
  \end{aligned}
  \eeq
  \beq\label{est:force_41}
  \begin{aligned}
  \|F^{4,1}\|_{L^2W_\delta^{1/2}}^2\lesssim \mathfrak{K}(\eta)\Big((1+\|\eta_0\|_{W_{\delta}^{5/2}}^2)(1+\|u_0^\epsilon\|_{W_{\delta}^2}^2)+\|p\|_{L^2W_\delta^1}^2+\|u\|_{L^2W_\delta^2}^2+\|\xi+\epsilon\pa_t\xi\|_{L^2W_\delta^{5/2}}^2\\
  +\|\pa_tp\|_{L^2 W_\delta^1}^2
  +\|\pa_tu\|_{L^2 W_\delta^2}^2+\|\pa_t\xi+\epsilon\pa_t^2\xi\|_{L^2 W_\delta^{5/2}}^2\Big),
  \end{aligned}
  \eeq
  \beq\label{est:force_51}
    \|F^{5,1}\|_{L^2W_\delta^{1/2}}^2
    \lesssim \mathfrak{K}(\eta)\left(\|u\|_{L^2W_\delta^2}^2+\|\pa_tu\|_{L^2 W_\delta^2}^2+\|u\|_{L^2H^1}^2\right).
  \eeq
\end{lemma}
\begin{proof}
  The estimates follow from simple but lengthy computations, invoking the arguments of Appendix C and Appendix D in \cite{GT2}. For this reason, we only give a sketch of proving these estimates.

  According to the definition of $F^{1,1}$, $F^{4,1}$ and $F^{5,1}$ in \eqref{def:F_11}, we use Leibniz rule to rewrite $F^{i,1}$ as a sum of products for two terms. One term is a product of various derivatives of $\bar{\eta}$, and the other is linear for derivatives of $u$, $p$ and $\xi$. Then for a.e. $t\in[0,T]$, we estimate these resulting products using the weighted Sobolev theory in Appendix C and Appendix D in \cite{GT2}, the usual Sobolev embedding theorems and Lemma \ref{lem:continuous}. Then the resulting inequalities after integrating over $[0,T]$ reveals
  \beq
  \begin{aligned}
  \|F^{1,1}\|_{L^2W_\delta^0}^2&\lesssim P(\mathfrak{E}(\eta))\mathfrak{D}(\eta)(\|u\|_{L^2H^1}^2+\|\pa_tu\|_{L^2H^1}^2+\|u\|_{L^2W_\delta^2}^2+\|\pa_tu\|_{L^2W_\delta^2}^2\\
  &\quad+\|p\|_{L^2\mathring{W}_\delta^1}^2+\|\pa_tp\|_{L^2\mathring{W}_\delta^1}^2),
  \end{aligned}
  \eeq
  where $P(\cdot)$ is a polynomial. Since $\mathfrak{K}(\eta)\le1$, we know that $P(\mathfrak{E}(\eta))\mathfrak{D}(\eta)\lesssim \mathfrak{K}(\eta)$. Thus we have the bounds for \eqref{est:force_11}. Similarly, we have the bounds for \eqref{est:force_41} and \eqref{est:force_51}, and \eqref{est:force_41} also needs \eqref{est:continuous_xi_purterbation}.
\end{proof}

\begin{lemma}\label{lem:est_force_dual}
It holds that
\beq
\begin{aligned}
  \|F^{1,1}-F^{4,1}-F^{5,1}\|_{L^2(\mathcal{H}^1)^\ast}^2\lesssim \mathfrak{E}(\eta)(\|p\|_{L^2\mathring{W}_\delta^1}^2+\|u\|_{L^2 W_\delta^2}^2+\|\xi+\epsilon\pa_t\xi\|_{L^2 W_\delta^{5/2}}^2),
\end{aligned}
\eeq
and
  \beq
  \begin{aligned}
  \|\pa_t(F^{1,1}-F^{4,1}-F^{5,1})\|_{(\mathcal{H}^1_T)^\ast}^2\lesssim \mathfrak{K}(\eta)( 1+ \|u\|_{L^2H^1}^2+\|\pa_tu\|_{L^2H^1}^2+\|u\|_{L^2W_\delta^2}^2+\|\pa_tu\|_{L^2W_\delta^2}^2\\
  \quad+\|p\|_{L^2\mathring{W}_\delta^1}^2+\|\pa_tp\|_{L^2\mathring{W}_\delta^1}^2+\|\xi+\epsilon\pa_t\xi\|_{L^2W_\delta^{5/2}}^2+\|\pa_t\xi+\epsilon\pa_t^2\xi\|_{L^2 W_\delta^{5/2}}^2).
  \end{aligned}
  \eeq
  Then $F^{1,1}-F^{4,1}-F^{5,1}\in C([0,T];(\mathcal{H}^1)^\ast)$. Moreover,
  \beq
  \|(F^{1,1}-F^{4,1}-F^{5,1})(0)\|_{(\mathcal{H}^1)^\ast}^2\lesssim \mathfrak{E}_0.
  \eeq
\end{lemma}
\begin{proof}
Since the proof of the first two inequalities are similar, we only give the proof second inequality.
  From the notation in Remark \ref{def:dual}, we have that
  \beq\label{est:dual_1}
\left<\pa_t(F^{1,1}-F^{4,1}-F^{5,1}),v\right>_{(\mathcal{H}^1_T)^\ast}=\int_0^T\int_{\Om}\pa_tF^{1,1}\cdot vJ-\int_0^T\int_{-\ell}^{\ell}\pa_tF^{4,1}\cdot v-\int_0^T\int_{\Sigma_s}\pa_tF^{5,1}(v\cdot\tau)J,
\eeq
for each $v\in \mathcal{V}$.
Since we assume that $F^i=0$, \eqref{est:dual_1} reduces to
\beq\label{est:dual_2}
\left<\pa_t(F^{1,1}-F^{4,1}-F^{5,1}),v\right>_{(\mathcal{H}^1_T)^\ast}=\int_0^T\int_{\Om}\pa_tG^1\cdot vJ-\int_0^T\int_{-\ell}^{\ell}\pa_tG^4\cdot v-\int_0^T\int_{\Sigma_s}\pa_tG^5(v\cdot\tau)J,
\eeq
for each $v\in \mathcal{V}$. Then we use an integration by parts to compute
\beq
\int_\Om\dive_\mathcal{A}(\mathbb{D}_\mathcal{A}(Ru))vJ=-\frac12\int_\Om\mathbb{D}_\mathcal{A}(Ru):\mathbb{D}_\mathcal{A}vJ+\int_{-\ell}^\ell\mathbb{D}_\mathcal{A}(Ru)\mathcal{N}\cdot v+\int_{\Sigma_s}\mathbb{D}_\mathcal{A}(Ru)\nu\cdot\tau(v\cdot\tau)J,
\eeq
which reduces \eqref{est:dual_2} to the following equality:
\beq\label{est:dual_3}
\begin{aligned}
&\left<\pa_t(F^{1,1}-F^{4,1}-F^{5,1}),v\right>_{(\mathcal{H}^1_T)^\ast}=\frac12\int_0^T\int_\Om\pa_t(\mathbb{D}_\mathcal{A}(Ru)):\mathbb{D}_\mathcal{A}vJ\\
&\quad+\int_0^T\int_{\Om}\left[\pa_t(G^1-\mu\dive_\mathcal{A}(\mathbb{D}_\mathcal{A}(Ru)))+\mu\dive_{\pa_t\mathcal{A}}(\mathbb{D}_\mathcal{A}(Ru))\right]\cdot vJ\\
&\quad-\int_0^T\int_{-\ell}^{\ell}\pa_t(G^4-\mu\mathbb{D}_\mathcal{A}(Ru)\mathcal{N})\cdot v-\int_0^T\int_{\Sigma_s}\pa_t(G^5-\mu\mathbb{D}_\mathcal{A}(Ru)\nu\cdot\tau)(v\cdot\tau)J.
\end{aligned}
\eeq
Then we use H\"older's inequality and the same computation in Lemma \ref{lemma:est_force} to derive the resulting bounds.
\end{proof}

Now, we give some estimates for the difference between $\pa_tu$ and $D_tu$. The proof is similar as that of Lemma \ref{lemma:est_force}, so we omit it here.
\begin{lemma}\label{lem:difference_u_Dt_u}
\beq
\|\pa_tu-D_tu\|_{L^2W_\delta^2}^2\lesssim \mathfrak{D}(\eta)(\|u_0^\epsilon\|_{W_{\delta}^2}^2+\|u\|_{L^2 W_{\delta}^2}^2+\|\pa_tu\|_{L^2 W_{\delta}^2}^2),
\eeq
  \beq
  \|\pa_t u-D_tu\|_{L^2H^1}^2+\|\pa_tu-D_tu\|_{L^2H^0(\Sigma_s)}^2\lesssim \mathfrak{E}(\eta)(\|u\|_{L^2H^1}^2+\|u\|_{L^2 W_\delta^2}^2),
  \eeq
  and
  \beq
  \|\pa_t^2u-\pa_tD_tu\|_{L^2H^1}^2+\|\pa_t^2u-\pa_tD_tu\|_{L^2H^0(\Sigma_s)}^2\lesssim \mathfrak{K}(\eta)(\|u\|_{L^2 W_\delta^2}^2+\|\pa_tu\|_{L^2 W_\delta^2}^2).
  \eeq
\end{lemma}

Now, we define the quantities we need to estimate as follows.
\beq\label{def:dissipation}
\begin{aligned}
\mathfrak{D}(u,p,\xi):&=\sum_{j=0}^2\left(\|\pa_t^ju\|_{L^2H^1}^2+\|\pa_t^ju\|_{L^2H^0(\Sigma_s)}^2+\int_0^T\left[\pa_t^ju\cdot\mathcal{N}\right]_\ell^2\right)\\
&\quad+\sum_{j=0}^2\left(\|\pa_t^jp\|_{L^2H^0}^2+\|\pa_t^j\xi\|_{L^2H^{3/2}}^2\right)+\|\pa_t^3\xi\|_{L^2W_\delta^{1/2}}^2\\
&\quad+\sum_{j=0}^1\left(\|\pa_t^ju\|_{L^2W_\delta^2}^2+\|\pa_t^jp\|_{L^2\mathring{W}_\delta^1}^2+\|\pa_t^j\xi\|_{L^2W_\delta^{5/2}}^2\right),
\end{aligned}
\eeq
\beq\label{def:energy}
\begin{aligned}
\mathfrak{E}(u,p,\xi):&=\|u\|_{L^\infty W_\delta^2}^2+\|\pa_tu\|_{L^\infty H^1}^2+\|p\|_{L^\infty \mathring{W}_\delta^1}^2+\|\pa_tp\|_{L^\infty H^0}^2\\
&\quad+\|\xi\|_{L^\infty W_\delta^{5/2}}^2+\|\pa_t\xi\|_{L^\infty H^{3/2}}^2+\sum_{j=0}^2\|\pa_t^j\xi\|_{L^\infty H^1}^2,
\end{aligned}
\eeq
and
\beq\label{def:DEK}
\mathfrak{K}(u,p,\xi):=\mathfrak{E}(u,p,\xi)+\mathfrak{D}(u,p,\xi).
\eeq

Now for convenience, we introduce two new spaces
\beq\label{def:x}
\begin{aligned}
\mathcal{X}=\Big\{&(u,p,\eta)|u\in C^0([0,T];W_\delta^2(\Om)),\pa_tu\in C^0([0,T];H^1(\Om)), p\in C^0 ([0,T];\mathring{W}_\delta^1(\Om)),\\
&\pa_tp\in C^0([0,T];\mathring{H}^0(\Om)), \eta\in C^0([0,T];W_\delta^{5/2}(-\ell,\ell)), \pa_t\eta \in C^0([0,T];H^{3/2}(-\ell,\ell)),\\
&\pa_t\eta\in C^0([0,T];\mathring{H}^1(-\ell,\ell)),\pa_t^2\eta\in C^0([0,T];\mathring{H}^1(-\ell,\ell))\Big\},
\end{aligned}
\eeq
endowed with norm $\|(u,p,\eta)\|_{\mathcal{X}}=\sqrt{\mathfrak{E}(u,p,\eta)}$, and
\beq\label{def:y}
\begin{aligned}
  \mathcal{Y}=\Big\{&(u,p,\eta)|u\in L^2([0,T];W_\delta^2(\Om))\cap L^2([0,T];H^1(\Om))\cap L^2([0,T];H^0(\Sigma_s)),\\
  &\pa_t u\in L^2([0,T];W_\delta^2(\Om))\cap L^2([0,T];H^1(\Om))\cap L^2([0,T];H^0(\Sigma_s)), j=0,1,2,\\
  &\pa_t^2u\in L^2([0,T];H^1(\Om))\cap L^2([0,T];H^0(\Sigma_s)), [\pa_t^ju\cdot\mathcal{N}]_\ell\in L^2([0,T]),\\
  &p\in L^2([0,T];\mathring{W}_\delta^1(\Om))\cap L^2([0,T];H^0(\Om)),\pa_tp\in L^2([0,T];\mathring{W}_\delta^1(\Om))\cap L^2([0,T];H^0(\Om)),\\
  &\pa_t^2p\in L^2([0,T];\mathring{H}^0(\Om)),\eta\in L^2([0,T];W_\delta^{5/2}((-\ell,\ell)))\cap L^2([0,T];\mathring{H}^{3/2}((-\ell,\ell))),\\
  &\pa_t\eta\in L^2([0,T];W_\delta^{5/2}((-\ell,\ell)))\cap L^2([0,T];\mathring{H}^{3/2}((-\ell,\ell))), \pa_t^2\eta\in L^2([0,T];\mathring{H}^{3/2}((-\ell,\ell))),\\
  &\pa_t^3\eta\in L^2([0,T];\mathring{W}_\delta^{1/2}((-\ell,\ell)))\Big\}
\end{aligned}
\eeq
endowed with the norm $\|(u,p,\eta)\|_{\mathcal{Y}}=\sqrt{\mathfrak{D}(u,p,\eta)}$.

In the following theorem, we set the forcing terms $F^i=0$, $i=1, 4, 5$ for the sake of brevity since this is all we will need in our subsequent analysis.  A version of the theorem may also be proved  with the forcing terms under some natural regularity assumptions.
\begin{theorem}\label{thm:higher order}
  Suppose that $\eta_0\in W_\delta^{5/2}(-\ell,\ell)$, $\pa_t\eta(0)\in \mathring{H}^{3/2}((-\ell,\ell))$ and $\pa_t^2\eta(0)\in \mathring{H}^1((-\ell,\ell))$ satisfy the compatibility \eqref{compati_1} and \eqref{compati_2}, that $\mathfrak{K}(\eta)\le\alpha$ is sufficiently small satisfying the assumption in Lemma \ref{lem:equivalence_norm} and Theorem 5.9 in \cite{GT2},  and that $F^i=0$, $i=1,4,5$. Let $u_0^\epsilon\in W_\delta^2(\Om)$, $D_tu^\epsilon(0)\in H^1(\Om)$, $p_0^\epsilon\in \mathring{W}_\delta^1(\Om)$, $\pa_tp^\epsilon(0)\in \mathring{H}^0(\Om)$, $\pa_t\xi(0)\in H^{3/2}((-\ell,\ell))$ and $\pa_t^2\xi(0)\in H^1((-\ell,\ell))$, all be determined in terms of $\eta_0$, $\pa_t\eta(0)$ and $\pa_t^2\eta(0)$ as in Section \ref{subsection_initial}.
  Then for each $0<\epsilon\le1$ satisfying \eqref{eq:ep}, there exists $T_\epsilon>0$ such that for $0<T\le T_\epsilon$, then there exists a unique strong solution $(u,p,\xi)$ to \eqref{eq:modified_linear} on $[0, T]$ such that
  \beq
  (u,p,\xi)\in\mathcal{X}\cap\mathcal{Y}.
  \eeq
  The pair $(D_t^ju, \pa_t^jp, \pa_t^j\xi)$ satisfies
  \beq\label{eq:higher}
  \left\{
  \begin{aligned}
    &-\mu\Delta_{\mathcal{A}}(D_t^ju)+\nabla_{\mathcal{A}}\pa_t^jp=F^{1,j},\quad &\text{in} \quad &\Om,\\
    &\dive_{\mathcal{A}}(D_t^ju)=0,\quad &\text{in} \quad &\Om,\\
    &S_{\mathcal{A}}(\pa_t^jp, D_t^ju)\mathcal{N}=\mathcal{L}(\pa_t^j\xi+\epsilon\pa_t^{j+1}\xi)\mathcal{N}-\sigma\pa_1\pa_t^jF^3\mathcal{N}+F^{4,j},\quad &\text{on} \quad &\Sigma,\\
    &(S_{\mathcal{A}}(\pa_t^jp, D_t^ju)\nu-\beta(D_t^ju))\cdot\tau=F^{5,j},\quad &\text{on} \quad &\Sigma_s,\\
    &D_t^ju\cdot\nu=0,\quad &\text{on} \quad &\Sigma_s,\\
    &\pa_t^{j+1}\xi=D_t^ju\cdot\mathcal{N},\quad &\text{on} \quad &\Sigma,\\
    &\mp\sigma\frac{\pa_1\pa_t^j\xi}{(1+|\pa_1\zeta_0|^2)^{3/2}}(\pm\ell)=\kappa(D_t^ju\cdot\mathcal{N})(\pm\ell)\pm \sigma \pa_t^jF^3(\pm\ell)-\kappa\pa_t^j\hat{\mathscr{W}}(\pa_t\eta),
  \end{aligned}
  \right.
  \eeq
  in the strong sense with initial data $(D_t^ju(0), \pa_t^jp(0), \pa_t^j\xi(0))$ for $j=0,1$ and in the weak sense for $j=2$. Moreover, the solution satisfies the estimate
  \beq\label{est:higher}
  \mathfrak{K}(u,p,\xi)\le C(\epsilon)T(\mathfrak{K}(\eta)+\mathfrak{E}_0)+C_0(\mathfrak{E}_0+\mathfrak{E}(\eta)\mathfrak{K}(\eta)),
  \eeq
  where $C_0$ is a positive constant independent of $\epsilon$.
\end{theorem}
  \begin{proof}
    Step 1 -- Following Theorem \ref{thm:linear_low}.

    First consider the case $j=0$.   Since the compatibility condition in Section \ref{subsection_initial} is satisfied and $\mathfrak{K}(\eta)$ is small enough, Theorem \ref{thm:linear_low} guarantees the existence of $(u,p,\xi)$ satisfying \eqref{strong_1} and \eqref{strong_2}. $(D_t^ju,\pa_t^jp,\pa_t^j\xi)$ is a unique solution of \eqref{eq:higher} in the strong sense when $j=0$ and in the weak sense when $j=1$. For $j=1$, the assumption of Theorem \ref{thm:linear_low} are satisfied by Lemma \ref{lemma:est_force}, Lemma \ref{lem:est_force_dual}, and the compatibility conditions in section \ref{subsection_initial}.  Then according to  Theorem \ref{thm:linear_low} and the elliptic estimate for $\xi+\epsilon\pa_t\xi$, we have that $(D_tu, \pa_tp, \pa_t\xi)$ is a unique strong solution of \eqref{eq:higher}, and $(D_t^2u, \pa_t^2p, \pa_t^2\xi)$ is a unique weak solution of \eqref{eq:higher}.
    Moreover,
  \beq
  \begin{aligned}
  &\|D_tu\|_{L^2 H^1}^2+\|D_tu\|_{L^2 H^0(\Sigma_s)}^2+\|[D_tu\cdot\mathcal{N}]_\ell\|_{L^2([0,T])}^2+\|D_tu\|_{L^2 W_\delta^2}^2+\|\pa_t D_tu\|_{L^2H^1}^2\\
  &\quad+\|\pa_t D_tu\|_{L^2 H^0(\Sigma_s)}^2+\|[\pa_t D_tu\cdot\mathcal{N}]_\ell\|_{L^2([0,T])}^2+\|\pa_tp\|_{L^2 H^0}^2+\|\pa_tp\|_{L^2\mathring{W}_\delta^1}^2+\|\pa_t^2p\|_{L^2H^0}^2\\
  &\quad+\|\pa_t\xi\|_{L^\infty H^1}^2+\|\pa_t^2\xi\|_{L^2 H^{3/2}}^2+\|\pa_t\xi+\epsilon\pa_t^2\xi\|_{L^2 W_\delta^{5/2}}^2
  +\|\pa_t^2\xi\|_{L^\infty H^1}^2+\|\pa_t\xi\|_{L^2H^{3/2}}^2\\
  &\lesssim C(\epsilon)T\mathfrak{E}(\eta)+\mathfrak{E}_0+\|(F^{1,1}-F^{4,1}-F^{5,1})(0)\|_{(\mathcal{H}^1)^\ast}^2+\mathfrak{E}(\eta)\mathfrak{K}(\eta)\\
  &\quad+(1+\mathfrak{E}(\eta))(\|F^{1,1}\|_{L^2 W_\delta^0}^2+\|F^{4,1}\|_{L^2 W_\delta^{1/2}}^2+\|F^{5,1}\|_{L^2 W_\delta^{1/2}}^2)\\
   &\quad
  +(1+\mathfrak{E}(\eta))\|\pa_t(F^{1,1}-F^{4,1}-F^{5,1})\|_{(\mathcal{H}^1_T)^{\ast}}^2.
  \end{aligned}
  \eeq
  Since $\mathfrak{K}(\eta)$ is sufficiently small, then Lemma \ref{lem:difference_u_Dt_u} and the fact that $D_tu\cdot\mathcal{N}(\pm\ell)=\pa_tu\cdot\mathcal{N}(\pm\ell)$ can reduce the above estimate to
  \beq\label{est:D_1}
  \begin{aligned}
  &\|\pa_tu\|_{L^2 H^1}^2+\|\pa_tu\|_{L^2 H^0(\Sigma_s)}^2+\|[\pa_tu\cdot\mathcal{N}]_\ell\|_{L^2([0,T])}^2+\|\pa_tu\|_{L^2 W_\delta^2}^2+\|\pa_t^2 u\|_{L^2H^1}^2\\
  &\quad+\|\pa_t^2u\|_{L^2 H^0(\Sigma_s)}^2+\|[\pa_t^2u\cdot\mathcal{N}]_\ell\|_{L^2([0,T])}^2+\|\pa_tp\|_{L^2 H^0}^2+\|\pa_tp\|_{L^2\mathring{W}_\delta^1}^2+\|\pa_t^2p\|_{L^2H^0}^2\\
  &\quad+\|\pa_t\xi\|_{L^\infty H^1}^2+\|\pa_t^2\xi\|_{L^2 H^{3/2}}^2+\|\pa_t\xi+\epsilon\pa_t^2\xi\|_{L^2 W_\delta^{5/2}}^2
  +\|\pa_t^2\xi\|_{L^\infty H^1}^2+\|\pa_t\xi\|_{L^2H^{3/2}}^2\\
  &\lesssim C(\epsilon)T\mathfrak{E}(\eta)+\mathfrak{E}_0+\mathfrak{E}(\eta)\mathfrak{K}(\eta)+\mathfrak{K}(\eta)( \|u\|_{L^2H^1}^2+\|u\|_{L^2W_\delta^2}^2+\|p\|_{L^2\mathring{W}_\delta^1}^2+\|\xi+\epsilon\pa_t\xi\|_{L^2 W_\delta^{5/2}}^2)\\
  &\lesssim C(\epsilon)T(\mathfrak{E}_0+\mathfrak{K}(\eta))+\mathfrak{E}_0+\mathfrak{E}(\eta)\mathfrak{K}(\eta).
  \end{aligned}
  \eeq
  Then from the extension and restriction theory of weighted Sobolev spaces, we find that
  \beq\label{est:D_2}
  \begin{aligned}
  \|\pa_t\xi\|_{L^2W_\delta^{5/2}}^2&\lesssim \epsilon^2\|\pa_t\xi(0)\|_{L^2W_\delta^{5/2}}^2+\|\pa_t\xi+\epsilon\pa_t^2\xi\|_{L^2 W_\delta^{5/2}}^2\\
  &\lesssim\|\xi(0)\|_{L^2W_\delta^{5/2}}^2+\|\xi(0)+\epsilon\pa_t\xi(0)\|_{L^2W_\delta^{5/2}}^2+\|\pa_t\xi+\epsilon\pa_t^2\xi\|_{L^2 W_\delta^{5/2}}^2\\
  &\lesssim \mathfrak{E}_0+\|\pa_t\xi+\epsilon\pa_t^2\xi\|_{L^2W_\delta^{5/2}}^2.
  \end{aligned}
  \eeq
  We can directly estimate $\pa_t^3\xi$ by
  \beq\label{est:D_3}
  \begin{aligned}
  \|\pa_t^3\xi\|_{W_\delta^{1/2}}^2&=\|\pa_tD_tu\cdot\mathcal{N}+D_tu\cdot\pa_t\mathcal{N}\|_{W_\delta^{1/2}}^2\\
  &\lesssim \|\pa_tD_tu\|_1^2\|\eta\|_{W_\delta^{5/2}}^2+\|D_tu\|_1^2\|\pa_t\eta\|_{W_\delta^{5/2}}^2.
  \end{aligned}
  \eeq
  Then from \eqref{lem:difference_u_Dt_u}, we see that
  \beq
  \|\pa_t^3\xi\|_{L^2W_\delta^{1/2}}^2\lesssim\mathfrak{K}(\eta)(\|u\|_{L^2H^1}^2+\|u\|_{L^2 W_\delta^2}^2+\|\pa_tu\|_{L^2 W_\delta^2}^2)
  \eeq
  Thus\eqref{est:D_1} -- \eqref{est:D_3} imply
  \beq\label{est:higher_1}
  \mathfrak{D}(u,p,\xi)\le C(\epsilon)T(\mathfrak{K}(\eta)+\mathfrak{E}_0)+C_0(\mathfrak{E}_0+\mathfrak{E}(\eta)\mathfrak{K}(\eta)).
  \eeq
  Step 2 -- Other terms in $\mathfrak{E}$.

 Arguing as in Lemma \ref{lem:continuous}, we may directly derive the bounds
    \begin{align*}
  \|\pa_tu\|_{L^\infty H^1}^2&\lesssim \|\pa_tu^\epsilon(0)\|_{ H^1}^2+\|\pa_tu\|_{L^2H^1}^2+\|\pa_t^2u\|_{L^2H^1}^2,\\
  \|\pa_tp\|_{L^\infty H^0}^2&\lesssim \|\pa_tp^\epsilon(0)\|_{ H^1}^2+\|\pa_tp\|_{L^2H^1}^2+\|\pa_t^2p\|_{L^2H^1}^2,\\
  \|\pa_t\xi\|_{L^\infty H^{3/2}}^2&\lesssim \|\pa_t\xi(0)\|_{ H^{3/2}}^2+\|\pa_t\xi\|_{L^2H^{3/2}}^2+\|\pa_t^2\xi\|_{L^2H^{3/2}}^2,
  \end{align*}
  which together with Lemma \ref{lem:continuous}, Lemma \ref{lem:continuous_2} and the construction of the initial data imply that
  \beq\label{est:higher_2}
  \mathfrak{E}(u,p,\xi)\le C(\epsilon)T(\mathfrak{K}(\eta)+\mathfrak{E}_0)+C_0(\mathfrak{E}_0+\mathfrak{E}(\eta)\mathfrak{K}(\eta)).
  \eeq
  Then \eqref{est:higher_1} and \eqref{est:higher_2} imply the conclusion \eqref{est:higher}.
  \end{proof}

\section{Local well-posedness for the full nonlinear equation}

We now consider the local well-posedness of the full problem \eqref{eq:geometric perturbation}. We first construct an approximate solution $(u^\epsilon, p^\epsilon, \eta^\epsilon)$ for \eqref{eq:geometric perturbation} and for each $\epsilon>0$. Then our plan is to let $\epsilon\rightarrow0$ to obatin the solution of \eqref{eq:geometric perturbation}.

\subsection{Existence of approximate solutions}
We now construct a sequence of approximate solutions $(u^\epsilon, p^\epsilon, \eta^\epsilon)$ for each $0<\epsilon\le1$ satisfying \eqref{eq:ep}. For simplicity, we will typically drop $\epsilon$ in the notation and  denote the unknown as $(u,p,\eta)$ instead of $(u^\epsilon, p^\epsilon, \eta^\epsilon)$.

Now we consider the $\epsilon$--perturbation problem of the original system \eqref{eq:geometric perturbation} as
\begin{equation}\label{eq:modified_geometric perturbation}
  \left\{
  \begin{aligned}
    &\dive_{\mathcal{A}}S_{\mathcal{A}}(p,u)=-\mu\Delta_{\mathcal{A}}u+\nabla_{\mathcal{A}}p=0,\quad &\text{in}&\quad \Om,\\
    &\dive_{\mathcal{A}}u=0,\quad &\text{in}&\quad \Om,\\
    &S_{\mathcal{A}}(p,u)\mathcal{N}=g(\eta+\epsilon\eta_t)\mathcal{N}-\sigma\pa_1\left(\frac{\pa_1\eta+\epsilon\pa_1\eta_t}{(1+|\pa_1\zeta_0|)^{3/2}}\right)\mathcal{N}-\sigma\pa_1(\mathcal{R}(\pa_1\zeta_0,\pa_1\eta))\mathcal{N},\quad &\text{on}&\quad\Sigma,\\
    &(S_{\mathcal{A}}(p,u)\nu-\beta u)\cdot\tau=0,\quad &\text{on}&\quad\Sigma_s,\\
    &u\cdot\nu=0,\quad &\text{on}&\quad\Sigma_s,\\
    &\pa_t\eta=u\cdot\mathcal{N},\quad &\text{on}&\quad\Sigma,\\
    &\kappa\pa_t\eta(\pm\ell,t)=\mp\sigma\frac{\pa_1\eta}{(1+|\zeta_0|^2)^{3/2}}(\pm\ell,t)\mp\mathcal{R}(\pa_1\zeta_0,\pa_1\eta)(\pm\ell,t)-\kappa\hat{\mathscr{W}}(\pa_t\eta(\pm\ell,t)).
  \end{aligned}
  \right.
\end{equation}
where $\mathcal{A}$, $\mathcal{N}$ are in terms of $\eta^\epsilon$ and the initial data are $\eta(x_1,0)=\eta_0(x_1)$, $\pa_t\eta(x_1,0)$ and $\pa_t^2\eta(x_1,0)$.

Our strategy is to
work in a metric space that requires high regularity estimates to hold
but that is endowed with a low-regularity metric. First we will find a complete metric space, endowed with a weak choice of a metric, compatible with the linear estimates in Theorem \ref{thm:higher order}. Then we will prove that the fixed point on this metric space gives a solution to \eqref{eq:modified_geometric perturbation}.

We now define the desired metric space.
\begin{definition}\label{def:S}
  Suppose that $T>0$. For $\sigma\in(0,\infty)$ we define the space
  \beq
  \begin{aligned}
&S(T, \sigma)
=\Big\{(u,p,\eta)\in L^2 H^1\times L^2 \mathring{H}^0\times (L^\infty W_\delta^{5/2}\cap \dot{H}^1([0,T];\pm\ell),\Big|(u,p,\eta))\in\mathcal{X}\cap\mathcal{Y},\ \text{with}\\
   &\quad \mathfrak{K}(u,p,\eta)^{1/2}\le\sigma\ \text{and}\ (u,p,\eta)\ \text{achieve the initial data as Section \ref{subsection_initial}}\Big\}.
  \end{aligned}
  \eeq
  We endow this space with the metric
  \beq
  d((u,p,\eta),(v,q,\xi))=\|u-v\|_{L^2 H^1}+\|p-q\|_{L^2 \mathring{H}^0}+\|\eta-\xi\|_{L^\infty W_\delta^{5/2}}+\|[\pa_t\eta-\pa_t\xi]_\ell\|_{L^2([0,T])},
  \eeq
  where here the temporal norm is evaluated on the set $[0,T]$.
\end{definition}
In order to use the contraction mapping principle we need to first show that this metric space is complete.
\begin{theorem}
  $S(T, \sigma)$ is a complete metric space.
\end{theorem}
\begin{proof}
  Suppose that $\{(u^m,p^m,\eta^m)\}_{m=0}^\infty\subseteq S(T, \sigma)$ is a Cauchy sequence. Since $L^2 H^1\times L^2 \mathring{H}^0\times (L^\infty W_\delta^{5/2}\cap \dot{H}^1([0,T];\pm\ell))$ is a Banach space, there exists $(u,p,\eta)\in L^2 H^1\times L^2 \mathring{H}^0\times (L^\infty W_\delta^{5/2}\cap \dot{H}^1([0,T];\pm\ell))$ such that
  \[
  u^m\rightarrow u\ \text{in}\ L^2 H^1,\quad p^m\rightarrow p\ \text{in}\ L^2\mathring{H}^0, \quad \eta^m\rightarrow\eta\ \text{in}\ L^\infty W_\delta^{5/2}, \quad \pa_t\eta^m\rightarrow \pa_t\eta \ \text{in}\ L^2([0,T];\pm\ell)
  \]
  as $m\rightarrow\infty$.

  For each $m$, we have that $\mathfrak{K}(u^m,p^m,\eta^m)\le\sigma^2$. Then up to the extraction of a subsequence we have that
  \beq
  (u^m,p^m,\eta^m)\rightharpoonup (u,p,\eta)\ \text{weakly--}\ast\ \text{in}\ \mathcal{X},\quad (u^m,p^m,\eta^m)\rightharpoonup (u,p,\eta)\ \text{weakly}\ \text{in}\ \mathcal{Y},
  \eeq
  which imply that $(u,p,\eta)\in\mathcal{X}\cap\mathcal{Y}$. Then according to lower semicontinuity,
  \beq
  \mathfrak{K}(u,p,\eta)^{1/2}\le \inf_m\mathfrak{K}(u^m,p^m,\eta^m)^{1/2}\le\sigma.
  \eeq
Thus $S(T,\sigma)$ is complete.
\end{proof}

Next we employ the metric space $S(T,\sigma)$ and a contraction mapping argument to produce a solution to \eqref{eq:modified_geometric perturbation}.

\begin{theorem}\label{thm: fixed point}
  There exists a constant $C>0$ such that for each $0<\epsilon\le\min\{1, 1/(8C)\}$ there exists a unique solution $(u^\epsilon, p^\epsilon, \eta^\epsilon)$ to \eqref{eq:modified_geometric perturbation} belong to the metric space $S(T_\epsilon, \sigma)$, where $T_\epsilon>0$ and $\sigma>0$ are sufficiently small.  In particular $(u^\epsilon, p^\epsilon, \eta^\epsilon)\in \mathcal{X}\cap\mathcal{Y}$, where $\mathcal{X}$ and $\mathcal{Y}$ are defined in \eqref{def:x} and \eqref{def:y}.
\end{theorem}
\begin{proof}
Throughout the proof $P(\cdot)$ denotes a polynomial such that $P(0)=0$, which is allowed to be changed from line to line.

Step 1 -- The metric space.

Suppose that $\mathfrak{K}(\eta)\le \alpha$ is sufficiently small. Then $C_0\mathfrak{E}(\eta)\mathfrak{K}(\eta)\le \alpha/4$.   Let $C(\epsilon)$ and $C_0$ are the same as in \eqref{est:higher}.   Now  take
$T_\epsilon>0$ small enough such that $C(\epsilon)T_\epsilon\alpha\le\alpha/4$. Then we take the initial data small enough such that $C(\epsilon)T_\epsilon\mathfrak{E}_0\le\alpha/4$ and $C_0\mathfrak{E}_0\le\alpha/4$.
Then we take $\sigma\le\alpha^{1/2}$.
For every $(u,p,\eta)\in S(T_\epsilon, \sigma)$, let $(\tilde{u},\tilde{p},\tilde{\eta})$ be the unique solution of the linear problem of
\begin{equation}\label{eq:linear_fix}
  \left\{
  \begin{aligned}
    &\dive_{\mathcal{A}}S_{\mathcal{A}}(\tilde{p},\tilde{u})=-\mu\Delta_{\mathcal{A}}\tilde{u}+\nabla_{\mathcal{A}}\tilde{p}=0,\quad &\text{in}&\quad \Om,\\
    &\dive_{\mathcal{A}}\tilde{u}=0,\quad &\text{in}&\quad \Om,\\
    &S_{\mathcal{A}}(\tilde{p},\tilde{u})\mathcal{N}=g(\tilde{\eta}+\epsilon\pa_t\tilde{\eta})\mathcal{N}-\sigma\pa_1\left(\frac{\pa_1\tilde{\eta}+\epsilon\pa_1\pa_t\tilde{\eta}}{(1+|\pa_1\zeta_0|)^{3/2}}\right)\mathcal{N}-\sigma\pa_1(\mathcal{R}(\pa_1\zeta_0,\pa_1\eta))\mathcal{N},\quad &\text{on}&\quad\Sigma,\\
    &(S_{\mathcal{A}}(\tilde{p},\tilde{u})\nu-\beta \tilde{u})\cdot\tau=0,\quad &\text{on}&\quad\Sigma_s,\\
    &\tilde{u}\cdot\nu=0,\quad &\text{on}&\quad\Sigma_s,\\
    &\pa_t\tilde{\eta}=\tilde{u}\cdot\mathcal{N},\quad &\text{on}&\quad\Sigma,\\
    &\kappa\pa_t\tilde{\eta}(\pm\ell,t)=\mp\sigma\frac{\pa_1\tilde{\eta}}{(1+|\zeta_0|^2)^{3/2}}(\pm\ell,t)\mp\mathcal{R}(\pa_1\zeta_0,\pa_1\eta)(\pm\ell,t)-\kappa\hat{\mathscr{W}}(\pa_t\eta(\pm\ell,t)),
  \end{aligned}
  \right.
\end{equation}
where $\mathcal{A}$ and $\mathcal{N}$ are in terms of $\eta$, and the initial data $\tilde{\eta}(0)=\eta_0$, $\pa_t\tilde{\eta}(0)=\pa_t\eta(0)$ and $\pa_t^2\tilde{\eta}(0)=\pa_t^2\eta(0)$.
By  Theorem \ref{thm:higher order}  we have the estimate
\beq\label{est:regularity_linear}
\mathfrak{K}(\tilde{u}, \tilde{p}, \tilde{\eta})\le \sigma^2,
\eeq
which implies that
\beq
(\tilde{u},\tilde{p},\tilde{\eta})\in S(T_\epsilon,\sigma).
\eeq

Step 2 -- Contraction.

Define $A: (u,p,\eta)=(\tilde{u},\tilde{p}, \tilde{\eta})$. Now we prove that
\[
A:S(T_\epsilon, \sigma)\rightarrow S(T_\epsilon,\sigma)
 \]
 is a strict contraction mapping with the metric in the Definition \ref{def:S}.
Choose $(u^i,p^i,\eta^i)\in S(T_\epsilon, \sigma)$, and define $A(u^i,p^i,\eta^i)=(\tilde{u}^i,\tilde{p}^i, \tilde{\eta}^i)$ as above, $i=1,2$. For simplicity, we will abuse notation and denote $u=u^1-u^2$, $p=p^1-p^2$, $\eta=\eta^1-\eta^2$ and the same for $\tilde{u},\tilde{p}, \tilde{\eta}$. From the difference of equation for $(\tilde{u}^i,\tilde{p}^i, \tilde{\eta}^i)$, $i=1,2$, we know that
\begin{equation}\label{linear_fix1}
  \left\{
  \begin{aligned}
    &\dive_{\mathcal{A}^1}S_{\mathcal{A}^1}(\tilde{p},\tilde{u})=\mu\dive_{\mathcal{A}^1}(\mathbb{D}_{\mathcal{A}^1-\mathcal{A}^2}\tilde{u}^2)+R^1,\quad &\text{in}&\quad \Om,\\
    &\dive_{\mathcal{A}^1}\tilde{u}=R^2,\quad &\text{in}&\quad \Om,\\
    &S_{\mathcal{A}^1}(\tilde{p},\tilde{u})\mathcal{N}^1=\mu\mathbb{D}_{\mathcal{A}^1-\mathcal{A}^2}\tilde{u}^2\mathcal{N}^1 +g(\tilde{\eta}+\epsilon\pa_t\tilde{\eta})\mathcal{N}^1-\sigma\pa_1\left(\frac{\pa_1\tilde{\eta}+\epsilon\pa_1\pa_t\tilde{\eta}}{(1+|\pa_1\zeta_0|)^{3/2}}\right)\mathcal{N}^1\\
    &\quad-\sigma\pa_1F^3\mathcal{N}^1+R^3,\quad &\text{on}&\quad\Sigma,\\
    &(S_{\mathcal{A}^1}(\tilde{p},\tilde{u})\nu-\beta \tilde{u})\cdot\tau=\mu\mathbb{D}_{\mathcal{A}^1-\mathcal{A}^2}\tilde{u}^2\nu\cdot\tau,\quad &\text{on}&\quad\Sigma_s,\\
    &\tilde{u}\cdot\nu=0,\quad &\text{on}&\quad\Sigma_s,\\
    &\pa_t\tilde{\eta}=\tilde{u}\cdot\mathcal{N}^1+R^5,\quad &\text{on}&\quad\Sigma,\\
    &\kappa\pa_t\tilde{\eta}(\pm\ell,t)=\mp\sigma\frac{\pa_1\tilde{\eta}}{(1+|\zeta_0|^2)^{3/2}}(\pm\ell,t)\mp F^3(\pm\ell,t))-R^6,\\
    &\tilde{u}(x,0)=0,\quad \tilde{\eta}(x_1,0)=0,
  \end{aligned}
  \right.
\end{equation}
where $R^1$, $R^2$, $R^3$, $R^4$, $R^5$, $R^6$ are defined by
  \begin{align*}
    R^1&=\mu\dive_{(\mathcal{A}^1-\mathcal{A}^2)}(\mathbb{D}_{\mathcal{A}^2}\tilde{u}^2)-\nabla_{(\mathcal{A}^1-\mathcal{A}^2)}\tilde{p}^2,\\
    R^2&=-\dive_{(\mathcal{A}^1-\mathcal{A}^2)}\tilde{u}^2,\\
    R^3&=-\tilde{p}^2(\mathcal{N}^1-\mathcal{N}^2)+\mathbb{D}_{\mathcal{A}^2}\tilde{u}^2(\mathcal{N}^1-\mathcal{N}^2)+g(\tilde{\eta}^2+\epsilon\pa_1\pa_t\tilde{\eta}^2)(\mathcal{N}^1-\mathcal{N}^2)\\
    &\quad-\sigma\pa_1\left(\frac{\pa_1\tilde{\eta}^2+\epsilon\pa_1\pa_t\tilde{\eta}^2}{(1+|\zeta_0|^2)^{3/2}}\right)(\mathcal{N}^1-\mathcal{N}^2)-\sigma\pa_1 F^{3,2}(\mathcal{N}^1-\mathcal{N}^2),\\
    R^5&=\tilde{u}^2\cdot(\mathcal{N}^1-\mathcal{N}^2),\\
    R^6&=\kappa(\hat{\mathscr{W}}(\pa_t\eta^1(\pm\ell,t))-\hat{\mathscr{W}}(\pa_t\eta^2(\pm\ell,t))),
  \end{align*}
and $\mathcal{A}^i$, $\mathcal{N}^i$, $F^{3, i}=\mathcal{R}(\pa_1\zeta_0,\pa_1\eta^i)$ are in terms of $\eta^i$, $i=1,2$. Here $F^3=F^{3,1}-F^{3,2}$.

We now have the pressureless weak formulation of \eqref{linear_fix1} as
\beq\label{weak_1}
\begin{aligned}
  &\frac\mu2\int_\Om\mathbb{D}_{\mathcal{A}^1}\tilde{u}:\mathbb{D}_{\mathcal{A}^1}wJ^1+\beta\int_{\Sigma_s}J^1(\tilde{u}\cdot\tau)(w\cdot\tau)+(\tilde{\eta}+\epsilon\pa_t\tilde{\eta},w\cdot\mathcal{N}^1)_{1,\Sigma}+[\tilde{u}\cdot\mathcal{N}^1,w\cdot\mathcal{N}^1]_\ell\\
  &\quad-\epsilon \mathfrak{b}(\pa_t\tilde{\eta},w\cdot\mathcal{N}^1)_\ell\\
  &=\mu\int_\Om \mathbb{D}_{\mathcal{A}^1-\mathcal{A}^2}\tilde{u}^2:\mathbb{D}_{\mathcal{A}^1}wJ^1+\int_\Om R^1\cdot wJ^1-\int_{-\ell}^\ell\sigma F^3\pa_1(w\cdot\mathcal{N}^1)+R^3\cdot w-[w\cdot\mathcal{N}^1, R^5+R^6]_\ell,
\end{aligned}
\eeq
for each $w\in\mathcal{V}(t)$. Then according to Theorem 4.6 in \cite{GT2}, there exists a unique $\tilde{p}\in\mathring{H}^0(\Om)$ such that
\beq
\begin{aligned}
  &\frac\mu2\int_\Om\mathbb{D}_{\mathcal{A}^1}\tilde{u}:\mathbb{D}_{\mathcal{A}^1}wJ^1+\beta\int_{\Sigma_s}J^1(\tilde{u}\cdot\tau)(w\cdot\tau)-(\tilde{p}, \dive_{\mathcal{A}^1}w)_0+(\tilde{\eta}+\epsilon\pa_t\tilde{\eta},w\cdot\mathcal{N}^1)_{1,\Sigma}\\
  &\quad+[\tilde{u}\cdot\mathcal{N}^1,w\cdot\mathcal{N}^1]_\ell-\epsilon \mathfrak{b}(\pa_t\tilde{\eta},w\cdot\mathcal{N}^1)_\ell\\
  &=\mu\int_\Om \mathbb{D}_{\mathcal{A}^1-\mathcal{A}^2}\tilde{u}^2:\mathbb{D}_{\mathcal{A}^1}wJ^1+\int_\Om R^1\cdot wJ^1-\int_{-\ell}^\ell\sigma F^3\pa_1(w\cdot\mathcal{N}^1)+R^3\cdot w\\
  &\quad-[\tilde{u}\cdot\mathcal{N}^1, R^5+R^6]_\ell\quad,
\end{aligned}
\eeq
for each $w\in\mathcal{W}(t)$. Moreover,
\beq\label{est:tilde_p}
\|\tilde{p}\|_0\lesssim \|\tilde{u}\|_1+\|\eta\|_{3/2}\left((\|\eta^1\|_{3/2}+\|\eta^2\|_{3/2})\|\tilde{u}^2\|_{W_\delta^2}+\|\tilde{p}^2\|_{\mathring{W}_\delta^1}\right).
\eeq

Multiplying the first equation of \eqref{linear_fix1} by $\tilde{u}J^1$ and integrating by parts reveals that
\beq\label{energy_fix1}
\begin{aligned}
  &\pa_t\left(\int_{-\ell}^\ell\frac g2|\tilde{\eta}|^2+\frac\sigma2\frac{|\pa_1\tilde{\eta}|^2}{(1+|\pa_1\zeta_0|^2)^{3/2}}\right)+\epsilon\int_{-\ell}^\ell g|\pa_t\tilde{\eta}|^2+\sigma\frac{|\pa_1\pa_t\tilde{\eta}|^2}{(1+|\pa_1\zeta_0|^2)^{3/2}}\\
  &\quad+\frac\mu2\int_\Om|\mathbb{D}_{\mathcal{A}^1}\tilde{u}|^2J^1+[\tilde{u}\cdot\mathcal{N}^1]_\ell^2-\epsilon \mathfrak{b}(\tilde{u}\cdot\mathcal{N}^1)_\ell^2\\
  &=\frac\mu2\int_\Om \mathbb{D}_{\mathcal{A}^1-\mathcal{A}^2}\tilde{u}^2:\mathbb{D}_{\mathcal{A}^1}\tilde{u}J^1+\int_\Om R^1\cdot\tilde{u}J^1+\tilde{p}R^2J^1-\int_{\Sigma_s}J^1(\tilde{u}\cdot\tau)R^4\\
  &\quad-\int_{-\ell}^\ell\sigma F^3\pa_1(\tilde{u}\cdot\mathcal{N}^1)+R^3\cdot\tilde{u}-g(\tilde{\eta}+\epsilon\pa_t\tilde{\eta})R^5-\sigma\frac{\pa_1(\tilde{\eta}+\epsilon\pa_t\tilde{\eta})\pa_1R^5}{(1+|\pa_1\zeta_0|^2)^{3/2}}\\
  &\quad-[\tilde{u}\cdot\mathcal{N}^1, R^5+R^6]_\ell+\epsilon \mathfrak{b}(\tilde{u}\cdot\mathcal{N}^1, R^5)_\ell.
\end{aligned}
\eeq

We will now estimate the terms in right-hand side of \eqref{energy_fix1}.  First:
\begin{align*}
&\int_\Om R^1\cdot\tilde{u}J^1+\tilde{p}R^2J^1\\
&=\int_\Om \left(\dive_{(\mathcal{A}^1-\mathcal{A}^2)}(\mathbb{D}_{\mathcal{A}^2}\tilde{u}^2)-\nabla_{(\mathcal{A}^1-\mathcal{A}^2)}\tilde{p}^2\right)\cdot\tilde{u}J^1+\tilde{p}\dive_{(\mathcal{A}^1-\mathcal{A}^2)}\tilde{u}^2J^1\\
&\lesssim\int_\Om |\nabla\bar{\eta}|\left(|\nabla^2\bar{\eta}^2||\nabla\tilde{u}^2|+|\nabla\bar{\eta}^2||\nabla^2\tilde{u}^2|+|\nabla\tilde{p}^2|\right)|\tilde{u}|+|\tilde{p}||\nabla\bar{\eta}||\nabla\tilde{u}^2|\\
&\lesssim \left(\|\tilde{u}\|_1\|\eta^2\|_{W_\delta^{5/2}}+\|\tilde{p}\|_0\right)\|\eta\|_{3/2}\|\eta^1\|_{W_\delta^{5/2}}\|\tilde{u}^2\|_{W_\delta^2},
\end{align*}
Now we consider the integrals on $(-\ell,\ell)$. We know that $\mathcal{N}^1-\mathcal{N}^2=(-\pa_1\eta, 0)$ and
\[
\pa_1F^{3,2}=\pa_1\mathcal{R}(\pa_1\zeta_0,\pa_1\eta^2)=\pa_y\mathcal{R}\pa_1^2\zeta_0+\pa_z\mathcal{R}\pa_1^2\eta^2
\]
where $|\pa_y\mathcal{R}|\lesssim |\pa_1\eta^2|^2$, $|\pa_z\mathcal{R}|\lesssim |\pa_1\eta^2|$.
Then we take $\frac1p+\frac3q=1$, $\frac1p+\frac2r=1$, with $1<p<\frac{2}{1+\delta}$ and use H\"older inequality, Sobolev inequality and trace theory to derive that
\begin{align*}
  &\int_{-\ell}^\ell R^3\cdot\tilde{u}=\int_{-\ell}^\ell \Bigg[-\tilde{p}^2(\mathcal{N}^1-\mathcal{N}^2)+\mathbb{D}_{\mathcal{A}^2}\tilde{u}^2(\mathcal{N}^1-\mathcal{N}^2)+g(\tilde{\eta}^2+\epsilon\pa_1\pa_t\tilde{\eta}^2)(\mathcal{N}^1-\mathcal{N}^2)\\
    &\quad-\sigma\pa_1\left(\frac{\pa_1\tilde{\eta}^2+\epsilon\pa_1\pa_t\tilde{\eta}^2}{(1+|\zeta_0|^2)^{3/2}}\right)(\mathcal{N}^1-\mathcal{N}^2)-\sigma\pa_1 F^{3,2}(\mathcal{N}^1-\mathcal{N}^2)\Bigg]\cdot\tilde{u}\\
    &\lesssim \|\tilde{p}^2\|_0\|\pa_1\eta\|_{L^4}\|\tilde{u}\|_{L^4(\Sigma)}+\|\pa_1\eta^2\|_{L^q(\Sigma)}\|\nabla\tilde{u}^2\|_{L^p(\Sigma)}\|\pa_1\eta\|_{L^q(\Sigma)}\|\tilde{u}\|_{L^q(\Sigma)}\\
    &\quad+\left\|g(\tilde{\eta}^2+\epsilon\pa_1\pa_t\tilde{\eta}^2)-\sigma\pa_1\left(\frac{\pa_1\tilde{\eta}^2+\epsilon\pa_1\pa_t\tilde{\eta}^2}{(1+|\zeta_0|^2)^{3/2}}\right)\right\|_{L^p(\Sigma)}\|\pa_1\eta\|_{L^r(\Sigma)}\|\tilde{u}\|_{L^r(\Sigma)}\\
    &\quad+\|\pa_y\mathcal{R}\|_{L^2}\|\pa_1\eta\|_{L^4(\Sigma)}\|\tilde{u}\|_{L^4(\Sigma)}+\|\pa_z\mathcal{R}\|_{L^q}\|\pa_1\eta^2\|_{L^p(\Sigma)}\|\pa_1\eta\|_{L^q(\Sigma)}\|\tilde{u}\|_{L^q(\Sigma)}\\
    &\lesssim \bigg( \|\tilde{p}^2\|_0\|\pa_1\eta\|_{1/2}+\|\pa_1\eta^2\|_{1/2}\|\nabla\tilde{u}^2\|_{W_\delta^{1/2}(\Sigma)}\|\pa_1\eta\|_{1/2}+\|\tilde{\eta}^2+\epsilon\pa_t\tilde{\eta}^2\|_{W_\delta^{5/2}}\|\pa_1\eta\|_{1/2}\\
    &\quad+\|\pa_1\eta^2\|_{1/2}^2\|\pa_1\eta\|_{1/2}+\|\pa_1\eta^2\|_{1/2}\|\eta^2\|_{W_\delta^{5/2}}\|\pa_1\eta\|_{1/2}\bigg)\|\tilde{u}\|_{1/2(\Sigma)}\\
    &\lesssim \bigg( \|\tilde{p}^2\|_0\|\eta\|_{3/2}+\|\eta^2\|_{3/2}\|\tilde{u}^2\|_{W_\delta^2}\|\eta\|_{3/2}+\|\tilde{\eta}^2+\epsilon\pa_t\tilde{\eta}^2\|_{W_\delta^{5/2}}\|\eta\|_{3/2}\\
    &\quad+\|\eta^2\|_{3/2}^2\|\eta\|_{3/2}+\|\eta^2\|_{3/2}\|\eta^2\|_{W_\delta^{5/2}}\|\eta\|_{3/2}\bigg)\|\tilde{u}\|_1
\end{align*}
Similarly,
\begin{align*}
&\int_{-\ell}^\ell\sigma F^3\pa_1(\tilde{u}\cdot\mathcal{N}^1)-g(\tilde{\eta}+\epsilon\pa_t\tilde{\eta})R^5-\sigma\frac{\pa_1(\tilde{\eta}+\epsilon\pa_t\tilde{\eta})\pa_1R^5}{(1+|\pa_1\zeta_0|^2)^{3/2}}\\
&\lesssim \|\eta\|_{3/2}\left(\|\eta\|_{3/2}+\|\eta^2\|_{3/2}\right)\|\pa_t\tilde{\eta}\|_1+\|\tilde{u}\|_1\|\tilde{u}^2\|_1\|\eta\|_{3/2}\|\eta^1\|_{3/2}\\
&\quad+\|\tilde{\eta}^2+\epsilon\pa_t\tilde{\eta}^2\|_{W_\delta^{5/2}}\|\tilde{u}\|_1\|\eta\|_{3/2}+\|\eta^2\|_{3/2}\|\eta\|_{3/2}+\|\eta^2\|_{3/2}^2\|\eta\|_{3/2}\\
&\quad+\|\tilde{\eta}+\epsilon\pa_t\tilde{\eta}\|_1\left(\|\eta\|_{3/2}\|\tilde{u}^2\|_{W_\delta^2}+\|\tilde{u}^2\|_1\|\eta\|_{W_\delta^{5/2}}\right),
\end{align*}
\[
[\tilde{u}\cdot\mathcal{N}^1, R^5+R^6]_\ell-\epsilon \mathfrak{b}(\tilde{u}\cdot\mathcal{N}^1, R^5)_\ell\lesssim [\tilde{u}\cdot\mathcal{N}^1]_\ell([\pa_t\eta^1]_\ell+[\pa_t\eta^2]_\ell)[\pa_t\eta]_\ell,
\]
where  here we have used the fact that $R^5=0$ at the end points $x_1=\pm\ell$ since $u^i_1$ vanishes there, for $i=1,2$ and we denote that $u^i=(u^i_1, u^i_2)$.
Then the Cauchy-Schwarz inequality, weighted Sobolev embedding theorem and Gronwall's inequality imply that
\beq\label{difference_energy}
\begin{aligned}
&\sup_{0\le t\le T}\|\tilde{\eta}\|_1^2+\epsilon\int_0^T\|\pa_t\tilde{\eta}\|_1^2+\int_0^T\|\tilde{u}\|_1^2+[\tilde{u}\cdot\mathcal{N}^1]_\ell^2\\
&\lesssim \int_0^T\|\eta\|_{W_\delta^{5/2}}^2\left((\|\eta^1\|_{W_\delta^{5/2}}^2+\|\eta^2\|_{W_\delta^{5/2}}^2)\|\tilde{u}^2\|_{W_\delta^2}^2+\|\eta^1\|_{3/2}^2+\|\eta^2\|_{3/2}^2+\|\tilde{p}^2\|_0^2\right)\\
&\quad+\int_0^T\|\tilde{\eta}^2+\epsilon\pa_t\tilde{\eta}^2\|_{W_\delta^{5/2}}^2\|\eta\|_{3/2}^2+\|\eta\|_{W_\delta^{5/2}}^2\|\tilde{u}^2\|_{W_\delta^2}^2+\int_0^T(\|\pa_t\eta^1\|_1^2+\|\pa_t\eta^2\|_1^2)[\pa_t\eta]_\ell^2.
\end{aligned}
\eeq

From the weak formulation \eqref{weak_1} and the Theorem 4.11 in \cite{GT2},
\beq\label{difference_enhance_1}
\begin{aligned}
\|\tilde{\eta}+\epsilon\pa_t\tilde{\eta}\|_{3/2}^2\lesssim \|\tilde{u}\|_1^2+[\tilde{u}\cdot\mathcal{N}^1]_\ell^2+\|\eta\|_{3/2}^2(\|\eta^1\|_{W_\delta^{5/2}}^2\|\tilde{u}^2\|_{W_\delta^2}^2+\|\tilde{p}^2\|_{\mathring{W}_\delta^1}^2+\|\tilde{\eta}^2+\epsilon\pa_t\tilde{\eta}^2\|_{W_\delta^{5/2}}^2)\\
+\|F^3\|_{1/2}^2+[R^5+R^6]_\ell^2\\
\lesssim \|\tilde{u}\|_1^2+[\tilde{u}\cdot\mathcal{N}^1]_\ell^2+\|\eta\|_{3/2}^2(\|\eta^1\|_{W_\delta^{5/2}}^2\|\tilde{u}^2\|_{W_\delta^2}^2+\|\tilde{p}^2\|_{\mathring{W}_\delta^1}^2+\|\tilde{\eta}^2+\epsilon\pa_t\tilde{\eta}^2\|_{W_\delta^{5/2}}^2)\\
+\|\eta\|_{3/2}^2(\|\eta\|_{3/2}^2+\|\tilde{u}^2\|_{W_\delta^2}^2)+(\|\pa_t\eta^1\|_1^2+\|\pa_t\eta^2\|_1^2)[\pa_t\eta]_\ell^2.
\end{aligned}
\eeq
Since
\beq\label{def:tilde_eta}
\tilde{\eta}=\frac1\epsilon\int_0^t e^{-\frac{t-s}{\epsilon}} (\tilde{\eta}+\epsilon\pa_t\tilde{\eta}),
\eeq
thus
\[
\pa_t\tilde{\eta}=\frac1\epsilon(\tilde{\eta}+\epsilon\pa_t\tilde{\eta})-\frac{1}{\epsilon^2}\int_0^t e^{-\frac{t-s}{\epsilon}} (\tilde{\eta}+\epsilon\pa_t\tilde{\eta}),
\]
then we have
\beq\label{difference_enhance_2}
\begin{aligned}
 \|\pa_t\tilde{\eta}\|_{L^2H^{3/2}}^2&\le \frac1{\epsilon^2}\|\tilde{\eta}+\epsilon\pa_t\tilde{\eta}\|_{L^2H^{3/2}}^2+\int_0^T\left(\frac1{\epsilon^2}\int_0^t e^{-\frac{t-s}{\epsilon}} \|\tilde{\eta}+\epsilon\pa_t\tilde{\eta}\|_{3/2}\right)^2\\
 &\lesssim C(\epsilon)\|\tilde{\eta}+\epsilon\pa_t\tilde{\eta}\|_{L^2H^{3/2}}^2\lesssim C(\epsilon)(\|\eta\|_{L^\infty W_\delta^{5/2}}^2+\|[\pa_t\eta]_\ell\|_{L^2([0,T])}^2)P(\sigma).
 \end{aligned}
\eeq

From the Theorem 5.9 in \cite{GT2}, we have that
\beq\label{difference_elliptic_1}
\begin{aligned}
  &\|\tilde{u}\|_{W_\delta^2}^2+\|\tilde{p}\|_{\mathring{W}_\delta^1}^2\\
  &\lesssim \|-\mu\dive_{\mathcal{A}^1}(\mathbb{D}_{\mathcal{A}^1-\mathcal{A}^2}\tilde{u}^2)+R^1\|_{W_\delta^0}^2+\|R^2\|_{W_\delta^1}^2+\|\pa_t\tilde{\eta}-R^5\|_{W_\delta^{3/2}}^2\\
  &\quad+\|\mu\mathbb{D}_{\mathcal{A}^1-\mathcal{A}^2}\tilde{u}^2\mathcal{N}^1+R^3\|_{W_\delta^{1/2}}^2+\|\mu\mathbb{D}_{\mathcal{A}^1-\mathcal{A}^2}\tilde{u}^2\nu\cdot\tau\|_{W_\delta^{1/2}}^2\\
  &\lesssim \|\eta\|_{W_\delta^{5/2}}^2\left(\Big(1+\|\eta^1\|_{W_\delta^{5/2}}^2+\|\eta^2\|_{W_\delta^{5/2}}^2\Big)\|\tilde{u}^2\|_{W_\delta^2}^2+\|\tilde{p}^2\|_{\mathring{W}_\delta^1}^2\right)+\|\pa_t\tilde{\eta}\|_{W_\delta^{3/2}}^2\\
  &\quad+\|\eta\|_{W_\delta^{5/2}}^2\Big(\|\tilde{\eta}^2+\epsilon\pa_t\tilde{\eta}^2\|_{W_\delta^{5/2}}^2+\|\eta^1\|_{W_\delta^{5/2}}^2+\|\eta^2\|_{W_\delta^{5/2}}^2\Big),
\end{aligned}
\eeq
then the Theorem 5.10 in \cite{GT2} implies
\beq\label{difference_elliptic_2}
\begin{aligned}
  &\|\tilde{\eta}+\epsilon\pa_t\tilde{\eta}\|_{W_\delta^{5/2}}^2\\
  &\lesssim \|\tilde{u}\|_{W_\delta^2}^2+\|\tilde{p}\|_{\mathring{W}_\delta^1}^2+\|\pa_1F^3\|_{W_\delta^{1/2}}^2+\|\mu\mathbb{D}_{\mathcal{A}^1-\mathcal{A}^2}\tilde{u}^2\mathcal{N}^1+R^3\|_{W_\delta^{1/2}}^2+[F^3]_\ell^2+[R^6]_\ell^2\\
  &\lesssim \|\eta\|_{W_\delta^{5/2}}^2\left(\Big(1+\|\eta^1\|_{W_\delta^{5/2}}^2+\|\eta^2\|_{W_\delta^{5/2}}^2\Big)\|\tilde{u}^2\|_{W_\delta^2}^2+\|\tilde{p}^2\|_{\mathring{W}_\delta^1}^2\right)+\|\pa_t\tilde{\eta}\|_{W_\delta^{3/2}}^2\\
  &\quad+\|\eta\|_{W_\delta^{5/2}}^2\Big(\|\tilde{\eta}^2+\epsilon\pa_t\tilde{\eta}^2\|_{W_\delta^{5/2}}^2+\|\eta^1\|_{W_\delta^{5/2}}^2+\|\eta^2\|_{W_\delta^{5/2}}^2\Big).
\end{aligned}
\eeq
Combining \eqref{difference_energy}--\eqref{difference_enhance_2}, the Cauchy-Schwarz inequality, weighted Sobolev embeddings, and the linear estimates of Theorem \ref{thm:linear_low} and \ref{thm:higher order} then show that
\beq\label{est:contraction_1}
\begin{aligned}
\|\tilde{\eta}\|_{L^\infty W_\delta^{5/2}}^2&\le C(\epsilon)T\|\tilde{\eta}+\epsilon\pa_t\tilde{\eta}\|_{L^2W_\delta^{5/2}}^2\\
&\lesssim C(\epsilon)TP(\|\eta^1\|_{L^\infty W_\delta^{5/2}}^2,\|\eta^2\|_{L^\infty W_\delta^{5/2}}^2,\|\tilde{u}^2\|_{L^2W_\delta^2}^2,\|\tilde{p}^2\|_{L^2\mathring{W}_\delta^1}^2)\\
&\quad\times\left(\|\eta\|_{L^\infty W_\delta^{5/2}}^2+\|[\pa_t\eta]_\ell\|_{L^2([0,T])}^2\right)\\
&\lesssim C(\epsilon)TP(\sigma)\left(\|\eta\|_{L^\infty W_\delta^{5/2}}^2+\|[\pa_t\eta]_\ell\|_{L^2([0,T])}^2\right),
\end{aligned}
\eeq
where the first inequality is obtained by \eqref{def:tilde_eta} and the second inequality used the fact that $\|\tilde{\eta}^2+\epsilon\pa_t\tilde{\eta}^2\|_{L^2 W_\delta^{5/2}}^2\le \sigma$ which is included in the proof of Theorem \ref{thm:linear_low}.

Then \eqref{est:tilde_p} and \eqref{difference_energy} imply that
\beq\label{est:contraction_2}
\begin{aligned}
  &\|\tilde{u}\|_{L^2H^1}^2+\|\tilde{p}\|_{L^2\mathring{H}^0}^2+\|[\tilde{u}\cdot{N}^1]_\ell\|_{L^2([0,T])}^2\\
  &\lesssim P(\sigma)\left(\|\eta\|_{L^\infty W_\delta^{5/2}}^2+\|[\pa_t\eta]_\ell\|_{L^2([0,T])}^2\right).
\end{aligned}
\eeq
Since at the corner points, $[\pa_t\tilde{\eta}]_\ell=[\tilde{u}\cdot{N}^1]_\ell$,
\eqref{est:contraction_1} and \eqref{est:contraction_2} reveals that
\beq\label{est:contraction_3}
\begin{aligned}
  &\|\tilde{u}\|_{L^2H^1}^2+\|\tilde{p}\|_{L^2\mathring{H}^0}^2+\|[\pa_t\tilde{\eta}]_\ell\|_{L^2([0,T])}^2+\|\tilde{\eta}\|_{L^\infty W_\delta^{5/2}}^2\\
  &\le (C(\epsilon)T+C)P(\sigma)\left(\|\eta\|_{L^\infty W_\delta^{5/2}}^2+\|[\pa_t\eta]_\ell\|_{L^2([0,T])}^2\right),
\end{aligned}
\eeq
where $C$ is a universal constant independent of $\epsilon$.


We may restrict $\sigma$ such that $CP(\sigma)\le 1/8$. For each $0<\epsilon\le 1/(8CP(\sigma))$, we choose $T^\prime>0$ such that $C(\epsilon)T^\prime P(\sigma)\le 1/8$. This implies
\beq
d(A(u^1,p^1,\eta^1),A(u^2,p^2,\eta^2))=d((\tilde{u}^1, \tilde{p}^1,\tilde{\eta}^1),(\tilde{u}^2, \tilde{p}^2,\tilde{\eta}^2))\le\frac12 d((u^1,p^1,\eta^1),(u^2,p^2,\eta^2)).
\eeq
If $0<T^\prime<T_\epsilon$, we can repeat the above argument on intervals $[0, T^\prime]$, $[T^\prime, 2 T^\prime]$,etc. Finally we see that $A$ is a strict contraction on $S(T_\epsilon,\sigma)$. Since the metric space $S(T_\epsilon, \sigma)$ is complete, the contraction mapping principle reveals the existence of a unique $(u,p,\eta)\in S(T_\epsilon, \sigma)$ such that $A(u,p,\eta)=(\tilde{u}, \tilde{p},\tilde{\eta})=(u,p,\eta)$.
\end{proof}

\subsection{Energy estimates}

We want to send $\epsilon\rightarrow0$ to get a uniform $T>0$ independent of $\epsilon$, so we need some uniform estimates. For simplicity, we may abuse the same symbol of energy and dissipation in section 2.1 of \cite{GT2} and still denote the unknown $(u^\epsilon,p^\epsilon,\eta^\epsilon)$ as $(u,p,\eta)$.
\begin{theorem}\label{thm:uniform_energy}
  There exists a universal constant $C$ and a universal $T>0$ independent of $\epsilon$ such that for each $\epsilon>0$ sufficiently small,
  \beq
  \sup_{0\le t\le T}\mathcal{E}(t)+\int_0^T\mathcal{D}(t)\,\mathrm{d}t\le C.
  \eeq
\end{theorem}
\begin{proof}
We shall use the continuity argument to prove the uniform bounds.
First, we define some variants of energy, dissipation and forcing terms.
\beq
\mathscr{E}:=\sum_{j=0}^2\int_{-\ell}^\ell\frac g2|\pa_t^j\eta|^2+\frac\sigma2\frac{|\pa_1\pa_t^j\eta|^2}{(1+|\pa_1\zeta_0|^2)^{3/2}},
\eeq
\beq
\mathscr{D}:=\sum_{j=0}^2\left(\frac\mu2\int_\Om\left|\mathbb{D}_{\mathcal{A}}\pa_t^ju\right|^2J+\beta\int_{\Sigma_s}\left|\pa_t^ju\cdot\tau\right|^2J+\left[\pa_t^ju\cdot\mathcal{N}\right]_\ell^2\right),
\eeq
and
\beq
\mathscr{F}:=\int_{-\ell}^{\ell}\left[\sigma\mathcal{Q}(\pa_1\zeta_0,\pa_1\eta)+\sigma\pa_z\mathcal{R}(\pa_1\zeta_0,\pa_1\eta)\frac{|\pa_1\pa_t^2\eta|^2}{2}+\sigma\pa_z^2\mathcal{R}(\pa_1\zeta_0,\pa_1\eta)(\pa_1\pa_t\eta)^2\pa_1\pa_t^2\eta\right].
\eeq

Suppose that
\[
\sup_{0<s\le t}\mathcal{E}(s)+\int_0^t\mathcal{D}\le \alpha\ \text{for each}\ t\in [0,T),
\]
where $\alpha>0$ is sufficiently small and $0<T<1$ is to be determined later.
Similar to the energy estimate and section 8 of \cite{GT2}, we can derive that
\beq\label{uniform_est_1}
\frac{d}{dt}(\mathscr{E}+\mathscr{F})+\mathscr{D}+\epsilon\left(\|\pa_t\eta\|_{1,\Sigma}^2+\|\pa_t^2\eta\|_{1,\Sigma}^2+\|\pa_t^3\eta\|_{1,\Sigma}^2\right)\lesssim \sqrt{\mathcal{E}}\mathcal{D}.
\eeq
Then in order to follow the proof of Theorem 8.2 in \cite{GT2}, we need to prove the uniform bounds of $\|\eta\|_{3/2}^2+\|\pa_t\eta\|_{3/2}^2+\|\pa_t^2\eta\|_{3/2}^2$ independent of $\epsilon$ . First, by following the proof of Theorem 8.2 in \cite{GT2}, we have known that
 \beq
 \|\eta+\epsilon\pa_t\eta\|_{3/2}^2+\|\pa_t\eta+\epsilon\pa_t^2\eta\|_{3/2}^2+\|\pa_t^2\eta+\epsilon\pa_t^3\eta\|_{3/2}^2\lesssim\bar{\mathcal{D}}_{\shortparallel}+\sqrt{\mathcal{E}}\mathcal{D}
 \eeq
 and
\beq\label{uniform_eta_1}
\begin{aligned}
\|\eta\|_{3/2}^2+\|\pa_t\eta\|_{3/2}^2\lesssim (\|\eta_0\|_{3/2}^2+\|\pa_t\eta(0)\|_{3/2}^2)e^{-2\frac t\epsilon}
+\left(\frac1\epsilon\int_0^te^{-\frac{t-s}{\epsilon}}(\bar{\mathcal{D}}_{\shortparallel}+\sqrt{\mathcal{E}}\mathcal{D})^{1/2}\right)^2,
\end{aligned}
\eeq
Then we denote $\vartheta=\pa_t^2\eta+\epsilon\pa_t^3\eta$ and the extension $\bar{\vartheta}=\pa_t^2\bar{\eta}+\epsilon\pa_t^3\bar{\eta}$, then the standard calculation and trace theory reveals that,
\beq
\epsilon\frac{d}{dt}\|\pa_t^2\bar{\eta}\|_{2}+\|\pa_t^2\bar{\eta}\|_{2}\le\|\bar{\vartheta}\|_{2}\lesssim\|\vartheta\|_{3/2}.
\eeq
This implies that
\beq
\begin{aligned}
&\epsilon^2\|\pa_t^2\eta\|_{L^\infty H^{3/2}}^2\le t\int_0^t\|\vartheta\|_{3/2}^2+\epsilon^2\|\pa_t^2\eta(0)\|_{3/2}^2\\
&\le t\int_0^t\|\vartheta\|_{3/2}^2+\|\pa_t\eta(0)\|_{3/2}^2+\|\pa_t\eta(0)+\epsilon\pa_t^2\eta(0)\|_{3/2}^2\lesssim t\int_0^t\|\vartheta\|_{3/2}^2+\mathfrak{E}_0,
\end{aligned}
\eeq
which also implies
\beq\label{uniform_eta_2}
\int_0^t\|\pa_t^2\eta\|_{3/2}^2\le \int_0^t\|\vartheta\|_{3/2}^2+t^2\epsilon^2\|\pa_t^2\eta\|_{L^\infty H^{3/2}}\lesssim (1+t^3)\int_0^t\|\vartheta\|_{3/2}^2+t^2\mathfrak{E}_0.
\eeq

Then following the proof of Theorem 8.2 in \cite{GT2} together with \eqref{uniform_eta_1} and \eqref{uniform_eta_2}, for $t\le T<1$, we may derive that
\beq
\int_0^t\mathcal{D}_{\shortparallel}\lesssim \int_0^t(\mathscr{D}+\sqrt{\mathcal{E}}\mathcal{D})+\mathfrak{E}_0,
\eeq
which reveals
\beq\label{uniform_est_2}
\int_0^t\mathcal{D}\lesssim \int_0^t(\mathscr{D}+\sqrt{\mathcal{E}}\mathcal{D})+\mathfrak{E}_0
\eeq
after similar estimate for $\|\pa_t\eta\|_{W_\delta^{5/2}}^2$ derived from $\|\pa_t\eta+\epsilon\pa_t^2\eta\|_{W_\delta^{5/2}}$.
Then similar to the proof of Theorem 8.4 in \cite{GT2}, combining \eqref{uniform_est_1} and \eqref{uniform_est_2}, we have \beq
\sup_{0<s\le t}\mathcal{E}(s)+\int_0^t\mathcal{D}\le C\mathcal{E}(0)\le C'\mathfrak{E}_0,
\eeq
for each $t\in[0,T]$, and the second inequality follows from the initial data in section \ref{section_initial}.
Restricting the initial data implies that
\beq
\sup_{0<s\le t}\mathcal{E}(s)+\int_0^t\mathcal{D}\le\frac\alpha2,
\eeq
for each $t\in[0,T]$.
\end{proof}

\subsection{Existence of solutions}

In this section, we consider the solution of original problem \eqref{eq:geometric perturbation}.
\begin{theorem}
  There exists a solution $(u,p,\eta)\in \mathcal{X}\cap\mathcal{Y}$ solving the equation \eqref{eq:geometric perturbation}.
\end{theorem}
\begin{proof}
According to the energy estimate in Theorem \ref{thm:uniform_energy}, there exists a sequence $\epsilon_k$ tends to zero and a pair $(u,p,\eta)$ such that $(u,p,\eta)\in \mathcal{X}\cap\mathcal{Y}$ with
\begin{equation}\label{convergence}
  \left\{
  \begin{aligned}
  &(u^{\epsilon_k},p^{\epsilon_k},\eta^{\epsilon_k})\stackrel{\ast}\rightharpoonup (u,p,\eta)\ &\text{weakly-}\ast\ \text{in}\ \mathcal{X},\\
  &(u^{\epsilon_k},p^{\epsilon_k},\eta^{\epsilon_k})\rightharpoonup (u,p,\eta)\ &\text{weakly in}\ \mathcal{Y}.
  \end{aligned}
  \right.
\end{equation}
Choose a function $w\in\mathcal{W}$, then from the weak formulation, we deduce that
\beq
\begin{aligned}
\int_0^T\int_{-\ell}^\ell g\eta^{\epsilon_k}(w\cdot\mathcal{N}^{\epsilon_k})+\sigma\frac{\pa_1\eta^\epsilon_k\pa_1(w\cdot\mathcal{N}^{\epsilon_k})}{(1+|\pa_1\zeta_0|^2)^{3/2}}+\epsilon_k\int_0^T\int_{-\ell}^\ell g\pa_t\eta^{\epsilon_k}(w\cdot\mathcal{N}^{\epsilon_k})+\sigma\frac{\pa_1\pa_t\eta^\epsilon_k\pa_1(w\cdot\mathcal{N}^{\epsilon_k})}{(1+|\pa_1\zeta_0|^2)^{3/2}}\\
+\int_0^T\int_{\Om}\frac\mu2\mathbb{D}_{\mathcal{A}^{\epsilon_k}}u^{\epsilon_k}:\mathbb{D}_{\mathcal{A}^{\epsilon_k}}wJ^{\epsilon_k}+\int_0^T\int_{\Sigma_s}\beta(u^{\epsilon_k}\cdot\tau)(w\cdot\tau)J^{\epsilon_k}-\int_0^T\int_{\Om}p^{\epsilon_k}\dive_{\mathcal{A}^{\epsilon_k}}wJ^{\epsilon_k}\\
+[u^{\epsilon_k}\cdot\mathcal{N}^{\epsilon_k}+\hat{\mathscr{W}}(u^{\epsilon_k}\cdot\mathcal{N}^{\epsilon_k}),w\cdot\mathcal{N}^{\epsilon_k}]-\epsilon \mathfrak{b}(\pa_t\xi^{\epsilon_k},w\cdot\mathcal{N}^{\epsilon_k})_\ell=-\sigma\int_0^T\int_{-\ell}^\ell \mathcal{R}(\pa_1\zeta_0,\pa_1\eta^\epsilon_k)\pa_1(w\cdot\mathcal{N}^{\epsilon_k}).
\end{aligned}
\eeq
Passing the limit $\epsilon_k\rightarrow0$, the convergence \eqref{convergence} reveals that
\beq
\begin{aligned}
\int_0^T\int_{-\ell}^\ell g\eta(w\cdot\mathcal{N})+\sigma\frac{\pa_1\eta\pa_1(w\cdot\mathcal{N})}{(1+|\pa_1\zeta_0|^2)^{3/2}}
+\int_0^T\int_{\Om}\frac\mu2\mathbb{D}_{\mathcal{A}}u:\mathbb{D}_{\mathcal{A}}wJ+\int_0^T\int_{\Sigma_s}\beta(u\cdot\tau)(w\cdot\tau)J\\
-\int_0^T\int_{\Om}p\dive_{\mathcal{A}}wJ
+[u\cdot\mathcal{N}+\hat{\mathscr{W}}(u\cdot\mathcal{N}),w\cdot\mathcal{N}]_\ell=-\sigma\int_0^T\int_{-\ell}^\ell \mathcal{R}(\pa_1\zeta_0,\pa_1\eta)\pa_1(w\cdot\mathcal{N}).
\end{aligned}
\eeq
Thus the limit $(u,p,\eta)$ is a weak solution of \eqref{eq:geometric perturbation}.
Then integrating by parts,
\beq
\begin{aligned}
\int_0^T\int_{-\ell}^\ell g\eta(w\cdot\mathcal{N})-\sigma\pa_1\left(\frac{\pa_1\eta}{(1+|\pa_1\zeta_0|^2)^{3/2}}+\mathcal{R}(\pa_1\zeta_0,\pa_1\eta)\right)w\cdot\mathcal{N}
-\int_0^T\int_{\Om}\mu(\Delta_{\mathcal{A}}u)wJ\\
+\int_0^T\int_{-\ell}^\ell\mu\mathbb{D}_{\mathcal{A}}u\mathcal{N}\cdot w+\int_0^T\int_{\Sigma_s}\mu\mathbb{D}_{\mathcal{A}}u\nu\cdot w+\beta(u\cdot\tau)(w\cdot\tau)J+\int_0^T\int_{\Om}\nabla_{\mathcal{A}}p\cdot wJ\\
-\int_0^T\int_{-\ell}^\ell p\mathcal{N}\cdot w
-\int_0^T\int_{\Sigma_s} p\nu\cdot wJ+\left[\sigma\left(\frac{\pa_1\eta}{(1+|\pa_1\zeta_0|^2)^{3/2}}+\mathcal{R}(\pa_1\zeta_0,\pa_1\eta)\right),w\cdot\mathcal{N}\right]_\ell\\
+[u\cdot\mathcal{N}+\hat{\mathscr{W}}(u\cdot\mathcal{N}),w\cdot\mathcal{N}]_\ell=0,
\end{aligned}
\eeq
we know that $(u,p,\eta)$ satisfy the boundary condition of \eqref{eq:geometric perturbation}.

In the following, we show that  $(u,p,\eta)$ achieves the initial data in Section \ref{subsection_1}. We take $t=0$ for \eqref{eq:modified_geometric perturbation} to derive the weak formulation
\beq
\begin{aligned}
((u_0^\epsilon,v))-(p_0^\epsilon,\dive_{\mathcal{A}(0)}v)_{\mathcal{H}^0}+(\eta_0+\epsilon\pa_t\eta(0),v\cdot\mathcal{N}(0))_{1,\Sigma}+[u_0^\epsilon\cdot\mathcal{N}(0),v\cdot\mathcal{N}(0)]_{\ell}\\
=-\int_{-\ell}^{\ell}\sigma F^3(0)\pa_1(v\cdot\mathcal{N}(0))-[v\cdot\mathcal{N}(0),\hat{\mathscr{W}}(\pa_t\eta(0))]_\ell+\epsilon \mathfrak{b}(\pa_t\eta(0),v\cdot\mathcal{N}(0))_\ell,
\end{aligned}
\eeq
for each $v\in\mathcal{W}$. Since the boundedness of $u_0^\epsilon$ and $p_0^\epsilon$, we extract a subsequence $\epsilon_k$ such that when $\epsilon_k\to 0$,
\[
u_0^{\epsilon_k}\rightharpoonup \varphi\ \text{in}\ W_\delta^2(\Om)\cap\mathcal{V}(0),\quad p_0^{\epsilon_k}\rightharpoonup\psi\ \text{in}\ \mathring{W}_\delta^1(\Om),
\]
and
\beq
\begin{aligned}
((\varphi,v))-(\psi,\dive_{\mathcal{A}(0)}v)_{\mathcal{H}^0}+(\eta_0,v\cdot\mathcal{N}(0))_{1,\Sigma}+[\varphi\cdot\mathcal{N}(0),v\cdot\mathcal{N}(0)]_{\ell}\\
=-\int_{-\ell}^{\ell}\sigma F^3(0)\pa_1(v\cdot\mathcal{N}(0))-[v\cdot\mathcal{N}(0),\hat{\mathscr{W}}(\pa_t\eta(0))]_\ell,
\end{aligned}
\eeq
which is exactly the same weak formulation of \eqref{eq:geometric perturbation} when $t=0$. We then employ the uniqueness for \eqref{eq:geometric perturbation} when $t=0$ to derive that $\varphi=u_0$ and $\psi=p_0$. Similarly, we could derive that
\[
D_tu^\epsilon(0)\rightharpoonup D_tu(0)\ \text{in}\ H^1(\Om),\quad \pa_tp^\epsilon(0)\rightharpoonup\pa_tp(0)\ \text{in}\ \mathring{H}^0(\Om).
\]

Thus $(u,p,\eta)$ is a strong solution of \eqref{eq:geometric perturbation} because of its regularity.
\end{proof}

\subsection{Uniqueness}

We refer to velocities as $u^j$, pressures as $p^j$, surface functions as $\eta^j$, for $j=1,2$.
\begin{theorem}
Let $u^1$, $u^2$, $p^1$, $p^2$ and $\eta^1$, $\eta^2$ satisfy
\beq
\sup_{0\le t\le T}\{\mathcal{E}(u^1,p^1,\eta^1), \mathcal{E}(u^2,p^2,\eta^2)\}<\varepsilon,\quad
\text{and}\ \int_0^T\{\mathcal{D}(u^1,p^1,\eta^1), \mathcal{D}(u^2,p^2,\eta^2)\}<\varepsilon,
\eeq
with $T>0$.
  Suppose that for $j=1, 2$,
  \beq\label{eq:contraction}
  \left\{
  \begin{aligned}
    &-\mu\Delta_{\mathcal{A}^j}u^j+\nabla_{\mathcal{A}^j}p^j=0,\quad &\text{in}&\quad\Om,\\
    &\dive_{\mathcal{A}^j}u^j=0,\quad &\text{in}&\quad\Om,\\
    &S_{\mathcal{A}^j}(p^j,u^j)\mathcal{N}^j=g\eta^j\mathcal{N}^j-\sigma\pa_1\left(\frac{\pa_1\eta^j}{1+|\pa_1\zeta_0|^2}+F^{3,j}\right)\mathcal{N}^j,\quad&\text{on}&\quad\Sigma,\\
    &\left(S_{\mathcal{A}^j}(p^j,u^j)\nu-\beta u^j\right)\cdot\tau=0,\quad&\text{on}&\quad\Sigma_s,\\
    &u^j\cdot\nu=0,\quad&\text{on}&\quad\Sigma_s,\\
    &\pa_t\eta^j=u^j\cdot\mathcal{N}^j,\quad&\text{on}&\quad\Sigma,\\
    &\kappa\pa_t\eta^j(\pm\ell,t)+\kappa\hat{\mathscr{W}}(\pa_t\eta^j(\pm\ell,t))=\mp\sigma\left(\frac{\pa_1\eta^j}{(1+|\zeta_0|^2)^{3/2}}+F^{3,j}\right)(\pm\ell,t).
  \end{aligned}
  \right.
  \eeq
  where $\mathcal{A}^j$, $\mathcal{N}^j$, $F^{3,j}$ are determined by $\eta^j$ as usual. Suppose that $u^1(0)=u^2(0)$, $p^1(0)=p^2(0)$ and $\pa_t^k\eta^1(0)=\pa_t^k\eta^2(0)$ for $k=0, 1$.

  Then there exist $\varepsilon_1>0$, $T_1>0$ such that if $0<\varepsilon\le \varepsilon_1$ and $0<T\le T_1$, then
  \beq
  u^1=u^2,\quad p^1=p^2,\quad \eta^1=\eta^2.
  \eeq
\end{theorem}
\begin{proof}
  First, we define $v=u^1-u^2$, $q=p^1-p^2$, $\theta=\eta^1-\eta^2$ and derive the PDEs satisfied by $v$, $q$, $\theta$. We still use $F^3$ to denote $F^3=F^{3,1}-F^{3,2}$.

  Step 1 -- PDEs and energy for differences.

  Subtracting equations in \eqref{eq:contraction} with $j=2$ from the same equations with $j=1$, we can write the resulting equations in terms of $v$, $q$, $\theta$ as
  \beq\label{eq:difference_1}
  \left\{
  \begin{aligned}
    &\dive_{\mathcal{A}^1}S_{\mathcal{A}^1}(q,v)=\mu\dive_{\mathcal{A}^1}\left(\mathbb{D}_{(\mathcal{A}^1-\mathcal{A}^2)}u^2\right)+H^1,\quad&\text{in}&\quad\Om,\\
    &\dive_{\mathcal{A}^1}v=H^2,\quad&\text{in}&\quad\Om,\\
    &S_{\mathcal{A}^1}(q,v)\mathcal{N}^1=\mu\mathbb{D}_{(\mathcal{A}^1-\mathcal{A}^2)}u^2\mathcal{N}^1+g\theta\mathcal{N}^1-\sigma\pa_1\left(\frac{\pa_1\theta}{(1+|\pa_1\zeta_0|^2)^{3/2}}\right)\mathcal{N}^1\\
    &\quad\quad-\sigma\pa_1F^3\mathcal{N}^1+H^3,\quad&\text{on}&\quad\Sigma,\\
    &(S_{\mathcal{A}^1}(q,v)\nu-\beta v)\cdot\tau=\mu\mathbb{D}_{(\mathcal{A}^1-\mathcal{A}^2)}u^2\nu\cdot\tau,\quad&\text{on}&\quad\Sigma_s,\\
    &v\cdot\nu=0,\quad&\text{on}&\quad\Sigma_s,\\
    &\pa_t\theta=v\cdot\mathcal{N}^1+H^5,\quad&\text{on}&\quad\Sigma,\\
    &\kappa\pa_t\theta(\pm\ell,t)=\mp\sigma\frac{\pa_1\theta}{(1+|\zeta_0|^2)^{3/2}}(\pm\ell,t)\mp F^3-H^6,\\
    &v(t=0)=0,\quad\theta(t=0)=0.
  \end{aligned}
  \right.
  \eeq
  where $H^1$, $H^2$, $H^3$, $H^4$, $H^5$, $H^6$ are defined by
  \begin{align*}
    H^1&=\mu\dive_{(\mathcal{A}^1-\mathcal{A}^2)}(\mathbb{D}_{\mathcal{A}^2}u^2)-\nabla_{(\mathcal{A}^1-\mathcal{A}^2)}p^2,\\
    H^2&=-\dive_{(\mathcal{A}^1-\mathcal{A}^2)}u^2,\\
    H^3&=-p^2(\mathcal{N}^1-\mathcal{N}^2)+\mathbb{D}_{\mathcal{A}^1}u^2(\mathcal{N}^1-\mathcal{N}^2)-\mathbb{D}_{(\mathcal{A}^1-\mathcal{A}^2)}u^2\mathcal{N}^2+g\eta^2(\mathcal{N}^1-\mathcal{N}^2)\\
    &\quad-\sigma\pa_1\left(\frac{\pa_1\eta^2}{(1+|\zeta_0|^2)^{3/2}}\right)(\mathcal{N}^1-\mathcal{N}^2)-\sigma\pa_1 F^{3,2}(\mathcal{N}^1-\mathcal{N}^2),\\
    H^5&=u^2\cdot(\mathcal{N}^1-\mathcal{N}^2),\\
    H^6&=\kappa(\hat{\mathscr{W}}(\pa_t\eta^1(\pm\ell,t))-\hat{\mathscr{W}}(\pa_t\eta^2(\pm\ell,t))).
  \end{align*}
  The solutions are sufficiently regular for us to differentiate \eqref{eq:difference_1} in time, which results in the equations
  \beq\label{eq:difference_2}
  \left\{
  \begin{aligned}
    &\dive_{\mathcal{A}^1}S_{\mathcal{A}^1}(\pa_tq,\pa_tv)=\mu\dive_{\mathcal{A}^1}\left(\mathbb{D}_{(\pa_t\mathcal{A}^1-\pa_t\mathcal{A}^2)}u^2\right)+\tilde{H}^1,\quad&\text{in}&\quad\Om,\\
    &\dive_{\mathcal{A}^1}\pa_tv=\tilde{H}^2,\quad&\text{in}&\quad\Om,\\
    &S_{\mathcal{A}^1}(\pa_tq,\pa_tv)\mathcal{N}^1=\mu\mathbb{D}_{(\pa_t\mathcal{A}^1-\pa_t\mathcal{A}^2)}u^2\mathcal{N}^1+g\pa_t\theta\mathcal{N}^1-\sigma\pa_1\left(\frac{\pa_1\pa_t\theta}{(1+|\pa_1\zeta_0|^2)^{3/2}}\right)\mathcal{N}^1\\
    &\quad\quad-\sigma\pa_1\pa_t(F^{3,1}-F^{3,2})\mathcal{N}^1+\tilde{H}^3,\quad&\text{on}&\quad\Sigma,\\
    &(S_{\mathcal{A}^1}(\pa_tq,\pa_tv)\nu-\beta \pa_tv)\cdot\tau=    \mu\mathbb{D}_{(\pa_t\mathcal{A}^1-\pa_t\mathcal{A}^2)}u^2\nu\cdot\tau+\tilde{H}^4,\quad&\text{on}&\quad\Sigma_s,\\
    &\pa_tv\cdot\nu=0,\quad&\text{on}&\quad\Sigma_s,\\
    &\pa_t^2\theta=\pa_tv\cdot\mathcal{N}^1+\tilde{H}^5,\quad&\text{on}&\quad\Sigma,\\
    &\kappa\pa_t^2\theta(\pm\ell,t)=\mp\sigma\frac{\pa_1\pa_t\theta}{(1+|\zeta_0|^2)^{3/2}}(\pm\ell,t)\mp \tilde{H}^6,\\
    &\pa_tv(t=0)=0,\quad\pa_t\theta(t=0)=0,
  \end{aligned}
  \right.
  \eeq
  where
  \begin{align*}
    \tilde{H}^1&=\pa_tH^1+\dive_{\pa_t\mathcal{A}^1}(\mathbb{D}_{(\mathcal{A}^1-\mathcal{A}^2)}u^2)+\dive_{\mathcal{A}^1}(\mathbb{D}_{(\mathcal{A}^1-\mathcal{A}^2)}\pa_t u^2)+\dive_{\pa_t\mathcal{A}^1}(\mathbb{D}_{\mathcal{A}^1}v)\\
    &\quad+\dive_{\mathcal{A}^1}(\mathbb{D}_{\pa_t\mathcal{A}^1}v)-\nabla_{\pa_t\mathcal{A}^1}q,\\
    \tilde{H}^2&=\pa_tH^2-\dive_{\pa_t\mathcal{A}^1}v,\\
    \tilde{H}^3&=\pa_tH^3+\mathbb{D}_{(\mathcal{A}^1-\mathcal{A}^2)}\pa_tu^2\mathcal{N}^1+\mathbb{D}_{(\mathcal{A}^1-\mathcal{A}^2)}u^2\pa_t\mathcal{N}^1-S_{\mathcal{A}^1}(q,v)\pa_t\mathcal{N}^1+\mathbb{D}_{\pa_t\mathcal{A}^1}v\mathcal{N}^1\\
    &\quad+g\theta\pa_t\mathcal{N}^1-\sigma\pa_1\left(\frac{\pa_1\theta}{(1+|\pa_1\zeta_0|^2)^{3/2}}\right)\pa_t\mathcal{N}^1,\\
    \tilde{H}^4&=    \mu\mathbb{D}_{(\mathcal{A}^1-\mathcal{A}^2)}\pa_tu^2\nu\cdot\tau+\mathbb{D}_{\pa_t\mathcal{A}^1}v\nu\cdot\tau,\\
    \tilde{H}^5&=\pa_tH^5+v\cdot\pa_t\mathcal{N}^1,\\
    \tilde{H}^6&=\pa_tH^6.
  \end{align*}

  Now we multiply \eqref{eq:difference_2} by $J^1\pa_tu^1$, integrate over $\Om$ and integrate by parts to deduce that
  \beq\label{eq:energy_1}
  \begin{aligned}
    &\pa_t\left(\int_{-\ell}^\ell\frac{g}{2}|\pa_t\theta|^2+\frac{\sigma}{2}\frac{|\pa_1\pa_t\theta|^2}{(1+|\pa_1\zeta_0|^2)^{3/2}}\right)+\frac{\mu}{2}\int_\Om|\mathbb{D}_{\mathcal{A}^1}\pa_tv|^2J^1+\beta\int_{\Sigma_s}J^1|\pa_tv\cdot\tau|^2+[\pa_tv\cdot\mathcal{N}^1]_\ell^2\\
    &=\int_\Om\mu\dive_{\mathcal{A}^1}(\mathbb{D}_{(\pa_t\mathcal{A}^1-\pa_t\mathcal{A}^2)}u^2)\cdot\pa_tvJ^1+\tilde{H}^1\cdot\pa_tvJ^1+\pa_tq\tilde{H}^2J^1-\int_{\Sigma_s}J^1(\pa_tv\cdot\tau)\tilde{H}^4\\
    &\quad-\int_{-\ell}^\ell\sigma \pa_tF^3\pa_1(\pa_tv\cdot\mathcal{N}^1)+(\mathbb{D}_{(\pa_t\mathcal{A}^1-\pa_t\mathcal{A}^2)}u^2\mathcal{N}^1+\tilde{H}^3)\cdot \pa_tv-g\pa_t\theta\tilde{H}^5-\sigma\frac{\pa_1\pa_t\theta\pa_1\tilde{H}^5}{(1+|\pa_1\zeta_0|^2)^{3/2}}\\
    &\quad-\int_{\Sigma_s}J^1(\pa_tv\cdot\tau)\mu\mathbb{D}_{(\pa_t\mathcal{A}^1-\pa_t\mathcal{A}^2)}u^2\nu\cdot\tau-\left[\pa_tv\cdot\mathcal{N}^1, \tilde{H}^5+\tilde{H}^6\right]_\ell.
  \end{aligned}
  \eeq
  Here we notice that
  \beq
  \sum_{a=\pm1}\kappa(\pa_tv\cdot\mathcal{N}^1)(a\ell)\tilde{H}^5(a\ell)=0,
  \eeq
  since $v_1^1=v_1^2=0$ at the endpoints $x_1=\pm\ell$, where we denote that $u^1=(v_1^1, v_2^1)$ and $u^2=(v_1^2, v_2^2)$.

  Another integration by parts reveals that
  \beq\label{eq:energy_2}
  \begin{aligned}
  \int_\Om\mu\dive_{\mathcal{A}^1}(\mathbb{D}_{(\pa_t\mathcal{A}^1-\pa_t\mathcal{A}^2)}u^2)\cdot\pa_tvJ^1=-\frac{\mu}{2}\int_\Om J^1\mathbb{D}_{(\pa_t\mathcal{A}^1-\pa_t\mathcal{A}^2)}u^2:\mathbb{D}_{\mathcal{A}^1}\pa_tv\\
  +\int_{-\ell}^\ell\mathbb{D}_{(\pa_t\mathcal{A}^1-\pa_t\mathcal{A}^2)}u^2\mathcal{N}^1\cdot\pa_tv+\int_{\Sigma_s}\mathbb{D}_{(\pa_t\mathcal{A}^1-\pa_t\mathcal{A}^2)}u^2\nu\cdot\pa_tv J^1.
  \end{aligned}
  \eeq
  We combine \eqref{eq:energy_1} and \eqref{eq:energy_2}, and then integrate in time from $0$ to $t<T$ to derive that
  \beq\label{eq:energy_3}
  \begin{aligned}
    &\int_{-\ell}^\ell\frac{g}{2}|\pa_t\theta|^2+\frac{\sigma}{2}\frac{|\pa_1\pa_t\theta|^2}{(1+|\pa_1\zeta_0|^2)^{3/2}}+\frac{\mu}{2}\int_0^t\int_\Om|\mathbb{D}_{\mathcal{A}^1}\pa_tv|^2J^1+\beta\int_0^t\int_{\Sigma_s}J^1|\pa_tv\cdot\tau|^2+\int_0^t[\pa_tv\cdot\mathcal{N}^1]_\ell^2\\
    &=-\frac{\mu}{2}\int_\Om J^1\mathbb{D}_{(\pa_t\mathcal{A}^1-\pa_t\mathcal{A}^2)}u^2:\mathbb{D}_{\mathcal{A}^1}\pa_tv+\int_0^t\int_\Om\tilde{H}^1\cdot\pa_tvJ^1+\pa_tq\tilde{H}^2J^1-\int_0^t\int_{\Sigma_s}J^1(\pa_tv\cdot\tau)\tilde{H}^4\\
    &\quad-\int_0^t\int_{-\ell}^\ell\sigma \pa_tF^3\pa_1(\pa_tv\cdot\mathcal{N}^1)+\tilde{H}^3\cdot \pa_tv-g\pa_t\theta\tilde{H}^5-\sigma\frac{\pa_1\pa_t\theta\pa_1\tilde{H}^5}{(1+|\pa_1\zeta_0|^2)^{3/2}}-[\pa_tv\cdot\mathcal{N}^1,\tilde{H}^6]_{\ell}.
  \end{aligned}
  \eeq

  Step 2 --  Estimate of pressure.

  In order  to handle the term related to $\pa_tq$, we multiply \eqref{eq:difference_2} by $J^1w$, integrate over $\Om$ and integrate by parts to deduce that
  \beq\label{eq:contraction_weak}
  \begin{aligned}
    \frac{\mu}{2}\int_\Om\mathbb{D}_{\mathcal{A}^1}\pa_tv:\mathbb{D}_{\mathcal{A}^1}w J^1+\beta\int_{\Sigma_s}(\pa_tv\cdot\tau)(w\cdot\tau)+(\pa_t\theta,w\cdot\mathcal{N}^1)_{1,\Sigma}+[\pa_tv\cdot\mathcal{N}^1,w\cdot\mathcal{N}^1]_\ell\\
    =\int_\Om(\mu\dive_{\mathcal{A}^1}\left(\mathbb{D}_{(\pa_t\mathcal{A}^1-\pa_t\mathcal{A}^2)}u^2\right)+\tilde{H}^1)\cdot wJ^1-\int_{\Sigma_s}J^1(w\cdot\tau)(\mu\mathbb{D}_{(\pa_t\mathcal{A}^1-\pa_t\mathcal{A}^2)}u^2\nu\cdot\tau+\tilde{H}^4)\\
    -\int_{-\ell}^\ell\sigma\pa_t(F^{3,1}-F^{3,2})\pa_1 (w\cdot\mathcal{N}^1)+(\mu\mathbb{D}_{(\pa_t\mathcal{A}^1-\pa_t\mathcal{A}^2)}u^2\mathcal{N}^1+\tilde{H}^3)\cdot w-[w\cdot\mathcal{N}^1,\tilde{H}^6]_{\ell},
  \end{aligned}
  \eeq
  for each $w\in\mathcal{V}(t)$ and a.e. $t\in[0, T]$. Then $\pa_tq\in \mathring{H}^0(\Om)$ might be recovered from Theorem 4.6 in \cite{GT2} such that
  \beq
  \begin{aligned}
    &\frac{\mu}{2}\int_\Om\mathbb{D}_{\mathcal{A}^1}v:\mathbb{D}_{\mathcal{A}^1}w J^1+\beta\int_{\Sigma_s}(v\cdot\tau)(w\cdot\tau)-(\pa_tq,\dive_{\mathcal{A}^1}w)_0+(\pa_t\theta,w\cdot\mathcal{N}^1)_{1,\Sigma}\\
    &\quad+[\pa_tv\cdot\mathcal{N}^1,w\cdot\mathcal{N}^1]_\ell\\
    &=\int_\Om(\mu\dive_{\mathcal{A}^1}\left(\mathbb{D}_{(\pa_t\mathcal{A}^1-\pa_t\mathcal{A}^2)}u^2\right)+\tilde{H}^1)\cdot wJ^1-\int_{\Sigma_s}J^1(w\cdot\tau)(\mu\mathbb{D}_{(\pa_t\mathcal{A}^1-\pa_t\mathcal{A}^2)}u^2\nu\cdot\tau+\tilde{H}^4)\\
    &\quad-\int_{-\ell}^\ell\sigma\pa_t(F^{3,1}-F^{3,2})\pa_1 (w\cdot\mathcal{N}^1)+(\mu\mathbb{D}_{(\pa_t\mathcal{A}^1-\pa_t\mathcal{A}^2)}u^2\mathcal{N}^1+\tilde{H}^3)\cdot w-[w\cdot\mathcal{N}^1,\tilde{H}^6]_{\ell},
  \end{aligned}
  \eeq
  for each $w\in\mathcal{W}(t)$ and a.e. $t\in[0, T]$. Moreover,
  \beq
  \|\pa_tq\|_{L^2\mathring{H}^0}^2\lesssim \|\pa_tv\|_{L^2H^1}^2+P(\sqrt{\varepsilon})(\|\pa_t\theta\|_{L^2
   \mathring{H}^{3/2}}^2+\|\theta\|_{L^2 W_\delta^{5/2}}^2+\|q\|_{L^2\mathring{W}_\delta^1}^2+\|v\|_{L^2W_\delta^2}^2),
  \eeq
  where the temporal $L^2$ norm is computed on $[0, T]$, and $P(\cdot)$ is a polynomial which would be allowed to change from line to line.

  Step 3 --  Estimates of the forcing terms.

  To handle the term $\pa_t(F^{3,1}-F^{3,2})$, we rewrite it as
  \beq\label{force_3}
  \begin{aligned}
  \int_{-\ell}^\ell\sigma \pa_tF^3\pa_1(\pa_tv\cdot\mathcal{N}^1)=\int_{-\ell}^\ell\sigma[\pa_z\mathcal{R}^1\pa_1\pa_t\theta+\pa_z(\mathcal{R}^1-\mathcal{R}^2)\pa_1\pa_t\eta^2]\pa_1(\pa_t^2\theta-\tilde{H}^5)\\
  =\frac{d}{dt}\left(\int_{-\ell}^\ell\pa_z\mathcal{R}^1\frac{|\pa_1\pa_t\theta|^2}{2}-\pa_z(\mathcal{R}^1-\mathcal{R}^2)\pa_1\pa_t\eta^2\pa_1\pa_t\theta\right)-\int_{-\ell}^\ell\pa_z^2\mathcal{R}^1\pa_1\pa_t\eta^1\frac{|\pa_1\pa_t\theta|^2}{2}\\
  -\int_{-\ell}^\ell|\pa_1\pa_t\theta|^2\pa_z^2\mathcal{R}^1\pa_1\pa_t\eta^2+\pa_z^2(\mathcal{R}^1-\mathcal{R}^2)(\pa_1\pa_t\eta^2)^2\pa_1\pa_t\theta\\
  +\pa_z(\mathcal{R}^1-\mathcal{R}^2)\pa_1\pa_t^2\eta^2\pa_1\pa_t\theta-\pa_z(\mathcal{R}^1-\mathcal{R}^2)\tilde{H}^5,
  \end{aligned}
  \eeq
  Then we rewrite \eqref{eq:energy_1} as
  \beq
  \begin{aligned}
    \frac{d}{dt}\left(\|\pa_t\theta\|_{1,\Sigma}^2+\int_{-\ell}^\ell\pa_z\mathcal{R}^1\frac{|\pa_1\pa_t\theta|^2}{2}-\pa_z(\mathcal{R}^1-\mathcal{R}^2)\pa_1\pa_t\eta^2\pa_1\pa_t\theta\right)+\frac{\mu}{2}\int_\Om|\mathbb{D}_{\mathcal{A}^1}\pa_tv|^2J^1\\
    +\beta\int_{\Sigma_s}J^1|\pa_tv\cdot\tau|^2+[\pa_tv\cdot\mathcal{N}^1]_\ell^2\\
    =-\frac{\mu}{2}\int_\Om J^1\mathbb{D}_{(\pa_t\mathcal{A}^1-\pa_t\mathcal{A}^2)}u^2:\mathbb{D}_{\mathcal{A}^1}\pa_tv+\int_\Om\tilde{H}^1\cdot\pa_tvJ^1+\pa_tq\tilde{H}^2J^1-\int_{\Sigma_s}J^1(\pa_tv\cdot\tau)\tilde{H}^4\\
  -\int_{-\ell}^\ell|\pa_1\pa_t\theta|^2\pa_z^2\mathcal{R}^1\pa_1\pa_t\eta^2+\pa_z^2(\mathcal{R}^1-\mathcal{R}^2)(\pa_1\pa_t\eta^2)^2\pa_1\pa_t\theta+\pa_z(\mathcal{R}^1-\mathcal{R}^2)(\pa_1\pa_t^2\eta^2\pa_1\pa_t\theta-\tilde{H}^5)\\
    -\int_{-\ell}^\ell\pa_z^2\mathcal{R}^1\pa_1\pa_t\eta^1\frac{|\pa_1\pa_t\theta|^2}{2}-\int_{-\ell}^\ell\left(\tilde{H}^3\cdot \pa_tv-g\pa_t\theta\tilde{H}^5-\sigma\frac{\pa_1\pa_t\theta\pa_1\tilde{H}^5}{(1+|\pa_1\zeta_0|^2)^{3/2}}\right)\\
    -[\pa_tv\cdot\mathcal{N}^1,\tilde{H}^6]_{\ell}.
  \end{aligned}
  \eeq
   We now estimate the terms on the right hand side of \eqref{eq:energy_3}.
  \beq\label{force_1}
  \frac{\mu}{2}\int_\Om J^1\mathbb{D}_{(\pa_t\mathcal{A}^1-\pa_t\mathcal{A}^2)}u^2:\mathbb{D}_{\mathcal{A}^1}\pa_tv\lesssim P(\sqrt{\varepsilon})\|\pa_tv\|_1\|\pa_t\theta\|_{3/2}.
  \eeq
  \beq
  \begin{aligned}
    \int_\Om\tilde{H}^1\cdot\pa_tvJ^1\lesssim P(\sqrt{\varepsilon})\|\pa_tv\|_1(\|\pa_t\theta\|_
   {3/2}+\|\theta\|_{W_\delta^{5/2}}+\|q\|_{\mathring{W}_\delta^1}+\|v\|_{W_\delta^2}).
  \end{aligned}
  \eeq
  \beq
  \int_{\Sigma_s}J^1(\pa_tv\cdot\tau)\tilde{H}^4\lesssim P(\sqrt{\varepsilon})\|\pa_tv\|_1(\|\theta\|_{W_\delta^{5/2}}+\|\pa_t\theta\|_{3/2}).
  \eeq
  \beq
  \int_\Om\pa_tq\tilde{H}^2J^1\lesssim P(\sqrt{\varepsilon})\|\pa_tq\|_0(\|\pa_t\theta\|_{3/2}+\|\theta\|_{ W_\delta^{5/2}}+\|v\|_{W_\delta^2}).
  \eeq
  By the direct computation for derivatives of \eqref{def:R}, we may employ the Sobolev embedding theory to derive that
  \beq
  \begin{aligned}
    -\int_{-\ell}^\ell|\pa_1\pa_t\theta|^2\pa_z^2\mathcal{R}^1\pa_1\pa_t\eta^2+\pa_z^2(\mathcal{R}^1-\mathcal{R}^2)(\pa_1\pa_t\eta^2)^2\pa_1\pa_t\theta+\pa_z(\mathcal{R}^1-\mathcal{R}^2)(\pa_1\pa_t^2\eta^2\pa_1\pa_t\theta-\tilde{H}^5)\\
    -\int_{-\ell}^\ell\pa_z^2\mathcal{R}^1\pa_1\pa_t\eta^1\frac{|\pa_1\pa_t\theta|^2}{2}\lesssim P(\sqrt{\varepsilon})(\|\pa_t\theta\|_{3/2}^2+\|\theta\|_{W_\delta^{5/2}}^2+\|\theta\|_{ W_\delta^{5/2}}+\|v\|_{W_\delta^2}),
  \end{aligned}
  \eeq
  and
  \beq
  \int_{-\ell}^\ell\pa_z\mathcal{R}^1\frac{|\pa_1\pa_t\theta|^2}{2}-\pa_z(\mathcal{R}^1-\mathcal{R}^2)\pa_1\pa_t\eta^2\pa_1\pa_t\theta\lesssim P(\sqrt{\varepsilon})(\|\pa_t\theta\|_{3/2}^2+\|\theta\|_{ W_\delta^{5/2}}).
  \eeq
  \beq
  \int_{-\ell}^\ell\sigma\tilde{H}^3\cdot \pa_tv\lesssim P(\sqrt{\varepsilon})\|\pa_tv\|_1(\|\theta\|_{W_\delta^{5/2}}+\|\pa_t\theta\|_{3/2}+\|q\|_{\mathring{W}_\delta^1}+\|v\|_{W_\delta^2}).
  \eeq
  Due to the fact that $v_1^1=v_1^2=0$ at the endpoints $x_1=\pm\ell$, after integrating by parts,
  \beq
  \begin{aligned}
  \int_{-\ell}^\ell-g\pa_t\theta\tilde{H}^5-\sigma\frac{\pa_1\pa_t\theta\pa_1\tilde{H}^5}{(1+|\pa_1\zeta_0|^2)^{3/2}}\lesssim P(\sqrt{\varepsilon})\left[\|\pa_t\theta\|_1(\|\theta\|_{W_\delta^{5/2}}
  +\|\pa_t\theta\|_1+\|v\|_{W_\delta^2})+\|\pa_t\theta\|_{3/2}^2\right]\\
  +\|\pa_t\eta^1\|_{W_\delta^{5/2}}\|\pa_t\theta\|_{3/2}\|v\|_1.
  \end{aligned}
  \eeq
  \beq\label{force_2}
  [\pa_tv\cdot\mathcal{N}^1,\tilde{H}^6]_{\ell}=[\pa_t^2\theta,\tilde{H}^6]_\ell\lesssim P(\sqrt{\varepsilon})\|\pa_t\theta\|_1^2.
  \eeq

  Then combining all the above estimates \eqref{force_1}--\eqref{force_2}, we can derive that
  \beq
  \begin{aligned}
  &\frac{d}{dt}\left(\|\pa_t\theta\|_{1,\Sigma}^2+\int_{-\ell}^\ell\pa_z\mathcal{R}^1\frac{|\pa_1\pa_t\theta|^2}{2}-\pa_z(\mathcal{R}^1-\mathcal{R}^2)\pa_1\pa_t\eta^2\pa_1\pa_t\theta\right)+\|\pa_tv\|_1^2+[\pa_tv\cdot\mathcal{N}^1]^2\\
  &\le CP(\sqrt{\varepsilon})\left(\|\pa_t\theta\|_{1,\Sigma}^2+\int_{-\ell}^\ell\pa_z\mathcal{R}^1\frac{|\pa_1\pa_t\theta|^2}{2}-\pa_z(\mathcal{R}^1-\mathcal{R}^2)\pa_1\pa_t\eta^2\pa_1\pa_t\theta\right)\\
  &\quad+CP(\sqrt{\varepsilon})\|\theta\|_{W_\delta^{5/2}}^2+\|\pa_t\theta\|_{3/2}^2+\|q\|_{\mathring{W}_\delta^1}^2+\|v\|_{W_\delta^2}^2.
  \end{aligned}
  \eeq
  Since
  \begin{align*}
  \sup_{0\le t\le T}\int_{-\ell}^\ell\pa_z\mathcal{R}^1\frac{|\pa_1\pa_t\theta|^2}{2}-\pa_z(\mathcal{R}^1-\mathcal{R}^2)\pa_1\pa_t\eta^2\pa_1\pa_t\theta\lesssim P(\sqrt{\varepsilon})(\|\pa_t\theta\|_{L^\infty\mathring{H}^1}^2+\|\theta\|_{L^\infty\mathring{H}^{3/2}}^2)\\
  \lesssim P(\sqrt{\varepsilon})(\|\pa_t\theta\|_{L^\infty\mathring{H}^1}^2+\|\pa_t\theta\|_{L^2\mathring{H}^{3/2}}^2+\|\theta\|_{L^2W_\delta^{5/2}}^2),
  \end{align*}
  Gronwall's lemma together with the smallness of $\varepsilon$ implies that
  \beq
  \begin{aligned}
    \|\pa_t\theta\|_{L^\infty\mathring{H}^1}^2+\|\pa_tv\|_{L^2H^1}^2+\int_0^T[\pa_tv\cdot\mathcal{N}^1]^2 \lesssim e^{CP(\sqrt{\varepsilon})T_1} CP(\sqrt{\varepsilon})(\|\theta\|_{L^2W_\delta^{5/2}}^2+\|\pa_t\theta\|_{L^2\mathring{H}^{3/2}}^2\\
    +\|\pa_tq\|_{L^2\mathring{H}^0}^2+\|q\|_{L^2\mathring{W}_\delta^1}^2+\|v\|_{L^2W_\delta^2}^2)\\
    \lesssim e^{CP(\sqrt{\varepsilon})T_1}CP(\sqrt{\varepsilon})(\|\theta\|_{L^2W_\delta^{5/2}}^2+\|\pa_t\theta\|_{L^2\mathring{H}^{3/2}}^2
    +\|q\|_{L^2\mathring{W}_\delta^1}^2+\|v\|_{L^2W_\delta^2}^2),
  \end{aligned}
  \eeq
  where the temporal $L^\infty$ and $L^2$ norms are computed over $[0,T]$ and $0<t<T\le T_1$.
  We assume that $\varepsilon_1$ and $T_1$ are sufficiently small for $e^{CP(\sqrt{\varepsilon})T_1}\le e^{CP(\sqrt{\varepsilon_1})T_1}\le 2$. Then we deduce the bound
  \beq
  \|\pa_t\theta\|_{L^\infty\mathring{H}^1}^2+\|\pa_tv\|_{L^2H^1}^2+\int_0^T[\pa_tv\cdot\mathcal{N}^1]^2 \lesssim P(\sqrt{\varepsilon})(\|\theta\|_{L^2W_\delta^{5/2}}^2+\|\pa_t\theta\|_{L^2\mathring{H}^{3/2}}^2
    +\|q\|_{L^2\mathring{W}_\delta^1}^2+\|v\|_{L^2W_\delta^2}^2).
  \eeq
  Since $\pa_t\theta\in \mathring{H}^1((-\ell,\ell))$ and \eqref{eq:contraction_weak}, with $\varepsilon$ sufficient small, Theorem 4.11 in \cite{GT2} reveals that
  \beq
  \|\pa_t\theta\|_{L^2\mathring{H}^{3/2}}^2\lesssim P(\sqrt{\varepsilon})(\|\theta\|_{L^2W_\delta^{5/2}}^2
    +\|q\|_{L^2\mathring{W}_\delta^1}^2+\|v\|_{L^2W_\delta^2}^2).
  \eeq

  Step 4 -- Elliptic estimates for $v$, $q$ and $\theta$.

  In order to close our estimates, we must be able to estimate $v$, $q$ and $\theta$. The elliptic estimates imply that
  \beq
  \begin{aligned}
  \|v\|_{W_\delta^2}^2+\|q\|_{\mathring{W}_\delta^1}^2+\|\theta\|_{W_\delta^{5/2}}^2\lesssim \|\dive_{\mathcal{A}^1}(\mathbb{D}_{\mathcal{A}^1-\mathcal{A}^2}u^2)+H^1\|_{W_\delta^0}^2+\|H^2\|_{W_\delta^1}^2\\
  +\|\pa_t\theta-H^5\|_{W_\delta^{3/2}}^2+\|H^3\|_{W_\delta^{1/2}}^2+\|\mathbb{D}_{\mathcal{A}^1-\mathcal{A}^2}u^2\nu\cdot\tau\|_{W_\delta^{1/2}}^2\\
  +\|\pa_1(F^{3,1}-F^{3,2})\|_{W_\delta^{1/2}}^2+[\pa_t\theta\pm H^6]_\ell^2.
  \end{aligned}
  \eeq
  Then after integrating temporally from $0$ to $T$, we have that
  \beq
  \begin{aligned}
  \|v\|_{L^2W_\delta^2}^2+\|q\|_{L^2\mathring{W}_\delta^1}^2+\|\theta\|_{L^2W_\delta^{5/2}}^2\lesssim P(\sqrt{\varepsilon})\|\theta\|_{L^2W_\delta^{5/2}}^2+\|\pa_t\theta\|_{L^2W_\delta^{3/2}}^2\\
  \le CP(\sqrt{\varepsilon})(\|\theta\|_{L^2W_\delta^{5/2}}^2
    +\|q\|_{L^2\mathring{W}_\delta^1}^2+\|v\|_{L^2W_\delta^2}^2),
    \end{aligned}
  \eeq
  where $P(0)=0$. Since $\varepsilon$ is sufficiently small, we might restrict $\varepsilon_1$ such that $CP(\sqrt{\varepsilon})<1$. Thus
  \beq
  \|v\|_{L^2W_\delta^2}^2+\|q\|_{L^2\mathring{W}_\delta^1}^2+\|\theta\|_{L^2W_\delta^{5/2}}^2=0.
  \eeq
\end{proof}

\subsection{Diffeomorphism of $\Phi$}
From the definition of $J$ and restrict theory in Sobolev spaces, we can derive that
\[
\|J\|_{L^\infty}\ge 1-C(\|\bar{\eta}\|_{L^\infty}+\|\pa_2\bar{\eta}\|_{L^\infty})\ge 1-C\|\eta\|_{W_\delta^{5/2}}.
\]
The smallness of $\mathfrak{K}(\eta)$ sufficiently guarantees that $\Phi$, defined in \eqref{def:map}, is a $C^1$ diffeomorphism for each $t\in[0, T]$. For more details, one can see \cite{GT0} in $3D$ domains.

\begin{appendix}

\section{Properties involving $\mathcal{A}$}
We now record some useful properties involving $\mathcal{A}$.
\begin{lemma}\label{lem:properties_A}
The following identities hold.
  \begin{enumerate}[(1)]
    \item $\pa_j(J\mathcal{A}_{ij})=0$ for $j=1,2$ and each $i=1,2$.
    \item $J\mathcal{A}\mathcal{N}_0=\mathcal{N}$ on $\Sigma$,
    \item $R^\top\mathcal{N}=-\pa_t\mathcal{N}$ on $\Sigma$, where $R$ is defined by \eqref{def:Dt_u}.
  \end{enumerate}
\end{lemma}
\begin{proof}
The first equality comes from Lemma A.3 in \cite{GT1}.
 On $\Sigma$,
 \begin{align*}
 J\mathcal{A}\mathcal{N}_0&=\left(
 \begin{array}{cc}
   J&-A\\
   0&1
 \end{array}\right)\left(
 \begin{array}{c}
   -\pa_1\zeta_0\\
   1
 \end{array}\right)\\
 &=\left(
 \begin{array}{c}
   -(1+\pa_2\bar{\eta}+\frac{\eta}{\zeta_0})\pa_1\zeta_0-\pa_1\eta+\pa_2\bar{\eta}\pa_1\zeta_0+\frac{1}{\zeta_0}\pa_1\zeta_0\eta\\
   1
 \end{array}\right)\\
 &=\left(
 \begin{array}{c}
   -\pa_1\zeta_0-\pa_1\eta\\
   1
 \end{array}\right)=\mathcal{N}.
 \end{align*}
 It is easily to compute that $R^\top=J\pa_tKI_{2\times2}-\pa_t\mathcal{A}\mathcal{A}^{-1}$. Since $J\mathcal{A}\mathcal{N}_0=\mathcal{N}$,
  \begin{align*}
    R^\top\mathcal{N}&=(J\pa_tK-\pa_t\mathcal{A}\mathcal{A}^{-1})J\mathcal{A}\mathcal{N}_0\\
    &=(-K\pa_tJ-\pa_t\mathcal{A}\mathcal{A}^{-1})J\mathcal{A}\mathcal{N}_0\\
    &=(-\pa_tJ\mathcal{A}-J\pa_t\mathcal{A})\mathcal{N}_0=-\pa_t(J\mathcal{A}\mathcal{N}_0)=-\pa_t\mathcal{N}.
  \end{align*}
\end{proof}

\end{appendix}

\end{document}